\documentclass[a4paper]{article}

\usepackage{hyperref}
\usepackage{float}
\usepackage{graphicx}
\usepackage{mathrsfs}
\usepackage{enumerate} 
\usepackage{amsxtra,amssymb,latexsym, amscd,amsthm}
\usepackage{indentfirst}
\usepackage{color}
\usepackage[utf8]{inputenc}
\usepackage[mathscr]{eucal}
\usepackage{amsfonts}
\usepackage{graphics}
\usepackage{multirow}
\usepackage{array}
\usepackage{subfigure}
\usepackage{wrapfig}
\usepackage{tikz}
\usepackage{diagbox}
\usepackage{graphicx}
\usepackage{natbib}
\usepackage{bibentry}
\usepackage{algorithm}
\usepackage{longtable}

\usepackage{pgfplots}
\pgfplotsset{compat=1.15}
\usepackage{mathrsfs}
\usepackage[left=2cm, right=2cm, top=2cm, bottom=2cm]{geometry}
\usetikzlibrary{arrows}

\def\M{\mathbf{M}}

\def\bA{\mathbf{A}}

\def\a{\mathbf{a}}
\def\e{\mathbf{e}}
\def\v{\mathbf{v}}
\def\x{\mathbf{x}}
\def\y{\mathbf{y}}
\def\z{\mathbf{z}}

\def\R{{\mathbb R}}

\def\N{{\mathbb N}}

\DeclareMathOperator{\diag}{diag}

\DeclareMathOperator{\supp}{supp}

\newcounter{algsubstate}
\renewcommand{\thealgsubstate}{\alph{algsubstate}}
\newenvironment{algsubstates}
  {\setcounter{algsubstate}{0}%
   \renewcommand{\State}{%
     \stepcounter{algsubstate}%
     \Statex {\footnotesize\thealgsubstate:}\space}}
  {}

\newtheorem{theorem}{\bf Theorem}
\newtheorem{lemma}{\bf Lemma}
\newtheorem{example}{\bf Example}
\newtheorem{proposition}{\bf Proposition}
\newtheorem{corollary}{\bf Corollary}
\newtheorem{remark}{\bf Remark}
\newtheorem{assumption}{\bf Assumption}

\newcommand{\revise}[1]{\textcolor{black}{#1}}

\definecolor{qqzzff}{rgb}{0,0.6,1}
\definecolor{ududff}{rgb}{0.30196078431372547,0.30196078431372547,1}
\definecolor{xdxdff}{rgb}{0.49019607843137253,0.49019607843137253,1}
\definecolor{ffzzqq}{rgb}{1,0.6,0}
\definecolor{qqzzqq}{rgb}{0,0.6,0}
\definecolor{ffqqqq}{rgb}{1,0,0}
\definecolor{uuuuuu}{rgb}{0.26666666666666666,0.26666666666666666,0.26666666666666666}

\usepackage{hyperref}
\usepackage{graphicx}
\usepackage{mathrsfs}
\usepackage{amsmath}
\usepackage{enumerate} 
\usepackage{amsxtra,amssymb,latexsym, amscd,amsthm}
\usepackage{indentfirst}
\usepackage{color}
\usepackage{algorithm}
\usepackage[noend]{algpseudocode}

\makeatletter
\def\algbackskip{\hskip-\ALG@thistlm}
\makeatother

\usepackage{ushort}
\usepackage{pdflscape}
\usepackage{rotating}
\usepackage{pgfplots}

\usepackage{subfigure}
\usepackage{tikz}
\usepackage{adjustbox}
\usepackage{longtable}

\usepackage{amsmath}
\usepackage{amssymb}
\usepackage{enumerate}
\usepackage{mathrsfs}
\usepackage{cases}
\usepackage{color}
\usepackage{tikz}
\usepackage{appendix}
\usepackage{makecell}
\usepackage{multirow}
\usepackage{algpseudocode}
\usepackage{algorithmicx,algorithm}
\usepackage{graphicx}
\definecolor{darkblue}{rgb}{0.0, 0.0, 0.55}
\definecolor{bordeaux}{rgb}{0.34, 0.01, 0.1}

\newcommand{\footremember}[2]{%
    \footnote{#2}
    \newcounter{#1}
    \setcounter{#1}{\value{footnote}}%
}
\newcommand{\footrecall}[1]{%
    \footnotemark[\value{#1}]%
}

\numberwithin{equation}{section}

\def\R{{\mathbb{R}}}

\def\N{{\mathbb{N}}}

\def\x{{\mathbf{x}}}
\def\c{{\mathbf{c}}}
\def\z{{\mathbf{z}}}

\def\e{{\mathbf{e}}}
\def\y{{\mathbf{y}}}

\def\a{{\boldsymbol{\alpha}}}

\def\b{{\boldsymbol{\beta}}}
\def\g{{\boldsymbol{\gamma}}}

\def\supp{\hbox{\rm{supp}}}

\def\cA{{\cal A}}
\DeclareMathOperator{\st}{s.t.}

\begin{document}

\providecommand{\keywords}[1]
{
  \textbf{\textbf{Keywords:}} #1
}

\title{Tractable hierarchies of convex relaxations for polynomial optimization on the nonnegative orthant}

\author{%
Ngoc Hoang Anh Mai\footremember{1}{Institute of Mathematics, Vietnam Academy of Science and Technology, 18 Hoang Quoc Viet, Cau Giay, Ha Noi, Viet Nam.}, %
  Victor Magron\footremember{2}{CNRS; LAAS; 7 avenue du Colonel Roche, F-31400 Toulouse; France.}\footremember{3}{Universit\'e de Toulouse; LAAS; F-31400 Toulouse, France.}, \\%
  Jean-Bernard Lasserre\footrecall{2} \footrecall{3} , %
  Kim-Chuan Toh\footremember{4}{Department of Mathematics, National University of Singapore; 10 Lower Kent Ridge Road, Singapore 119076.}
  }
\maketitle

\begin{abstract}
We consider {polynomial optimization problems (POP)} on a semialgebraic set contained in the nonnegative orthant {(every POP on a compact set can be put in this format by a simple translation of the origin).
Such a POP} can be converted to an equivalent POP by squaring each variable.
Using even symmetry and the concept of \emph{factor width}, we propose a  hierarchy of semidefinite  relaxations based on the extension of P\'olya's Positivstellensatz by Dickinson--Povh.  
As its distinguishing and crucial feature, the maximal matrix size of each resulting semidefinite relaxation can be chosen {\emph{arbitrarily} and in addition,
we prove that} the sequence of values returned by the new hierarchy converges to the optimal value of the original POP at the rate   $\mathcal{O}(\varepsilon^{-\mathfrak c})$ if the semialgebraic set has nonempty interior.
When applied to (i) robustness certification of multi-layer neural networks and (ii) computation of positive maximal singular values, our method based on P\'olya's Positivstellensatz provides better bounds and runs several hundred times faster than the standard Moment-SOS hierarchy.
\end{abstract}

\tableofcontents

\section{Introduction}

Polynomial optimization is concerned with computing the minimum value of a polynomial on a basic semialgebraic set.
A well-known methodology is to apply positivity certificates (representations of polynomials positive on basic semialgebraic sets) to design a hierarchy of convex relaxations to solve polynomial optimization problems (POPs).
Developed originally by Lasserre in \cite{lasserre2001global}, the hierarchy of semidefinite relaxations based on Putinar's Positivstellensatz is called the \emph{Moment-SOS hierarchy} \revise{and has been used in many applications arising} from optimization, operations research, signal processing, computational geometry, probability, statistics, control, PDEs, quantum information, and computer vision.
For more details, the interested reader is referred to, e.g., \cite{acopf,weaklyhard,sparsedynsys,schlosser2021converging,schlosser2020sparse,chen2021semialgebraic,chen2020semialgebraic,tacchi2021exploiting,oustry2019inner,tacchi2020approximating} and references therein.

\if{
The now popular methodology initially developed in \cite{lasserre2001global}
allows us to obtain a converging sequence of lower bounds by solving a hierarchy of convex relaxations of the  original polynomial optimization problem. 
Each relaxation of the Moment-SOS hierarchy is a semidefinite program (SDP),  solvable in time polynomial in its input size.
However, despite its theoretical efficiency (also observed in practice), in its canonical form this approach is facing a serious scalability issue mainly due to the increasing size of the resulting relaxations. 
Therefore its initial standard formulation is limited to problems of modest dimension unless some sparsity and/or symmetry can be exploited. 
}\fi
However, despite its theoretical efficiency (also observed in practice),  \revise{in its canonical form} the Moment-SOS hierarchy is facing a scalability issue mainly due to the increasing size of the resulting relaxations. 
Overcoming the scalability and efficiency issues has become a major scientific challenge in polynomial  optimization.
Many recent efforts
in this direction are mainly developed around the following ideas:
\begin{enumerate}
\item \emph{SDP-relaxations variants} with \emph{small} maximal matrix size solved efficiently by interior point methods.
This includes correlative sparsity \cite{waki2006sums,lasserre2006convergent}, term sparsity \cite{wang2021tssos,wang2021chordal,wang2022cs}, symmetry exploitation \cite{gatermann2004symmetry,riener2013exploiting}, Jordan  symmetry reduction \cite{brosch2022jordan}, sublevel relaxations \cite{chen2022sublevel}.
\item Exploit \emph{low-rank structures} of SDP-relaxations; see, e.g., \cite{weisser2018sparse,yang2023inexact}.
\item \emph{First-order methods} to solve  SDP-relaxations involving  matrix variables of potentially large size with constant trace  \cite{mai2023hierarchy,mai2022exploiting}.
\item 
Develop \emph{convex relaxations} that are based on alternatives to semidefinite cones. For example this includes linear programming  (LP)  \cite{lasserre2017bounded,ahmadi2019dsos}, second-order conic programming (SOCP) \cite{magron2023sonc,wang2020second,ahmadi2019dsos}, copositive programming  \cite{pena2015completely},
 non-symmetric  conic programming \cite{papp2019sum}, relative entropy programming \cite{dressler2017positivstellensatz,murray2021signomial}, geometric programming \cite{dressler2019approach}.
\end{enumerate}

Let $\R[\x]$ denote the set of real polynomials in vector of variables $\x=(x_1,\dots,x_n)$.
Given a real symmetric matrix $\mathbf M$, the notation $\mathbf M\succeq 0$ denotes that $\mathbf M$ is positive semidefinite, i.e., all its eigenvalue are nonnegative.
Given $r\in\N_{> 0}$, denote $[r]:=\{1,\dots,r\}$.

Sparsity exploitation is one of the notable methods to reduce the size of the Moment-SOS relaxations.
For POPs in the form
\begin{equation}\label{eq:POP}
    f^\star:=\min_{\x\in S(\mathfrak g)}f(\x)\,,
\end{equation}
where $f$ is a polynomial in $\R[\x]$ and $S(\mathfrak g)$ is the semialgebraic set associated with $\mathfrak g=\{g_1,\dots,g_m\}\subset \R[\x]$, i.e.,
\begin{equation}\label{eq:S(g)}
    S(\mathfrak g) :=\{\, \x\in\R^n:\: g_i(\x)\ge  0\,,\,i\in [m]\,\} \,,
\end{equation}
Waki et al. \cite{waki2006sums}
(resp. Wang et al. \cite{wang2021tssos}) have exploited correlative (resp. term) sparsity
to define appropriate sparse-variants of the associated standard SOS-relaxations.
Roughly speaking, in a given standard SOS-relaxation, they break each matrix variable 
into many blocks of smaller sizes and solve the new resulting SDP 
via an interior-point solver (e.g., {\tt Mosek}  \cite{aps2019mosek} or {\tt SDPT3}  \cite{tutuncu2003solving}). 
It is due to the fact that the most expensive part of interior-point methods for a standard SDP: 
\begin{equation*}\label{eq:standard.SDP}
\begin{array}{rl}
     \min\limits_{\z,\bA^{(t)}_j}&\c^\top \z\\
     \text{s.t.}&\z\in \R^w\,,\,\bA^{(t)}_j\in\R^{q\times q}\,,\\
     &\bA^{(t)}_0+\sum_{j=1}^wz_j\bA^{(t)}_j\succeq 0\,,\,t\in[u]\,,
\end{array}
\end{equation*}
is solving a square linear system in every iteration.
It has the complexity $O(u(wq^3+w^2q^2)) + O(w^3)$, 
which mainly depends on the matrix size $q$. 
Thus one can solve the above SDP  efficiently by using interior-point methods if $q,w$ are small, even when $u$ is large.
On one hand, correlative sparsity occurs for POP \eqref{eq:POP} being such that the objective polynomial has a decomposition $f=f_1+\dots+f_p$, where each polynomial 
$f_t$ involves only a small subset of variables $I_t\subset[n]$, and $f_t$ together with the constraint polynomials $(g_i)_{i\in J_t}$ (for some $J_t\subset[m]$) share the same variables.
On the other hand, term sparsity occurs for POP (1) where $f,g_1,\dots,g_m$ have a few nonzero terms.
To solve large-scale POPs, one can simultaneously exploit correlative sparsity and term sparsity as in \cite{wang2022cs}. 

Denote by $\|\cdot\|_2$ the l$_2$-norm of a vector in $\R^n$.
A polynomial $p\in\R[\x]$ is written as
$p=\sum_{\a\in\N^n} p_\a \x^\a$ with monomial $\x^\a:=x_1^{\alpha_1}\dots x_n^{\alpha_n}$ for some finite real sequence $(p_\a)_{\a\in\N^n}$.
Given $\a = (\alpha_1,\dots,\alpha_n) \in \N^n$, we denote $|\a| := \alpha_1 + \dots + \alpha_n$.
Define $\N^n_t:=\{\a\in\N^n\,:\,|\a|\le t\}$ for each $t\in\N$.
Given $\mathbf{u}\in\R^{r}$, $\diag(\mathbf{u})$ stands for the diagonal matrix of size $r$ with diagonal entries given by $\mathbf{u}$. Denote $\R_+:=[0,\infty)$.
Let $(\x^\a)_{\a\in\N^n}$ 
be the canonical basis of monomials for  $\R[\x]$ (ordered according to the graded lexicographic order) and 
$\v_t(\x)$ be the vector of all monomials up to degree $t$, with length $b(n,t)={\binom {n+t} n}$.

For each $\mathcal{A}\subset \N^n$, denote $\v_\cA(\x)=(\x^\a)_{\a\in\cA}$.
We say that a polynomial $q$ is \emph{even in each variable} if for every $j\in[n]$, 
$q(x_1,\dots,x_{j-1},-x_j,x_{j+1},\dots,x_{n})= q(x_1,\dots,x_{j-1},x_j,x_{j+1},\dots,x_{n})$.
A polynomial $q$ is called an \emph{SOS of monomials} if $q=\sum_{\a\in\N^n}\lambda_\a \x^{2\a}$ for some $\lambda_\a\ge 0$.
Accordingly, if $q$ is an SOS of monomials, then $q=\v_d^\top \diag(\mathbf{u}) \v_d$ for some $d\in\N$ and $\mathbf{u}\in\R_+^{b(n,d)}$.
For a given real-valued sequence $y=(y_\alpha)_{\alpha\in\N^n}$, let us define the Riesz linear functional $L_y:\R[ \x ] \to \R$ by $p\mapsto {L_y}( p ) := \sum_{\alpha} p_\alpha y_\alpha$, for all $ p=\sum_{\alpha}p_\alpha \x^\alpha\in\R[\x]$.

\paragraph{Factor width:}  Originally defined in \cite{boman2005factor}, the factor width of a real positive semidefinite matrix $\mathbf G$ is the smallest integer $s$ for which there exists a real matrix $\mathbf P$ such that $\mathbf G$ can be decomposed as $\mathbf G =\mathbf P\mathbf P^\top$ and each column of $\mathbf P$ contains at most $s$ nonzeros.
In this case, if $\mathbf u$ is a vector of several monomials in $\x$, the SOS polynomial $\mathbf u^\top \mathbf G\mathbf u$ can be written as $\mathbf u(\x)^\top \mathbf G\mathbf u(\x)=\sum_{i}(\mathbf q_i^\top \mathbf u(\x))^2$, where $\mathbf q_i$ is the $i$-th column of $\mathbf P$.
It is not hard to prove that the Gram matrix of each square $(\mathbf q_i^\top \mathbf u(\x))^2$ has size at most $s$ since $\mathbf q_i$ has at most $s$ nonzeros.
Thus, if an SOS polynomial has  Gram matrix of factor width at most $s$, it can be written as a sum of SOS polynomials with Gram matrix sizes at most $s$. 
The converse also holds true thanks to eigen-decomposition.
The applications of factor width for polynomial optimization can be found in, e.g., \cite{ahmadi2014dsos,majumdar2014control,ahmadi2019dsos,miller2022decomposed}.

\paragraph{POP with nonnegative variables:} In the present paper, we focus on the following POP on the nonnegative orthant:
\begin{equation}\label{eq:constrained.problem.poly.even}
\begin{array}{l}
f^\star:=\inf\limits_{\x\in S} f(\x)\,,
\end{array}
\end{equation}
where $f$ is a polynomial and $S$ is a semialgebraic set defined by
\begin{equation}\label{eq:semial.set.def.2.even}
    S:=\{\x\in\R^n\,:\,x_j\ge 0\,,\,j\in[n]\,,\,g_i(\x)\ge 0\,,\,i\in[m]\}\,,
\end{equation}
for some $g_i\in\R[\x]$, $i\in[m]$ with $g_m:=1$.
Letting $\check q(\x):=q(\x^2)$ (with $\x^2:=(x_1^2,\dots,x_n^2)$) whenever $q\in \R[\x]$,
it follows immediately that
problem \eqref{eq:constrained.problem.poly.even} is equivalent to solving
\begin{equation}\label{eq:new.pop}
    f^\star=\inf_{\x\in \check S}\check f(\x)\,,
\end{equation}
where $\check S$ is a subset of $\R^n$ defined by
\begin{equation}\label{eq:new.semi.set}
    \check S:=\{\x\in\R^n\,:\,\check g_i(\x)\ge 0\,,\,i\in[m]\}\,.
\end{equation}

\paragraph{Contribution.}
\revise{In this paper}, we \revise{first} provide in Corollary \ref{coro:compact.even} a degree bound for the extension of P\'olya's Positivstellensatz originally stated in \cite{dickinson2015extension}.
Explicitly, if 

- $\check f, \check g_1,\dots,\check g_m$ are polynomials even in each variable, 

- $\check S$ defined as in \eqref{eq:new.semi.set} has nonempty interior, $\check g_1=R-\|\x\|_2^2$ for some $R>0$, 

- $\check f$ is of degree at most $2d_f$, each $\check g_i$ is of degree at most $2d_{g_i}$, and $\check f -f^\star$ is nonnegative on $\check S$, 

\noindent
then there exist positive constants $\bar{\mathfrak{c}}$ and $\mathfrak{c}$ depending on $\check f,\check g_i$ such that for all $\varepsilon>0$, for all $k\ge \bar{\mathfrak{c}}\varepsilon^{-\mathfrak{c}}$, 
    \begin{equation}\label{eq:Pu-va.ref.Pol}
    \begin{array}{l}
(1+ \|\x\|_2^2)^{k}(\check f-f^\star+\varepsilon)=\sum_{i\in[m]}\sigma_i\check g_i\,,
    \end{array}
    \end{equation}
    for some $\sigma_i$ being SOS of monomials such that $\deg(\sigma_i\check g_i)\le 2(k+d_f)$.
    (Here $\check g_m:=1$.)
    
    Consequently, the resulting LP-hierarchy of lower bounds $(\rho^{\textup{\revise{P\'olya}}}_k)_{k\in\N}$ for POP \eqref{eq:new.pop}:

\begin{equation}\label{eq:dual-sdp.0.even.LP.Han.intro}
\begin{array}{rl}
   {\rho _k^{\textup{\revise{P\'olya}}}}:= \sup\limits_{\lambda,\mathbf u_i} & \lambda\\
   \st& \lambda\in\R\,,\,\mathbf u_i\in\R_+^{b(n,k_i)}\,,\,i\in[m]\,,\\
   &\theta^k(\check f-\lambda)=\sum_{i\in[m]} \check g_i\v_{k_i}^\top \diag(\mathbf{u}_i) \v_{k_i}\,,
\end{array}
\end{equation}
where $\theta := 1+ \|\mathbf{x}\|_2^2$ and 
$k_{i}:=k+d_f - d_{g_i}$, for $i\in[m]$,
converges to $f^\star$ with a rate at least $\mathcal{O}(\varepsilon^{-\mathfrak{c}})$.
This linear hierarchy {was originally described in Dickinson and Povh \cite{dickinson2019new} and the novelty w.r.t. \cite{dickinson2019new} is that we now provide a convergence rate.}
Unfortunately, for large relaxation order $k$, this LP is potentially ill-conditioned (see for instance Example \ref{exam:AMGM}).

\revise{In our next contribution we address this issue. We} replace each diagonal Gram matrix $\diag(u_{j})$ in LP \eqref{eq:dual-sdp.0.even.LP.Han.intro} by a Gram matrix of factor width at most $s\in\N_{> 0}$ to obtain a semidefinite relaxation, which is tighter than LP \eqref{eq:dual-sdp.0.even.LP.Han.intro}.
Namely, consider the following SDP indexed by $k\in\N$ and $s\in\N_{> 0}$:
\begin{equation}\label{eq:primal.problem.0.even.Han.intro}
\begin{array}{rl}
   {\rho _{k,s}^\textup{\revise{P\'olya}}}:= \sup\limits_{\lambda,\mathbf{G}_{ij}} & \lambda\\
   \st& \lambda\in\R\,,\,\mathbf{G}_{ij}\succeq 0\,,\,j\in[b(n,k_i)]\,,\,i\in[m]\,,\\[5pt]
   &\theta^k(\check f-\lambda)=\sum_{i\in[m]} \check g_i\big(\sum_{j\in[b(n,k_i)]}\v_{\cA^{(s,k_i)}_j}^\top \mathbf{G}_{ij} \v_{\cA^{(s,k_i)}_j}\big)\,.
\end{array}
\end{equation}
where each $\mathcal A^{(s,d)}_r\subset \N^n_{d}$, chosen as in Section \ref{sec:dense.POP}, is such that $(\mathcal A^{(s,d)}_r)_{r\in[b(n,d)]}$ covers $\N^n_{d}$, i.e.,
\begin{equation}\label{eq:cover.diagonal}
    \begin{array}{l}
         \bigcup_{r=1}^{b(n,d)} \mathcal A^{(s,d)}_r = \N^n_{d}\,,
    \end{array}
\end{equation}
and the cardinality of  $\mathcal A^{(s,d)}_r$ is at most $s$.
Here $\check g_m:=1$.
We call $s$ the factor width upper bound associated with the semidefinite relaxation \eqref{eq:primal.problem.0.even.Han.intro}.
It is easy to see that the size of each Gram matrix $\mathbf{G}_{ij}$ in \eqref{eq:primal.problem.0.even.Han.intro} is at most $s$.
In addition, due to \eqref{eq:cover.diagonal}, we obtain the following estimate for every $s\in[b(n,k)]$:
\begin{equation}
    \rho_k^{\textup{\revise{P\'olya}}}= \rho_{k,1}^\textup{\revise{P\'olya}}\le \rho_{k,s}^\textup{\revise{P\'olya}}\le f^\star\,,
\end{equation}
so that for every fixed $s\in\N_{> 0}$, $\tau_{k,s}^\textup{\revise{P\'olya}}\to f^\star$ as $k$ increases, with a rate at least $\mathcal{O}(\varepsilon^{-\mathfrak{c}})$.
Notice that when $s=2$, \eqref{eq:primal.problem.0.even.Han.intro} becomes an SOCP thanks to \cite[Lemma 15]{magron2023sonc}.

We emphasize that in our semidefinite relaxation \eqref{eq:primal.problem.0.even.Han.intro}, 
for fixed $k$ the size of Gram matrices $\mathbf{G}_{ij}$ can be bounded from above by any $s\in\N_{> 0}$ while the maximal matrix size of the standard semidefinite relaxation for POP \eqref{eq:constrained.problem.poly.even} is fixed for each relaxation order $k$.
Nevertheless, since we convert \eqref{eq:constrained.problem.poly.even} to the form \eqref{eq:new.pop} (so as to use Corollary \ref{coro:compact.even}), 
the degrees of the {new} resulting objective and constraint polynomials are doubled, i.e., $\deg(\check f)=2\deg(f)$ and $\deg(\check g_i)=2\deg(g_i)$.

However, numerical experiments in Section \ref{sec:benchmark.polya} suggest that our method works { significantly} better than existing methods on examples of POPs with nonnegative variables. 
For instance, for $20$-variable dense POPs on the nonnegative orthant, the standard SOS-relaxations based on Putinar's Positivstellensatz provide a lower bound for $f^\star$ in {about} 356 seconds while  we can provide a better lower bound in {about} 5 seconds.

\revise{Finally, in our last contribution in Sections \ref{sec:linear.CS} and \ref{sec:interrup.CS}}, we provide two convergent hierarchies of linear and semidefinite relaxations for {large scale} POPs on the nonnegative orthant, that exploit \emph{correlative sparsity}, and with properties similar to those {in}  \eqref{eq:dual-sdp.0.even.LP.Han.intro} and \eqref{eq:primal.problem.0.even.Han.intro}.
Accordingly,
for POPs on the nonnegative orthant  with up to $1000$ variables, we can provide lower bounds 
in no more than $19$ seconds which are better than those obtained in about $56360$ seconds with the sparsity-adapted version
of the standard SOS-relaxations of  Waki et al. \cite{waki2006sums}.

\subsection*{Related works}
\paragraph{Exploiting sparsity:}

Structure exploitation in \eqref{eq:primal.problem.0.even.Han.intro} 
is comparable to term sparsity and correlative sparsity but 
here we can deal with dense POPs of the form \eqref{eq:constrained.problem.poly.even}. 
Moreover, the maximal block sizes involved in the sparsity-exploiting SDP relaxations mainly depend on the POP itself as well as on the relaxation order.
By comparison, the maximal block size of our SDP relaxations is controllable. 
\revise{
\paragraph{Exploiting even/sign symmetry:}
Another idea is to exploit even/sign symmetry based on Gatermann--Parrilo's work \cite[Section 8.1]{gatermann2004symmetry} for the semidefinite relaxations of POP \eqref{eq:new.pop};  see also L\"ofberg's paper \cite{lofberg2009pre}. 
Similarly to sparsity exploitation, even/sign symmetry exploitation relies on the even monomials existing in the input data to provide block-diagonal structures for the matrix variables of the relaxations; see Proposition \ref{prop:sparsity.Pu}. 
As explained later on in Remark \ref{rem:even.symmetry}, the block sizes of these form cannot be adjusted. 
Our method is also based on even/sign symmetry to obtain block-diagonal structures for the semidefinite relaxations. 
However, the block sizes of our relaxations can be appropriately calibrated. They can even be calibrated to one when using positivity certificates involving uniform denominators. 
}
\paragraph{Dickinson--Povh's hierarchy of linear relaxations:} Dickinson and Povh state in \cite{dickinson2015extension} the following constrained version of P\'olya's Positivstellensatz:
\revise{
\begin{theorem}\label{theo:dickinson.povh}
If $f,g_1,\dots,g_m$ are homogeneous polynomials, $S$ is defined as in \eqref{eq:semial.set.def.2.even}, and $f$ is positive on $S\backslash \{\mathbf{0}\}$, then  
    \begin{equation}\label{eq:PuVa.perturb2}
    \begin{array}{l}
         (\sum_{j\in[n]}x_j)^kf=\sum_{i\in[m]}\sigma_ig_i\,,
    \end{array}
    \end{equation}
    for some homogeneous polynomials $\sigma_i$ with positive coefficients. (Here $g_m:=1$.)
\end{theorem}
}
\revise{
Moreover, in \cite{dickinson2019new}, Dickinson and Povh constructed the following hierarchy of linear relaxations for problem \eqref{eq:constrained.problem.poly.even} with $f$ being a polynomial and $S$ being a semialgebraic set defined  as in \eqref{eq:semial.set.def.2.even}:
\begin{equation}\label{eq:Dickinson-Povh.hierarchy}
\begin{array}{rl}
   {\hat \rho^{\textup{\revise{P\'olya}}}_k}:= \sup\limits_{\lambda,\mathbf u_i} & \lambda\\
   \st& \lambda\in\R\,,\,\mathbf u_i\in\R_+^{b(n,k_i)}\,,\,i\in[m]\,,\\
   &(1+\sum_{j\in[n]} x_j)^k(f-\lambda)=\sum_{i\in[m]} g_i\v_{k_i}^\top\mathbf{u}_i)\,,
\end{array}
\end{equation}
where $f$ is of degree at most $d_f$, each $g_i$ is of degree at most $d_{g_i}$, and $k_{i}:=k+d_f - d_{g_i}$, for $i\in[m]$.
Using Theorem \ref{theo:dickinson.povh}, they obtain the convergence of $(\rho^{\textup{\revise{P\'olya}}}_k)_{k\in\N}$ toward $f^\star$.
}
    
The extension of P\'olya's Positivstellensatz restated in Corollary \ref{coro:compact.even} is indeed analogous to \eqref{eq:PuVa.perturb2}.
However the approach is different and importantly, the result is more convenient as we provide \emph{degree bounds} for the SOS of monomials involved in the representation.
Similarly, our corresponding linear relaxations  \eqref{eq:primal.problem.0.even.LP} are the analogues to
 \revise{relaxations \eqref{eq:Dickinson-Povh.hierarchy}} of Dickinson and Povh.
As shown in Example \ref{exam:AMGM} and other examples in Section \ref{sec:benchmark.polya}, this hierarchy of linear relaxations 
usually have a poor numerical behavior in practice when $k$ is large. \revise{The goal of
our new hierarchy of semidefinite relaxations \eqref{eq:primal.problem.0.even} is precisely to address this issue.}

\paragraph{DSOS and SDSOS:}  {In their recent work \cite{ahmadi2019dsos},
Ahmadi and Majumdar describe the two convex cones DSOS and SDSOS as
an alternative to the SOS cone. As the factor width of DSOS and SDSOS 
is at most $2$, they are more tractable than the SOS cone.} 
In the unconstrained case of POP \eqref{eq:new.pop}, our semidefinite hierarchy based on the extension of  P\'olya’s Positivstellensatz can be seen as a generalization of DSOS and SDSOS while using the notion of factor width, see Remark \ref{rem:SDSOS}.
In fact, to obtain our semidefinite relaxations for the constrained case \eqref{eq:new.pop}, we replace each SOS of monomials involved in the certificate \eqref{eq:Pu-va.ref.Pol} by an SOS polynomial whose Gram matrix has factor width at most $s$; see Remark \ref{rem:idea.replace.SOS}.

\section{Representation theorems: Extension of P\'olya's Positivstellensatz}

In this section, we derive representations of polynomials nonnegative on semialgebraic sets together with degree bounds. 

\subsection{Polynomials nonnegative on general semialgebraic sets}
We analyze the complexity of the extension of P\'olya's Positivstellensatz in the following theorem:
\begin{theorem}\label{theo:complex.putinar.vasilescu.even}
(Homogenized representation)
Let $g_1,\dots,g_m$ be homogeneous polynomials  such that $g_1,\dots,g_m$ are even in each variable.
Let $S$ be the semialgebraic set defined by
\begin{equation}\label{eq:semialgebraic.set}
    S:=\{\x\in\R^n\,:\,g_1(\x)\ge 0\dots,g_m(\x)\ge 0\}\,.
\end{equation}
Let $f$ be a homogeneous polynomial of degree $2d_f$ for some $d_f\in\N$ such that $f$ is even in each variable and nonnegative on $S$.
Then the following statements hold:
\begin{enumerate}
    \item For all $\varepsilon>0$, there exists $K_\varepsilon\in\N$ such that for all $k\ge K_\varepsilon$, there exist homogeneous SOS of monomials $\sigma_i$ satisfying
    \begin{equation}\label{eq:degree.SOS.even}
    \deg(\sigma_0)=\deg(\sigma_1g_1)=\dots=\deg(\sigma_mg_m)=2(k+d_f)
\end{equation}
and
\begin{equation}\label{eq:represent.even}
    \|\x\|_2^{2k}(f+\varepsilon\|\x\|_2^{2d_f})=\sigma_0+\sigma_1g_1+\dots+\sigma_mg_m\,.
\end{equation}
    \item If $S$ has nonempty interior, then there exist positive constants $\bar{\mathfrak{c}}$ and $\mathfrak{c}$ depending on $f,g_i$ such that for all $\varepsilon>0$,  one can take  $K_{\varepsilon} = 
 \bar{\mathfrak{c}}\varepsilon^{-\mathfrak{c}}$.
\end{enumerate}
\end{theorem}
The proof of Theorem \ref{theo:complex.putinar.vasilescu.even} is postponed to Section \ref{proof:PV.SOSmono}.

Note that some other homogeneous representations for globally nonnegative polynomials even in each variable have been studied in \cite{goel2017analogue,harris1999real,choi1987even}.
\revise{Accordingly, these polynomials have a small number of variables with low degrees and are represented as sums of squares of polynomials with degrees no more than half the maximal degree of the original polynomials.
Compared to this, Theorem \ref{theo:complex.putinar.vasilescu.even} provides a representation -- with denominators and implicit degree bounds -- of a homogeneous polynomial (with any number of variables and any degree) that is nonnegative over a semialgebraic set defined by homogeneous polynomial inequalities.
}

\begin{remark}
\revise{In Theorem \ref{theo:complex.putinar.vasilescu.even}}, the Gram matrix associated with each \revise{homogeneous} SOS of monomials is diagonal. 
In other word, it is a block-diagonal matrix with maximal block size one.
It would be interesting to know for which types of input polynomials we could obtain other representations involving SOS with block-diagonal Gram matrices of very small maximal block size, similarly to Theorem \ref{theo:complex.putinar.vasilescu.even}.
Some of them have been discussed in \cite{gouveia2022sums,magron2023sonc} that includes SOS of binomials, trinomials, tetranomials and SOS of any $s$-nomials.
We emphasize that such representations allow one to build up SDP relaxations of small maximal matrix size that can be solved efficiently by using interior-point methods as shown later in Section \ref{sec:benchmark.polya}.
\end{remark}

The following corollary is a direct consequence of Theorem \ref{theo:complex.putinar.vasilescu.even}.
\begin{corollary}
\label{coro:dehomo.even}
(Dehomogenized representation)
Let $g_1,\dots,g_m$ be polynomials even in each variable.  
Let $S$ be the semialgebraic set defined by \eqref{eq:semialgebraic.set}.
Let $f$ be a polynomial even in each variable and nonnegative on $S$.
Denote $d_f: = \lfloor \deg(f)/2\rfloor +1$.
Then the following statements hold:
\begin{enumerate}
    \item For all $\varepsilon>0$, there exists $K_\varepsilon\in\N$ such that for all $k\ge K_\varepsilon$, there exist SOS of monomials $\sigma_i$ satisfying
\begin{equation}\label{eq:bound.degree.dehomo}
    \deg(\sigma_0)\le 2(k+d_f)\quad\text{and}\quad\deg(\sigma_ig_i)\le 2(k+d_f)\,,\,i\in [m]\,,
\end{equation}
and 
\begin{equation}\label{eq:repre.dehomo}
   \theta^{k}(f+\varepsilon\theta^{d_f})=\sigma_0+\sigma_1g_1+\dots+\sigma_mg_m\,,
\end{equation}
where $\theta:=1+\|\x\|_2^2$.
\item If $S$ has nonempty interior,
there exist positive constants $\bar{\mathfrak{c}}$ and $\mathfrak{c}$ depending on $f,g_i$ such that for all $\varepsilon>0$, one can take 
$K_{\varepsilon} = \bar{\mathfrak{c}}\varepsilon^{-\mathfrak{c}}$. 
\end{enumerate} 
\end{corollary}
The proof of Corollary \ref{coro:dehomo.even} is similar to the proof of \cite[Corollary 1]{mai2022complexity}.

\subsection{Polynomials nonnegative on compact semialgebraic sets}

In this section, we provide a  representation of polynomials nonnegative on semialgebraic sets when the input polynomials are even in each variable.
We also derive in Section \ref{sec:sparse.rep} some sparse representations when the input polynomials have correlative sparsity.

The following \revise{result follows from} Corollary \ref{coro:dehomo.even}.
\begin{corollary}\label{coro:compact.even}
Let $f,g_i,S,d_f$ be as in Corollary \ref{coro:dehomo.even} such that $g_1:=R-\|\x\|_2^2$ for some $R>0$.
Then the following statements hold:
\begin{enumerate}
    \item For all $\varepsilon>0$, there exists $K_\varepsilon\in\N$ such that for all $k\ge K_\varepsilon$, there exist SOS of monomials $\sigma_i$ satisfying \eqref{eq:bound.degree.dehomo}
and 
\begin{equation}\label{eq:reperese.compact}
   (1+\|\x\|_2^2)^{k}(f+\varepsilon)=\sigma_0+\sigma_1g_1+\dots+\sigma_mg_m\,.
\end{equation}
\item If $S$ has nonempty interior,
there exist positive constants $\bar{\mathfrak{c}}$ and $\mathfrak{c}$ depending on $f,g_i$ such that for all $\varepsilon>0$, one can take 
$K_{\varepsilon} = \bar{\mathfrak{c}}\varepsilon^{-\mathfrak{c}}$.
\end{enumerate}
\end{corollary}
Corollary \ref{coro:compact.even} can be proved in the same way as \cite[
Corollary 2]{mai2022complexity}.
\begin{remark}
\label{re:non.block1.Pu}
If we remove the multiplier $(1+\|\x\|_2^2)^k$ in \eqref{eq:reperese.compact}, Corollary \ref{coro:compact.even} is no longer true.
Indeed, let $n=1$, $f:=(x^2-\frac{3}{2})^2$ and assume that  $f=\sigma_0+\sigma_1(1-x^2)$ for some SOS of monomials $\sigma_i$, $i=0,1$.
Note that $f$ is even and positive on $[-1,1]$.
We write $\sigma_i:=a_i+b_i x^2+x^4r_i(x)$ for some $a_i,b_i\in\R_+$ and $r_i\in\R[x]$.
It implies that
\begin{equation}
\begin{array}{rl}
     x^4-3x^2+\frac{9}{4}=(a_0+b_0 x^2+x^4r_0(x))+(a_1+b_1 x^2+x^4r_1(x))(1-x^2)\,.
\end{array}
\end{equation}
Then we obtain the system of linear equations:
$\frac{9}{4}=a_0+a_1$ and $-3=b_0-a_1+b_1$.
Summing gives $-\frac{3}{4}=a_0+b_0+b_1$. 
However, $a_0+b_0+b_1\ge 0$ since $a_i,b_i\in\R_+$.
This
contradiction yields the conclusion.
However, we are still able to exploit term sparsity/even symmetry for Putinar's Positivstellensatz in this case as shown later in Proposition \ref{prop:sparsity.Pu}.

It is not hard to see that with the multiplier $(1+x^2)^2$, we obtain the P\'olya’s Positivstellensatz as follows:
\begin{equation}\label{eq:Pu-Va.decomp.ins}
    \begin{array}{rl}
         (1+x^2)^2f=\bar \sigma_0+\bar \sigma_1(1-x^2)\,,
    \end{array}
\end{equation}
where $\bar\sigma_0:=x^8$ and $\bar \sigma_1:=x^4+\frac{15}{4}x^2+\frac{9}{4}$ are SOS of monomials.
\end{remark}
\revise{In the following proposition we prove the existence} of block-diagonal Gram matrices in Putinar's Positivstellensatz when the input polynomials are even in each variable:
\begin{proposition}
\label{prop:sparsity.Pu}
Let $f,g_1,\dots,g_m$ be polynomials in $\R[\x]$ such that $f,g_i$ are even in each variable.
Assume that there exists a decomposition:
\begin{equation}
\label{eq:Pu.rep}
\begin{array}{rl}
     f=\sum_{i=1}^m g_i\v_{d_i}^\top \mathbf{G}^{(i)}\v_{d_i}\,,
\end{array}
\end{equation}
for some $d_i\in\N$ and real symmetric  matrices $\mathbf{G}^{(i)}=({G}^{(i)}_{\a,\b})_{ \a,\b\in\N^n_{d_i}}$.
For every $i\in [m]$, define $\bar{\mathbf{G}}^{(i)}:=(\bar G^{(i)}_{\a,\b})_{ \a,\b\in\N^n_{d_i}}$, where:
\begin{equation}
    \bar G^{(i)}_{\a,\b}:=\begin{cases}
    G^{(i)}_{\a,\b} & \text{if }\a+\b\in 2\N^n\,,\\
    0& otherwise\,.
    \end{cases}
\end{equation}
Then $\bar{\mathbf{G}}^{(i)}$ are block-diagonal up to permutation and 
\begin{equation}
\label{eq:block.Pu}
\begin{array}{rl}
     f=\sum_{i=1}^m g_i\v_{d_i}^\top \bar{\mathbf{G}}^{(i)}\v_{d_i}\,.
\end{array}
\end{equation}
Moreover, if $\mathbf{G}^{(i)}\succeq 0$, then $\bar{\mathbf{G}}^{(i)}\succeq 0$.
\end{proposition}
\begin{proof}
The proof is inspired by \cite[Section 8.1]{gatermann2004symmetry}.
Removing all terms in \eqref{eq:Pu.rep} except the terms of monomials $\x^{2\a}$, $\a\in\N^n$, we obtain \eqref{eq:block.Pu}.
It is due to the fact that $f,g_i$ only have terms of the form $\x^{2\a}$, $\a\in\N^n$ and
\begin{equation}
\begin{array}{rl}
     \v_{d_i}^\top \mathbf{G}^{(i)}\v_{d_i}=\sum_{\a,\b\in\N^n_{d_i}}G^{(i)}_{\a,\b} \x^{\a+\b}\,.
\end{array}
\end{equation}
Next, we show the block-diagonal structure of $\bar{\mathbf{G}}^{(i)}$.
For every $\g\in\{0,1\}^n$, define
\begin{equation}
    \Lambda_\g^{(i)}:=\{\a\in\N^n_{d_i}\,:\ ,\a-\g\in2\N^n\}\,.
\end{equation}
Then $\Lambda_\g^{(i)}\cap \Lambda_{\boldsymbol{\eta}}^{(i)}=\emptyset$ if $\g\ne \boldsymbol{\eta}$ and 
$\N^n_{d_i}:=\cup_{\g\in\{0,1\}^n} \Lambda_\g^{(i)}$.
In addition, for all $\a,\b\in \Lambda_\g^{(i)}$, $\a+\b\in 2\N^n$. Moreover, if $\a,\b\in\N^n_{d_i}$ and $\a+\b\in 2\N^n$, then there exists $\g\in \{0,1\}^n$ such that $\a,\b\in \Lambda_\g^{(i)}$.
It implies that all blocks on the diagonal of $\bar{\mathbf{G}}^{(i)}$ must be 
\begin{equation}
    (\bar G^{(i)}_{\a,\b})_{\a,\b\in \Lambda_\g^{(i)}}\,,\,\g\in\{0,1\}^n\,.
\end{equation}
This yields the desired results.
\end{proof}
\begin{remark}\label{rem:even.symmetry}
The block-diagonal structure in Proposition \ref{prop:sparsity.Pu} can be obtained by using TSSOS \cite{wang2021tssos}.
For general input polynomials $f,g_i$, we cannot  ensure that the maximal block size in this form is upper bounded or possibly goes to infinity as each $d_i$ increases.
However, as shown in Remark \ref{re:non.block1.Pu}, we cannot obtain blocks of size one for this form.
\end{remark}

\revise{
\begin{remark}
The following Handelman-type Positivstellensatz is  a consequence of Theorem \ref{theo:complex.putinar.vasilescu.even}:
With the same assumption as in Corollary \ref{coro:compact.even}, $g_m:=1$, $d_{g_i}:=\lceil \deg(g_i)/2\rceil$, for all $\varepsilon>0$, there exists $K_\varepsilon\in\N$ such that for all $k\ge K_\varepsilon$, there exist SOS of monomials $\sigma_{i,j}$ satisfying
\begin{equation}\label{eq:bound.degree.dehomo.without.mul}
    \deg(\sigma_{i,j}g_1^jg_i)\le 2k
\end{equation}
and 
\begin{equation}\label{eq:repre.dehomo.without.mul}
\begin{array}{rl}
     f+\varepsilon=\sum_{i=1}^m\sum_{j=0}^{k-d_{g_i}}\sigma_{i,j}g_1^jg_i\,.
\end{array}
\end{equation}
If $S$ has nonempty interior, then
there exist positive constants $\bar{\mathfrak{c}}$ and $\mathfrak{c}$ depending on $f,g_i$ such that for all $\varepsilon>0$, one can take 
$K_{\varepsilon} = \bar{\mathfrak{c}}\varepsilon^{-\mathfrak{c}}$.
Although the representation \eqref{eq:repre.dehomo.without.mul} does not involve  denominators as in \eqref{eq:reperese.compact}, the number of SOS of monomials in the representation \eqref{eq:repre.dehomo.without.mul} is $\sum_{i=1}^m(k-d_{g_i}+1)$ which becomes larger when $k$ increases, while the number of SOS of monomials in the representation \eqref{eq:reperese.compact} is $m+1$, which does not depend on $k$.
Practical experiments have been conducted to confirm  that this Handelman-type hierarchy for polynomial optimization does not stand out compared to P\'olya's.
\end{remark}
}

\section{Polynomial optimization on the nonnegative orthant: Compact case}
\label{sec:compact.pop}
This section is concerned with some applications {of (i) the extension of P\'olya's Positivstellensatz \eqref{eq:reperese.compact} for polynomial optimization on compact semialgebraic subsets of the nonnegative orthant.
The noncompact case is postponed to Section \ref{sec:pop.noncompact}.
Moreover, Section \ref{sec:sparse.POP.nonneg} is devoted to some applications of the sparse representation provided in Section \ref{sec:sparse.rep} for polynomial optimization with correlative sparsity.

Consider the following POP:
\begin{equation}\label{eq:constrained.problem.poly.even.LP}
\begin{array}{l}
f^\star:=\inf\limits_{\x\in S} f( \x)\,,
\end{array}
\end{equation}
where $f\in\R[\x]$ and
\begin{equation}\label{eq:semial.set.def.2.even.LP}
    S=\{\x\in\R^n\,:\,x_j\ge 0\,,\,j\in[n]\,,\,g_i(\x)\ge 0\,,\,i\in[m]\}\,,
\end{equation}
for some $g_i\in\R[\x]$, $i\in[m]$, with $g_m=1$.
Throughout this section, we assume that $f^\star>-\infty$ and problem \eqref{eq:constrained.problem.poly.even.LP} has an optimal solution $\x^\star$. 
\begin{remark}\label{rem:convert.POP.none.orth}
 Every general POP in variable $\x=(x_1,\dots,x_n)$ can be converted to the form \eqref{eq:constrained.problem.poly.even.LP} { with $S$ as in \eqref{eq:semial.set.def.2.even.LP},} by replacing each variable $x_j$ by the difference of two new nonnegative variables $x_j^+-x_j^-$.
 If there are several constraints $x_j\ge a_j$, we can obtain an equivalent POP on the nonnegative orthant by defining new nonnegative variables $y_j:=x_j-a_j$.
 {In particular, we can easily convert a POP over a compact semialgebraic set to POP over the nonnegative orthant by changing the coordinate via an affine transformation.}
 However, we restrict ourselves to POPs on the nonnegative orthant in this paper.
\end{remark}

Recall that  $\check q(\x):=q(\x^2)$, for a given  polynomial $q$.
In this case, $\check q$ is even in each variable.
Then POP \eqref{eq:constrained.problem.poly.even.LP} is equivalent to
\begin{equation}\label{eq:equi.POP}
    f^\star:=\inf_{\x\in \check S}\check f\,,
\end{equation}
where
\begin{equation}
    \check S=\{\x\in\R^n\,:\,\check g_i(\x)\ge 0\,,\,i\in[m]\}\,,
\end{equation}
with $\x^{\star2}$ being an optimal solution.

Let $\theta:=1+\|\x\|_2^2$.
Denote $d_f: =  \deg(f) +1$, $d_{g_i}: =  {\deg ( {{g_i}} )} $, $i\in[m]$.

\subsection{Linear relaxations based on the extension of P\'olya's Positivstellensatz}
Consider the hierarchy of linear programs indexed by $k\in\N$: 
\begin{equation}\label{eq:dual-sdp.0.even.LP}
\begin{array}{rl}
{\tau_k^{\textup{\revise{P\'olya}}}}: = \inf\limits_\y &{L_\y}( {\theta ^k}\check f  )\\
\st&\y= {(y_\a )_{\a  \in \N^n_{2( {d_f + k})}}} \subset \R\,,\,{L_\y}( {{\theta ^k}}) = 1\,,\\
&\diag({\M_{k_i}}( {{\check g_i}\y}))\in \R_+^{b(n,k_i)}\,,\,i \in[m]\,,
\end{array}
\end{equation}
where $k_i:=k + d_f - d_{g_i}$, $i\in[m]$.
Note that $\check g_m=1$.
\begin{remark}
 The optimal value $\tau_k^{\textup{\revise{P\'olya}}}$ only depends on the subset of variables $\{y_{2\a}\,:\,\a\in\N^n_{d_f+k}\}$, i.e., the optimal value of LP \eqref{eq:dual-sdp.0.even.LP} does not change when we assign each of the other variables with any real number.
It is due to the fact that $\theta$, $\check f$, and $\check g_i$ only have nonzero coefficients associated to the monomials $\x^{2\a}$ for some $\a\in\N^n$.
\end{remark}
\begin{theorem}\label{theo:constr.theo.0.even.LP}
Let $f,g_i\in\R[\x]$, $i\in[m]$, with $g_m=1$ and $g_1:=R-\sum_{j\in[n]}x_j$ for some $R>0$.
Consider POP \eqref{eq:constrained.problem.poly.even.LP} with $S$ being defined as in \eqref{eq:semial.set.def.2.even.LP}.
For every $k\in\N$, the dual of \eqref{eq:dual-sdp.0.even.LP} reads as:
\begin{equation}\label{eq:primal.problem.0.even.LP}
\begin{array}{rl}
   {\rho _k^{\textup{\revise{P\'olya}}}}:= \sup\limits_{\lambda,\mathbf{u}_i} & \lambda\\
   \st& \lambda\in\R\,,\,\mathbf{u}_i\in\R_+^{b(n,k_i)}\,,\,i\in[m]\,,\\
   &\theta^k(\check f-\lambda)=\sum_{i\in[m]} \check g_i\v_{k_i}^\top \diag(\mathbf{u}_i) \v_{k_i}\,.
\end{array}
\end{equation}
The following statements hold:
\begin{enumerate}
\item For all $k\in\N$,
\begin{equation}\label{eq:ineq.value}
    \rho_k^{\textup{\revise{P\'olya}}}\le\rho_{k+1}^{\textup{\revise{P\'olya}}}\le f^\star\,.
\end{equation}
\item The sequence $(\rho_k^{\textup{\revise{P\'olya}}})_{k\in\N}$ converges to $f^\star$.
\item If $S$ has nonempty interior, there exist positive constants $\bar{\mathfrak{c}}$ and $\mathfrak{c}$ depending on $f,g_i$ such that 
$0\le f^\star-\rho_k^{\textup{\revise{P\'olya}}}\le \left(\frac{k}{\bar{\mathfrak{c}}}\right)^{-\frac{1}{\mathfrak{c}}}$.
\end{enumerate}
\end{theorem}
The proof of Theorem \ref{theo:constr.theo.0.even.LP} relies on Corollary \ref{coro:compact.even} and can be proved in almost
the same way as the proof of
\cite[Theorem 4]{mai2022complexity}.

\subsection{Semidefinite relaxations based on the extension of P\'olya's Positivstellensatz}
\label{sec:dense.POP}

\revise{To control the size of the matrix variables in the semidefinite relaxations based on the extension of P\'olya's Positivstellensatz, we will construct the sparsity pattern $(\mathcal A^{(s,d)}_j)_{j=1}^{b(n,d)}$ with $\mathcal A^{(s,d)}_j\subset \N^n_d$ inspired by even symmetry reduction in Proposition \ref{prop:sparsity.Pu}.
Accordingly, each $\mathcal A^{(s,d)}_j$ contains the indices of a principal submatrix of moment/localizing matrices existing in the moment relaxations. 
Therefore, $\mathcal A^{(s,d)}_j$ will be chosen such that it has no more than $s$ elements, for a given $s\in\N_{>0}$.
Furthermore, to ensure convergence for the new semidefinite relaxations associated with the sparsity pattern $(\mathcal A^{(s,d)}_j)_{j=1}^{b(n,d)}$, the construction needs to satisfy the condition $\bigcup_{j=1}^{b(n,d)} \mathcal A^{(s,d)}_j = \N^n_d$.}

To do so, we write $\N^n=\{\a_1,\a_2,\dots,\a_r,\a_{r+1},\dots\}$ such that 
\begin{equation}
    \a_1< \a_2< \dots< \a_r<\a_{r+1}<\dots\,.
\end{equation}
\revise{Here we arrange the elements of $\N^n$ in the lexicographic order.}
Let 
\begin{equation}
    W_j:=\{i\in\N\,:\,i\ge j\,,\,\a_i+\a_j\in2\N^n\}\,,\quad j\in\N_{> 0}\,.
\end{equation}
Then for all $j\in\N_{> 0}$, $W_j\ne \emptyset$ since $j\in W_j$.
For every $j\in\N$, we write $W_j:=\{i^{(j)}_1,i^{(j)}_{2},\dots\}$ such that $j=i^{(j)}_1<i^{(j)}_{2}<\dots$.
Let 
\begin{equation}
    \mathcal T_j^{(s,d)}=\{\a_{i^{(j)}_1},\dots,\a_{i^{(j)}_{s}}\}\cap \N^n_d\,,\quad j,s\in\N_{> 0}\,,\,d\in\N\,.
\end{equation}
\revise{Thus $\mathcal T_j^{(s,d)}$ has at most $s$ elements $\alpha$ such that $\alpha+\alpha_j$ has only even entries.
Since $j=i^{(j)}\in W_j$, if $\alpha_j\in \N^n_d$, then $\alpha_j\in \mathcal T_j^{(s,d)}$.}
We now construct the sequence $(\mathcal A^{(s,d)}_j)_{j=1}^{b(n,d)}$ by induction as follows:}
For every $s\in \N_{> 0}$ and $d\in\N$, define $\mathcal A^{(s,d)}_1:=\mathcal T_1^{(s,d)}$ and for $j=2,\dots,b(n,d)$, define
\begin{equation}
    \mathcal A^{(s,d)}_j:=\begin{cases}
         \mathcal T_j^{(s,d)}&\text{if } {\cal T}_j^{(s,d)}\backslash \mathcal A^{(s,d)}_l\ne \emptyset\,,\,\forall l\in[j-1]\,,\\
         \emptyset & \text{otherwise}\,.
         \end{cases}
\end{equation}
\revise{The condition ${\cal T}_j^{(s,d)}\backslash \mathcal A^{(s,d)}_l\ne \emptyset\,,\,\forall l\in[j-1]$ ensures that the set $\mathcal A^{(s,d)}_j$ is not contained in $\mathcal A^{(s,d)}_l$, for $l\in[j-1]$. This is because 
\begin{equation}
{\cal T}_j^{(s,d)}\backslash \mathcal A^{(s,d)}_l= \emptyset\Leftrightarrow {\cal T}_j^{(s,d)}\subset \mathcal A^{(s,d)}_l\,.
\end{equation}
In addition, if $\mathcal A^{(s,d)}_j\subset \mathcal A^{(s,d)}_l$, the positive semidefiniteness of the principal submatrix indexed by $\mathcal A^{(s,d)}_l$ implies the positive semidefiniteness of the principal submatrix indexed by $\mathcal A^{(s,d)}_j$.
This case leads to the duplication of the constraints in the later semidefinite relaxations.}
\revise{It is not hard to prove that with the above construction of $(\mathcal A^{(s,d)}_j)_{j=1}^{b(n,d)}$}, one has  $\bigcup_{j=1}^{b(n,d)} \mathcal A^{(s,d)}_j = \N^n_d$ and $| \mathcal A^{(s,d)}_j|\le s$.
Here $|\cdot|$ stands for the cardinality of a set.
Then the sequence
\begin{equation}
    (\a+\b)_{(\a,\b\in\mathcal A^{(s,d)}_j)}\,,\,j\in[b(n,d)]
\end{equation}
are overlapping blocks of size at most $s$ in 
$(\a+\b)_{(\a,\b\in\N^n_d)}$.
Note that $\a+\b\in2\N^n$ for all $\a,\b\in\mathcal A^{(s,d)}_j$.
\begin{example}
Consider the case of $n=d=s=2$.
The matrix  $(\a+\b)_{(\a,\b\in \N^2_2)}$
can be written explicitly as
\begin{equation}\label{eq:mat.exam}
\begin{bmatrix}
{\bf(0,0)} & (1,0) & (0,1)& {\bf(2,0)} & (1,1) & {\bf(0,2)}\\
(1,0) & {\bf(2,0)} & (1,1)& (3,0) & (2,1) & (1,2)\\
(0,1) & (1,1) & {\bf(0,2)}& (2,1) & (1,2) & (0,3)\\
{\bf(2,0)} & (3,0) & (2,1)& {\bf(4,0)} & (3,1) & {\bf(2,2)}\\
(1,1) & (2,1) & (1,2)& (3,1) & {\bf(2,2)} & (1,3)\\
{\bf(0,2)} & (1,2) & (0,3)& {\bf(2,2)} & (1,3) & {\bf(0,4)}
\end{bmatrix} \,.
\end{equation} 
In this matrix, the entries in bold belong to $2\N^2$.
Then  $W_1=\{1,4,6\}$. Since $s=2$, we get
$\mathcal A^{(2,2)}_1=\{{(0,0)} , {(2,0)}\}$. Similarly, we can obtain $\mathcal A^{(2,2)}_2=\{{(1,0)}\}$, $\mathcal A^{(2,2)}_3=\{{(0,1)}\}$, $\mathcal A^{(2,2)}_4=\{{(2,0)},{(0,2)}\}$, $\mathcal A^{(2,2)}_5=\{{(1,1)}\}$ and $\mathcal A^{(2,2)}_6=\emptyset$.
The blocks $(\a+\b)_{(\a,\b \in\mathcal A^{(2,2)}_j)}$, $j\in[5]$, are as follows:
\begin{equation}
\begin{bmatrix}
{\bf(0,0)} &  {\bf(2,0)} \\
{\bf(2,0)} & {\bf(4,0)}
\end{bmatrix}\,,\,
\begin{bmatrix}
{\bf(2,0)}
\end{bmatrix}\,,\,
\begin{bmatrix}
{\bf(0,2)}
\end{bmatrix}\,,\,
\begin{bmatrix}
{\bf(4,0)} &  {\bf(2,2)} \\
{\bf(2,2)} & {\bf(0,4)}
\end{bmatrix}\,,\,
\begin{bmatrix}
{\bf(2,2)}
\end{bmatrix}\,.
\end{equation}
\end{example}
For all $\mathcal{B}=\{\b_1,\dots,\b_r\}\subset \N^n$ such that $\b_1<\dots<\b_r$, for every $h=\sum_{\g}h_\g \x^\g\in\R[\x]$ and for every $\y=(y_\a)_{\a\in\N^n}\subset \R$, let us define 
\begin{equation}
\begin{array}{l}
     \v_{\mathcal{B}}:=\begin{bmatrix} \x^{\b_1}\\
    \dots\\
    \x^{\b_r}
    \end{bmatrix}\quad\text{and}\quad \M_{\mathcal{B}}(h\y):=(\sum_{\g}h_\g y_{\g+\b_i+\b_j})_{i,j\in[r]}\,.
\end{array}
\end{equation}

Consider the hierarchy of semidefinite programs indexed by  $s\in\N_{> 0}$ and $k\in\N$: 
\begin{equation}\label{eq:dual-sdp.0.even}
\begin{array}{rl}
{\tau_{k,s}^\textup{\revise{P\'olya}}}: = \inf\limits_\y &{L_\y}( {\theta ^k}\check f  )\\
\st&\y= {(y_\a )_{\a  \in \N^n_{2( {d_f + k})}}} \subset \R\,,\,{L_\y}( {{\theta ^k}}) = 1\,,\\[5pt]
&\M_{\cA^{(s,k_i)}_j}(\check g_i\y)\succeq 0\,,\,j\in[b(n,k_i)]\,,\,i \in[m]\,,
\end{array}
\end{equation}
where $k_i:=k + d_f - d_{g_i}$, $i\in[m]$. Here $\check g_m=1$.
\begin{remark}
\label{re:special.case}
If we assume that $\theta=1$, then  \eqref{eq:dual-sdp.0.even} becomes a moment relaxation based on Putinar's Positivstellensatz for POP \eqref{eq:new.pop}.
Here each constraint $\M_{k_i}(\check g_i\y)\succeq 0$ is replaced by the constraint  $\M_{\cA^{(s,k_i)}_j}(\check g_i\y)\succeq 0$.
If $s$ is large enough, \eqref{eq:dual-sdp.0.even} corresponds to an SDP relaxation obtained after exploiting term sparsity (see \cite{wang2021tssos}).
\end{remark}
\begin{theorem}\label{theo:constr.theo.0.even}
Let $f,g_i\in\R[\x]$, $i\in[m]$, with $g_m=1$ and $g_1:=R-\sum_{j\in[n]}x_j$ for some $R>0$.
Consider POP \eqref{eq:constrained.problem.poly.even.LP} with $S$ being defined as in \eqref{eq:semial.set.def.2.even.LP}.
For every $s\in\N_{> 0}$ and for every $k\in\N$, the dual of \eqref{eq:dual-sdp.0.even} reads as:
\begin{equation}\label{eq:primal.problem.0.even}
\begin{array}{rl}
   {\rho _{k,s}^\textup{\revise{P\'olya}}}:= \sup\limits_{\lambda,\mathbf{G}_{ij}} & \lambda\\
   \st& \lambda\in\R\,,\,\mathbf{G}_{ij}\succeq 0\,,\,j\in[b(n,k_i)]\,,\,i\in[m]\,,\\[5pt]
   &\theta^k(\check f-\lambda)=\sum_{i\in[m]} \check g_i\big(\sum_{j\in[b(n,k_i)]}\v_{\cA^{(s,k_i)}_j}^\top \mathbf{G}_{ij} \v_{\cA^{(s,k_i)}_j}\big)\,.
\end{array}
\end{equation}
The following statements hold:
\begin{enumerate}
\item For all $k\in\N$ and for every $s\in\N_{> 0}$,
$\rho_{k}^\textup{\revise{P\'olya}}=\rho_{k,1}^\textup{\revise{P\'olya}}\le\rho_{k,s}^\textup{\revise{P\'olya}}\le f^\star$.
\item For every $s\in\N_{> 0}$, the sequence $(\rho_{k,s}^\textup{\revise{P\'olya}})_{k\in\N}$ converges to $f^\star$.
\item If $S$ has nonempty interior, there exist positive constants $\bar{\mathfrak{c}}$ and $\mathfrak{c}$ depending on $f,g_i$ such that for every $s\in\N_{> 0}$ and for every $k\in\N$,
$0\le f^\star-\rho_{k,s}^\textup{\revise{P\'olya}}\le \left(\frac{k}{\bar{\mathfrak{c}}}\right)^{-\frac{1}{\mathfrak{c}}}$.
\item If $S$ has nonempty interior, for every $k\in\N$ and for every $s\in\N_{> 0}$, strong duality holds for the  primal-dual problems \eqref{eq:dual-sdp.0.even}-\eqref{eq:primal.problem.0.even}.
\end{enumerate}
\end{theorem}
\begin{proof}
It is not hard to prove the first statement. The second and third one are due to the first statement of Theorem \ref{theo:constr.theo.0.even.LP}.
The final statement is proved similarly to the third statement of \cite[Theorem 4]{mai2022complexity}.
\end{proof}
\begin{remark}\label{rem:idea.replace.SOS}
In order to construct the semidefinite relaxation \eqref{eq:primal.problem.0.even}, the SOS of monomials in the linear relaxation \eqref{eq:primal.problem.0.even.LP} are replaced by a sum of several SOS polynomials \revise{with associated Gram matrices of small size}. 
This idea is inspired by  \cite{weisser2018sparse} where the authors replace the first nonnegative scalar by an SOS polynomial in the linear relaxation based on Krivine-Stengle's Positivstellensatz.
\end{remark}
\revise{
\begin{remark}
SDP \eqref{eq:primal.problem.0.even} with value ${\rho _{k,s}^\textup{P\'olya}}$ has at most $\sum_{i=1}^m b(n,k_i)$ matrix variables with maximal block size $s$. 
Its number of affine equality constraints is $b(n,k+d_f)$.
When $s=1$, \eqref{eq:primal.problem.0.even} becomes LP \eqref{eq:primal.problem.0.even.LP} with value ${\rho _k^{\textup{\revise{P\'olya}}}}$. This latter program has the same number of equality constraints than SDP~ \eqref{eq:primal.problem.0.even} and $\sum_{i=1}^m b(n,k_i)$ nonnegative variables.
\end{remark}
}
\begin{remark}
\label{re:not.monomonic}
At fixed $s\in\N_{> 0}$, the sequence $(\rho_{k,s}^\textup{\revise{P\'olya}})_{k\in\N}$  may not be monotonic w.r.t.~$k$, and similarly at fixed $k\in\N$.
\end{remark}
\begin{example}\label{exam:AMGM}(AM-GM inequality)
Consider the case where $n=3$, $f=x_1+x_2+x_3$ and $S=\{\x\in\R^3\,:\,x_j\ge 0\,,\,j\in[3]\,,\,x_1x_2x_3 -1\ge 0\,,\,3-x_1-x_2-x_3\ge 0\}$.
Using AM-GM inequality, we have
\begin{equation}
    f(\x)\ge 3(x_1x_2x_3)^{1/3}\ge 3\,,\quad\forall x\in S\,,
\end{equation}
yielding $f^\star=3$.
\revise{We use {\tt JuMP} \cite{lubin2023jump} to model SDP \eqref{eq:dual-sdp.0.even} and}  solve it with {\tt Mosek}.
The corresponding  numerical results are reported in Table \ref{tab:AMGM}.
\begin{table}
    \caption{Numerical values (in the first subtable) and computing time (in the second subtable) for  $\tau_{k,s}^\textup{\revise{P\'olya}}$ in Example \ref{exam:AMGM}}
    \label{tab:AMGM}
    \small
\begin{center}
   \begin{tabular}{|c|c|c|c|c|c|c|c|c|}
        \hline
\diagbox[width=0.8cm]{$k$}{$s$}&1 & 2& 3 & 4& 5&6 &7 &8 \\
\hline
0&  0.0000 & 0.0000& 0.0000 & 0.0000& 0.0000 & 0.0000 & 0.0000 & 0.0000 \\
\hline
1&  0.0000 & 0.0000& 0.0000 & 0.0000& 0.0000 & 0.0000 & 0.0000 & 0.0000 \\
\hline
2&0.0000 &0.0000 & 0.4999 & \textbf{2.9999}& \textbf{2.9999} & \textbf{2.9999} & \textbf{2.9999} &\textbf{2.9999}\\
\hline
3& 1.0000 & 0.9999 & 0.9999 & 2.7454 & 2.8368 & 2.8383 & \textbf{2.9999} & \textbf{2.9999}\\
\hline
4& 1.4399 & 1.4999 & 1.4999& 1.4999& 1.4999 & 1.4999 & \textbf{2.9999}& \textbf{2.9999}  \\
\hline
5 & 1.8615 & 1.9961 & 1.9999& 1.9999 & 1.9999 & 1.9999 & \textbf{2.9999} & \textbf{2.9999}  \\
\hline
6& 2.1999 & 2.4526 & 2.4998 & 2.4999 & 2.4999 & 2.4999 & 2.4999& 2.4999\\
\hline
7& 2.3971 & 2.8090 & 2.9633 & 2.9950 & 2.9996 & \textbf{2.9999} & \textbf{2.9999} & \textbf{2.9999}\\
\hline
8& 2.4109 & 2.9022 & 2.9989 & \textbf{2.9999} & \textbf{2.9999} & \textbf{2.9999} &\textbf{2.9999} & \textbf{2.9999}\\
\hline
9& 2.5161 & 2.9137 & 2.9997 & \textbf{2.9999} &\textbf{2.9999} & \textbf{2.9999}&\textbf{2.9999} & \textbf{2.9999}\\
\hline
10& 2.5896 & 2.9520 & 2.9993 & \textbf{2.9999} &\textbf{2.9999} & \textbf{2.9999}&\textbf{2.9999} & \textbf{2.9999}\\
\hline
11& 2.6210 & 2.9607 & 2.9983 & \textbf{2.9999} &\textbf{2.9999} & \textbf{2.9999}&\textbf{2.9999} & \textbf{2.9999}\\
\hline
12 & 2.6937 & 2.9615 & 2.9973 & 2.9998 & \textbf{2.9999}&\textbf{2.9999} & \textbf{2.9999}&\textbf{2.9999} \\
\hline
13 & 2.7330 & 2.9662 & 2.9977 & \textbf{2.9999} & \textbf{2.9999}&\textbf{2.9999} & \textbf{2.9999}&\textbf{2.9999} \\
\hline
14 & 2.7390 & 2.9687 & 2.9974 & \textbf{2.9999} & \textbf{2.9999}&\textbf{2.9999} & \textbf{2.9999}&\textbf{2.9999} \\
\hline
15 & 2.3704 & 2.9697 & 2.9972 & 2.9998 & \textbf{2.9999} & \textbf{2.9999} & \textbf{2.9999} & \textbf{2.9999}  \\
\hline
16 & 2.4000 & 2.9710 & 2.9971 & 2.9997 & \textbf{2.9999} &\textbf{2.9999} & \textbf{2.9999} & \textbf{2.9999}  \\
\hline
17 & 1.5030 & 2.9723 & 2.9968 & \textbf{2.9999} & \textbf{2.9999} & \textbf{2.9999} & \textbf{2.9999} & \textbf{2.9999}\\
\hline
18 & 0.5833 & 2.9732 & 2.9966 & 2.9996 & \textbf{2.9999} & \textbf{2.9999} & \textbf{2.9999} & \textbf{2.9999}\\
\hline
19 & 0.8121 & 0.0000 & 0.0000 & 2.9995 & \textbf{2.9999} & \textbf{2.9999} & \textbf{2.9999} & \textbf{2.9999}\\
\hline
20 & 0.7457 & 0.0000 & 0.0000 & 2.9994 & \textbf{2.9999} & \textbf{2.9999} & \textbf{2.9999} &\textbf{2.9999}  \\
\hline
\end{tabular} 

\begin{tabular}{|c|c|c|c|c|c|c|c|c|}
        \hline
\diagbox[width=0.8cm]{$k$}{$s$}&1 & 2& 3 & 4& 5&6 &7 &8 \\
\hline
0&  1.1 & 1.3& 1.0 & 1.0& 1.0 & 1.1 & 1.0 & 1.0 \\
\hline
1&  1.1 & 1.1& 1.1 & 1.2& 1.1 & 1.1 & 1.1 & 1.1 \\
\hline
2&1.1 &1.1 & 1.1 & 1.1& 1.1 & 1.1 & 1.1 &1.1\\
\hline
3& 1.1 & 1.1 & 1.1 & 1.1 & 1.1 & 1.1 & 1.5 & 1.1\\
\hline
4& 1.1 & 1.2 & 1.1& 1.1& 1.1 & 1.1 & 1.2& 1.2  \\
\hline
5 & 1.1 & 1.1 & 1.1& 1.2 & 1.2 & 1.3 &1.2 & 1.3  \\
\hline
6& 1.2 & 1.2 & 1.2 & 1.3 & 1.3 & 1.5 & 1.3& 1.2\\
\hline
7& 1.4 & 1.2 & 1.2 & 1.4 & 1.4 & 1.6 & 1.6 & 1.4\\
\hline
8& 1.2 & 1.2 & 1.3 & 1.3 & 1.5 & 1.4 &1.8 & 1.9\\
\hline
9& 1.3 & 1.2 & 1.3 & 1.4 &1.5 & 1.7&1.7 & 1.6\\
\hline
10& 1.3 & 1.3 & 1.5 & 1.8 &1.7 & 1.9&2.2 & 1.9\\
\hline
11& 1.3 & 1.5 & 1.4 & 1.9 &1.9 & 2.0&2.2 & 2.3\\
\hline
12 & 1.3 & 1.7 & 1.8 & 2.1 & 2.2&2.3 & 2.7&2.6\\
\hline
13 & 1.4 & 1.6 & 1.9 & 2.2 & 2.3&2.4 & 2.9&3.4 \\
\hline
14 & 1.2 & 1.5 & 2.0 & 2.5 & 2.6&2.9 & 3.5&3.8 \\
\hline
15 & 1.2 & 1.6 & 2.3 & 2.8 & 3.1 & 3.5 & 4.2 & 5.0  \\
\hline
16 & 1.3 & 2.5 & 2.8 & 3.5 & 3.9 &4.4 & 5.9 & 7.1  \\
\hline
17 & 1.4 & 2.3 & 3.8 & 5.3& 6.2 & 7.2 & 7.9 & 9.7\\
\hline
18 & 1.6 & 2.9 & 5.2 & 7.2 & 7.9 & 9.7 & 10.6 & 12.3\\
\hline
19 & 1.5 & 2.7 & 4.1 & 9.8 & 13.3 & 14.0 & 14.1 & 16.6\\
\hline
20 & 1.4 & 3.4 & 4.8 & 12.6 & 16.5 & 20.8 & 24.5 &27.2  \\
\hline
\end{tabular} 
\end{center}
\end{table}
The table displays  $\tau_{2,4}^\textup{\revise{P\'olya}} \revise{\approx} 2.9999$  which is very close to $f^\star$.
However, $\tau_{17}^{\textup{\revise{P\'olya}}}=\tau_{17,1}^\textup{\revise{P\'olya}}$ \revise{has numerical value} $1.5030$, which is smaller than \revise{the numerical value} 2.4000 of $\tau_{16}^{\textup{\revise{P\'olya}}}=\tau_{16,1}^\textup{\revise{P\'olya}}$.
This violates the theoretical inequality \eqref{eq:ineq.value}. 
\revise{This is because the numerical values reported by the solver do not approximate accurately the true optimal value $\tau_{k,s}^\textup{\revise{P\'olya}}$ of the relaxations in this case.  \revise{The solver neither reports convergence to the optimum nor satisfaction} of the KKT conditions.}
The underlying reason is that the matrix $A$ used to define the convex polytope $P=\{\x\in\R^n\,:\,\x\ge 0\,,\,\mathbf A\x\le \mathbf b\}$ in the equivalent form $\min_{\x\in P} \mathbf c^\top \x$
of LP \eqref{eq:dual-sdp.0.even.LP} is ill-conditioned, and the solver
is not able to  accurately solve the LP corresponding to $\tau^\textup{\revise{P\'olya}}_{17}$.
\revise{Note that the LP could be solved exactly with a solver based on rational arithmetic in such unfavorable situations  but the underlying computational cost would be rapidly prohibitive.} 
\end{example}

\revise{
\begin{remark}
In practice, LP relaxations \eqref{eq:primal.problem.0.even.LP} with the value $\tau_{k}^\textup{P\'olya}$ may exhibit a much slower convergence rate towards $f^\star$ compared to SDP \eqref{eq:primal.problem.0.even} with the value ${\rho_{k,s}^\textup{P\'olya}}$ (as shown in Example \eqref{exam:AMGM}). Therefore, for $\tau_{k}^\textup{P\'olya}$ to serve as a good approximation of $f^\star$, the relaxation order $k$ must be set to a large value.
\revise{As a result, the matrix that defines the convex polytope of the LP relaxation \eqref{eq:primal.problem.0.even.LP} becomes very large and ill-conditioned, 
which in turn leads to less accurate computations of $\tau_{k}^\textup{P\'olya}$. }
Unlike the scenario described above, SDP \eqref{eq:primal.problem.0.even} with a small and appropriately chosen relaxation order $k$ and a factor width bound $s$, has a moderate size and an optimal value ${\rho_{k,s}^\textup{P\'olya}}$ that closely approximates $f^\star$. 
Consequently, solvers can easily provide accurate approximations for ${\rho_{k,s}^\textup{P\'olya}}$ that are close to $f^\star$.
\end{remark}
}
\revise{
\begin{remark}
If we don't use even symmetry, our original idea provides a block-diagonal structure $(z_{\a+\b})_{(\a,\b\in\mathcal A_j)}$, $j\in[r]$, with $\bigcup_{j\in[r]}\mathcal A_j=\N^n_d$ for each matrix variable $(z_{\a+\b})_{\a,\b\in\N^n_d}$ in the following semidefinite relaxations based on Putinar--Vasilescu's Positivstellensatz for POP \eqref{eq:equi.POP}:
\begin{equation}
\begin{array}{rl}
\inf\limits_\y &{L_\y}( {\theta ^k}\check f  )\\
\st&\y= {(y_\a )_{\a  \in \N^n_{2( {d_f + k})}}} \subset \R\,,\,{L_\y}( {{\theta ^k}}) = 1\,,\\
&\M_{k_i}( {{\check g_i}\y})\succeq 0\,,\,i \in[m]\,,
\end{array}
\end{equation}
However, this selection is not optimal as it does not make use of \revise{the even property of input polynomials $\check f, \check g_i$.
By using even symmetry, our method provides} more effective block-diagonal structures since all variables without even indices are removed.
\end{remark}
}

\revise{
In the following example, we \revise{compare our method with the} hierarchy based on Putinar's Positivstellensatz using even/sign symmetry:
}
\begin{example}
\revise{
Consider the case where $f=-\|x-a\|_2^2$ with $a=(\frac{1}{n},\dots,\frac{1}{n})$, $S=\{x\in\R^n:x_j\ge 0\,,\,\sum_{j=1}^n\le 1\}$. 
Then the optimal value of POP \eqref{eq:constrained.problem.poly.even.LP} is $f^\star=-\frac{n-1}{n}$.
We solve SDP relaxations for the equivalent POP \eqref{eq:equi.POP} by using our method \eqref{eq:primal.problem.0.even} and the standard method based on Putinar's Positivstellensatz that exploits even symmetry. To run the latter in practice, we rely on the TSSOS software library \revise{and
the corresponding  numerical results are reported in Table \ref{tab:comparison.evensymmetry}.}
}
\begin{table}

    \caption{\small A numerical comparison of our method based on the extension of P\'olya's Positivstellensatz, factor width and even symmetry  with the standard method based on Putinar's Positivstellensatz and even symmetry}
    \label{tab:comparison.evensymmetry}
\small
\begin{center}
\begin{tabular}{|c|c|c|c|c|c|c|c|c|c|c|c|c|c|c|c|c|c|c|c|}
        \hline
\multirow{2}{*}{Id}  &
\multicolumn{1}{c|}{POP size}&
\multicolumn{3}{c|}{Putinar $+$ even symmetry}&      \multicolumn{4}{c|}{P\'olya $+$ factor width $+$  even symmetry}\\ \cline{2-9}
&
{$n$}&
{$k$} & {val} & {time} & {$k$} & {$s$}& {val} & {time}\\
\hline
1 & 100 & 2&-0.989999  & 41 & 0 & 5 & -0.990000 & 12\\
2 & 200 & 2&-0.995000  & 911 & 0 & 5 & -0.995000 & 291\\
\hline
\end{tabular}
\begin{tabular}{|c|c|c|c|c|c|c|c|c|c|c|c|c|c|c|c|c|}
        \hline
\multirow{2}{*}{Id}&
\multicolumn{4}{c|}{Putinar $+$ even symmetry}&      \multicolumn{4}{c|}{P\'olya $+$ factor width $+$  even symmetry}\\ \cline{2-9}
&
nmat & msize & nscal & naff &nmat & msize & nscal & naff\\
\hline
1 & 1 & 101 & 5152 & 5151 & 97 & 5 &5152 &5151\\
2 & 1 & 201 & 20302 & 20301 & 197 & 5 &20302 &20301\\
\hline
\end{tabular}
\end{center}
\end{table}
\revise{Overall \revise{they suggest} that our method provides SDP relaxations with smaller block sizes, which can be solved more efficiently than the standard method exploiting even symmetry.}
\end{example}

\subsection{Obtaining an optimal solution} 
\label{sec:obtain.sol.dense}
A real sequence $( y_\a)_{\a  \in \N^n_t}$ has a representing measure if there exists a finite Borel measure $\mu$ such that $y_\a  = \int_{\R^n} {\x^\a d\mu(\x)}$ is satisfied for every $\a  \in {\N^n_t}$.

Next, we discuss about the extraction of an optimal solution $\x^\star$ of  POP \eqref{eq:constrained.problem.poly.even.LP} from the optimal solution $\y=(y_\a)_{\a\in\N^n_{2(d_f+k)}}$ of the semidefinite relaxations \eqref{eq:dual-sdp.0.even}.
\begin{remark}
 A naive idea is to define the new sequence of moments $\mathbf{u}=(u_\a)_{\a\in\N^n_{2(d_f+k)}}$ given by $u_\a:=y_{\a}^2$, for $\a\in\N^n_{2(d_f+k)}$.
Obviously, if $\y$ has a representing Dirac measure $\delta_{\z^\star}$, then $\mathbf{u}$ has a representing Dirac measure $\delta_{\z^{\star2}}$.
In this case, we take $\x^\star:=\z^{\star2}$.
However, there is no guarantee that $\mathbf{u}$ has a representing measure in general even if  $\y$ has one.
\end{remark}

Based on \cite[Remark 3.4]{mai2023hierarchy}, we  use  the heuristic  extraction  algorithm presented in Algorithm \ref{alg:extract.sol}.
\begin{algorithm}
\caption{Extraction algorithm for POPs on the nonnegative orthant}
\label{alg:extract.sol}
\textbf{Input:} precision parameter $\varepsilon>0$ and an optimal solution $(\lambda,\mathbf{G}_{ij})$ of SDP \eqref{eq:primal.problem.0.even}.\\
\textbf{Output:}  an optimal solution $\x^\star$ of POP \eqref{eq:constrained.problem.poly.even.LP}.
\begin{algorithmic}[1]
    \State For $j\in[b(n,k_m)]$, let $\bar{\mathbf{G}}_j=(w^{(j)}_{\mathbf p\mathbf q})_{\mathbf p,\mathbf q\in \N^n_{k_m}}$ such that $(w^{(j)}_{\mathbf p\mathbf q})_{\mathbf p,\mathbf q\in \cA^{(s,k_m)}_j}=\mathbf{G}_{j}$ and $w^{(j)}_{\mathbf p\mathbf q}=0$ if $(\mathbf p,\mathbf q)\notin (\cA^{(s,k_m)}_j)^2$.
    Then $\bar{\mathbf{G}}_j\succeq 0$ and
    \begin{equation}
    \begin{array}{r}
         \v_{\N^n_{k_m}}^\top \bar{\mathbf{G}}_j \v_{\N^n_{k_m}}=\v_{\cA^{(s,k_m)}_j}^\top \mathbf{G}_{j} \v_{\cA^{(s,k_m)}_j}\,;
    \end{array}
    \end{equation}
    \State Let $\mathbf{G}:=\sum_{j\in[b(n,k_m)]}\bar{\mathbf{G}}_j$. Then $\mathbf{G}$ is the Gram matrix corresponding to $\sigma_m$ in the SOS decomposition
    \begin{equation}
    \begin{array}{rl}
         \theta^k(\check f-\lambda)=\sum_{i\in[m]} \check g_i \sigma_i\,,
    \end{array}
    \end{equation}
    where $\sigma_i$ are SOS polynomials and $\check g_m=1$;
    \State Obtain an atom $ \z^\star\in\R^n$ by using the extraction algorithm of Henrion and Lasserre in  \cite{henrion2005detecting}, where the matrix $\mathbf V$ in \cite[(6)]{henrion2005detecting} is taken such that the columns of $\mathbf V$ form a basis of the null space $\{ \mathbf{u}\in\R^{\omega_k}\,:\, \mathbf{G} \mathbf{u}=0\}$;   
    \State Verify that $ \z^\star$ is  an approximate optimal solution of POP \eqref{eq:equi.POP} by checking the following inequalities:
    \begin{equation}\label{eq:check.sol}
        |\check f( \z^\star)-\lambda|\le \varepsilon \|\check f\|_{\max}\text{ and }\check g_i( \z^\star)\ge -\varepsilon \|\check g_i\|_{\max}\,,\, i\in [m]\,,
    \end{equation}
where  $\|q\|_{\max}:=\max_\a |q_\a|$ for any $q\in\R[\x]$.
\State If the inequalities \eqref{eq:check.sol} hold, set $\x^\star:=\z^{\star 2}$.
\end{algorithmic}
\end{algorithm}


\section{Numerical experiments}
\label{sec:benchmark.polya}
In this section we report results of numerical experiments obtained by solving the Moment-SOS relaxations of some random and nonrandom instances of POP \eqref{eq:constrained.problem.poly.even}. 
Notice that our relaxations from Section \ref{sec:compact.pop} are to deal with dense POPs while the ones from Section \ref{sec:sparse.POP.nonneg} are for POPs with correlative sparsity.

For numerical comparison purposes, recall the \revise{semidefinite} relaxation based on Putinar's Positivstellensatz for solving POP \eqref{eq:constrained.problem.poly.even} indexed by $k\in\N$:
\begin{equation}\label{eq:sdp.Pu}
\begin{array}{rl}
\tau_k^{\textup{\revise{Putinar}}}: = \inf\limits_\y &{L_\y}(  f  )\\
\st&\y= {(y_\a )_{\a  \in \N^n_{2k}}} \subset \R\,,\,y_\mathbf{0} = 1\,,\\[5pt]
&\M_{k-\lceil g_i \rceil}( g_i\y)\succeq 0\,,\,i \in[\bar m]\,.
\end{array}
\end{equation}
Here $\bar m:=m+n$, $g_{m+j}:=x_j$, $j\in [n]$, and 
$g_m:=1$. 
As shown by Baldi and Mourrain \cite{baldi2023effective}, the sequence $(\tau_k^{\textup{\revise{Putinar}}})_{k\in\N}$ converges to $f^\star$ with the rate of at least $\mathcal{O}(\varepsilon^{-\mathfrak{c}})$ when POP \eqref{eq:constrained.problem.poly.even} has a ball constraint, e.g., $g_1:=R-\|\x\|_2^2$ for some $R>0$.
If $g_1=R-\sum_{j\in[n]}x_j$ for some $R>0$, then $(\tau_k^{\textup{\revise{Putinar}}})_{k\in\N}$ still converges to $f^\star$ due to Jacobi-Prestel \cite[Theorem 4.2]{jacobi2001distinguished} (see also \cite[Theorem 1 (JP)]{averkov2013constructive}).
\begin{remark}
If we assume that $g_1:=R-\sum_{j\in[n]}x_j$ for some $R>0$, SDP \eqref{eq:sdp.Pu}  may be unbounded when $k$ is too small since its variable $\y$ is possibly unbounded.
This issue occurs later on, see, e.g., Section \ref{sec:densePOP.exp}.
However, if we assume that $g_1:=R-\|\x\|_2^2$ for some $R>0$, then SDP \eqref{eq:sdp.Pu}  is feasible for any order $k\ge 1$ (see Section \ref{sec:PMSV}).
\end{remark}

Recall the \revise{semidefinite} relaxation based on the sparse version of Putinar's Positivstellensatz for solving POP \eqref{eq:constrained.problem.poly.even} (under Assumption \ref{ass:sparsePOP}) indexed by $k\in\N$:
\begin{equation}\label{eq:sdp.Pu.sp}
\begin{array}{rl}
\tau_k^{\textup{\revise{SparsePutinar}}}: = \inf\limits_\y &{L_\y}(  f  )\\
\st&\y= {(y_\a )_{\a  \in \N^n_{2k}}} \subset \R\,,\,y_\mathbf{0} = 1\,,\\[5pt]
&\M_{\N^{I_c}_{k-\lceil g_i \rceil}}( g_i\y)\succeq 0\,,\,i \in J_c\,,\,c\in[p]\,.
\end{array}
\end{equation}
Here $I_c\subset[n]$, $\bar m:=m+n$, $g_{m+j}:=x_j$, $j\in [n]$, 
$g_m:=1$ and $m\in J_c\subset [\bar m]$, for $c\in[p]$.

The experiments are performed in Julia 1.3.1.
\revise{We use {\tt JuMP} \cite{lubin2023jump} to model our semidefinite relaxations.}
 We rely on  {\tt TSSOS} \cite{wang2021tssos}  to   solve the Moment-SOS relaxations of sparse POPs.

The implementation of our method is available online via the link:
\begin{center}
    \href{https://github.com/maihoanganh/InterRelax}{{\bf https://github.com/maihoanganh/InterRelax}}.
\end{center}

We use a desktop computer with an Intel(R) Core(TM) i7-8665U CPU @ 1.9GHz $\times$ 8 and 31.2 GB of RAM. 
The notation for the numerical results is given in Table \ref{tab:nontation.ineq}.
\begin{table}
    \caption{\small The notation}
    \label{tab:nontation.ineq}
\small
\begin{center}
\begin{tabular}{|m{3.2cm}|m{8.1cm}|}
\hline 
Pb &the ordinal number of a POP instance\\
\hline 
Id &the ordinal number of an SDP instance\\
\hline
$n$ & the number of nonnegative variables in POP \eqref{eq:constrained.problem.poly.even}\\
\hline
$m_{\text{ineq}}$& the number of inequality constraints of the form $g_i\ge 0$ in POP \eqref{eq:constrained.problem.poly.even}\\
\hline
$m_{\text{eq}}$& the number of equality constraints of the form $g_i=0$ in POP \eqref{eq:constrained.problem.poly.even}\\
\hline
\revise{Putinar SDP \eqref{eq:sdp.Pu}} & the SDP relaxation based on Putinar's  Positivstellensatz \eqref{eq:sdp.Pu} modeled by {\tt TSSOS} and solved by {\tt Mosek} 9.1\\
\hline
\revise{P\'olya SDP \eqref{eq:primal.problem.0.even}} & the SDP relaxation based on the extension of P\'olya's  Positivstellensatz \eqref{eq:primal.problem.0.even} modeled by  our software {\tt InterRelax} and solved by {\tt Mosek} 9.1 \\
\hline
\revise{Sparse Putinar SDP \eqref{eq:sdp.Pu.sp}}  & the SDP relaxation for a sparse POP based on Putinar's  Positivstellensatz \eqref{eq:sdp.Pu.sp} modeled by  {\tt TSSOS} and solved by {\tt Mosek} 9.1 \\
\hline
\revise{Sparse P\'olya SDP \eqref{eq:primal.problem.0.even.sparse}}  & the SDP relaxation for a sparse POP based on the extension of P\'olya's  Positivstellensatz \eqref{eq:primal.problem.0.even.sparse} modeled by  our software {\tt InterRelax} and solved by {\tt Mosek} 9.1\\
\hline
$k$&the relaxation order\\
\hline
$s$&the factor width upper bound used in SDP  
\\ 
\hline
$d$& the sparsity order of the SDP relaxation \eqref{eq:primal.problem.0.even.sparse}\\
\hline
nmat&the number of matrix variables of an SDP\\
\hline
msize& the largest size of matrix variables  of an SDP\\
\hline
nscal&the number of scalar variables of an SDP\\
\hline
naff&the number of affine constraints of an SDP\\
\hline
val& the value returned by the SDP relaxation\\
\hline
$^*$& there exists at least one optimal solution of the POP, which can be extracted by Algorithm \ref{alg:extract.sol} \\
\hline
time & the running time in seconds (including  modeling and solving time)\\
\hline
$\infty$ & the SDP relaxation is unbounded or infeasible\\
\hline 
$-$& the calculation runs out of space
\\
\hline
\end{tabular}
\end{center}
\end{table}

\revise{We report numerical values when the solver {\tt Mosek} indicates optimality (meaning KKT conditions are satisfied) or when it reports primal feasibility ({\tt Mosek} displays {\tt Primal status = FEASIBLE\_POINT}). In the latter case, the primal SDP, which is the SOS relaxation, yields a value at a feasible point. This value still serves as a lower bound for the optimal value of the original POP, as we maximize it within the SOS relaxation.}

\subsection{Dense QCQPs}\label{sec:densePOP.exp}
\paragraph{Test problems:} We construct randomly generated dense quadratically constrained quadratic programs (QCQPs) in the form \eqref{eq:constrained.problem.poly.even}-\eqref{eq:semial.set.def.2.even} as follows:
\begin{enumerate}
    \item Take $\mathbf{a}$ in the simplex
    \begin{equation}
    \begin{array}{l}
         \Delta_n:=\{\x\in\R^n\,:\,x_j\ge 0\,,\,j\in[n]\,,\,\sum_{j\in[n]}x_j\le 1\}
    \end{array}
    \end{equation}
    w.r.t. the uniform distribution.
    \item Let $g_1:=1-\sum_{j\in[n]}x_j$ and $g_2:=1$.
    \item Take every coefficient of $f$ and $g_i$, $i=2,\dots,m$, in $(-1,1)$ w.r.t. the uniform distribution.
    \item Update $g_i(\x):=g_i(\x)-g_i(\mathbf{a})+0.125$, for $i=2,\dots,m_{\text{ineq}}$.
    \item Update $g_{i+m_{\text{ineq}}}(\x):=g_{i+m_{\text{ineq}}}(\x)-g_{i+m_{\text{ineq}}}(\mathbf{a})$ and set $g_{i+m_{\text{eq}}+m_{\text{ineq}}}=-g_{i+m_{\text{ineq}}}$, for $i\in[m_{\text{eq}}]$.
\end{enumerate}
Here $m=m_{\text{ineq}}+2 m_{\text{eq}}$ with $m_{\text{ineq}}$ (resp. $m_{\text{eq}}$) being the number of inequality (resp. equality) constraints except the nonnegative constraints $x_j\ge 0$.
If $m_{\text{ineq}}=2$ and $m_{\text{eq}}=0$, we obtain the case of the minimization of a polynomial on the simplex $\Delta_n$.
The point $a$ is a feasible solution of POP \eqref{eq:constrained.problem.poly.even}.

The numerical results are displayed in Table \ref{tab:QCQP.on.unit.ball.nonneg}.

\begin{table}
    \caption{\small Numerical results for randomly generated dense QCQPs.}
    \label{tab:QCQP.on.unit.ball.nonneg}
\small
\begin{center}
\begin{tabular}{|c|c|c|c|c|c|c|c|c|c|c|c|c|c|c|c|c|c|c|}
        \hline
\multirow{2}{*}{Id} &\multirow{2}{*}{Pb}    &
\multicolumn{3}{c|}{POP size}&
\multicolumn{3}{c|}{\revise{Putinar SDP \eqref{eq:sdp.Pu}}}&      \multicolumn{4}{c|}{\revise{P\'olya SDP \eqref{eq:primal.problem.0.even}}}\\ \cline{3-12}
&&{$n$}&{$m_{\text{ineq}}$}& {$m_{\text{eq}}$}& {$k$} & {val} & {time} &{$k$} & {$s$}& {val} & {time}\\
\hline
\hline
1 &\multirow{2}{*}{1}& \multirow{2}{*}{20} & \multirow{2}{*}{2} & \multirow{2}{*}{0} & 1 & $\infty$  & 0.0&  \multirow{2}{*}{0} & \multirow{2}{*}{17} & \multirow{2}{*}{\bf -1.99792$^*$} & \multirow{2}{*}{1}\\
2& &  &  & &  2 & {\bf -1.99792} & 92&   & &  & \\
\hline
\hline
3 & \multirow{2}{*}{2}& \multirow{2}{*}{20} & \multirow{2}{*}{5} & \multirow{2}{*}{0} & 1 & $\infty$ & 0.03&  \multirow{2}{*}{1} & \multirow{2}{*}{20} & \multirow{2}{*}{\bf -0.265883$^*$} & \multirow{2}{*}{9} \\
4 & &  &  & &  2 & -0.350601 & 342 &  &  & & \\
\hline
\hline
5 & \multirow{2}{*}{3}& \multirow{2}{*}{20} & \multirow{2}{*}{5} & \multirow{2}{*}{4} & 1 & $\infty$ & 0.02&  \multirow{2}{*}{1} & \multirow{2}{*}{7} & \multirow{2}{*}{ -0.429442} & \multirow{2}{*}{5}\\
6 & & &  & &  2 & -0.431543 & 356&  & & & \\
\hline
\hline
7 &\multirow{2}{*}{4} &\multirow{2}{*}{30} & \multirow{2}{*}{2} & \multirow{2}{*}{0} & 1 & $\infty$ & 0.0&  \multirow{2}{*}{0} & \multirow{2}{*}{20} & \multirow{2}{*}{\bf -2.31695$^*$} & \multirow{2}{*}{2} \\
8 & & &  & &  2 & {\bf -2.31695} & 3545& &  & &  \\
\hline
\hline
9 & \multirow{2}{*}{5}&\multirow{2}{*}{30} & \multirow{2}{*}{7} & \multirow{2}{*}{0} & 1 & $\infty$ & 0.2& \multirow{2}{*}{0} & \multirow{2}{*}{31} & \multirow{2}{*}{\bf -1.79295} & \multirow{2}{*}{45}\\
10 & & &  & &  2 & -2.13423 & 15135 & &  &  &\\
\hline
\hline
11 & \multirow{2}{*}{6}&\multirow{2}{*}{30} & \multirow{2}{*}{7} & \multirow{2}{*}{6} & 1 & $\infty$ & 0.1&  \multirow{2}{*}{1} & \multirow{2}{*}{31} & \multirow{2}{*}{\bf -1.56374} & \multirow{2}{*}{54}  \\
12 & & &  & &  2 & {\bf -1.56374} & 12480&   &  & &   \\
\hline
\end{tabular}

\begin{tabular}{|c|c|c|c|c|c|c|c|c|c|c|c|c|c|c|c|c|c|c|c|c|}
        \hline
\multirow{2}{*}{Id}&
\multicolumn{4}{c|}{\revise{Putinar SDP \eqref{eq:sdp.Pu}}}&      \multicolumn{4}{c|}{\revise{P\'olya SDP \eqref{eq:primal.problem.0.even}}}\\ \cline{2-9}
&
nmat & msize & nscal & naff &nmat & msize & nscal & naff \\
\hline
\hline
1 & 1 & 21 & 22 & 231 & \multirow{2}{*}{5} & \multirow{2}{*}{17} &\multirow{2}{*}{232} &\multirow{2}{*}{231}\\
2 & 22 & 231 & 1 & 10626 &  &  & &\\
\hline
\hline
3 & 1 & 21 & 25 & 231 & \multirow{2}{*}{44} & \multirow{2}{*}{20} &\multirow{2}{*}{1604} &\multirow{2}{*}{1771}\\
4 & 25 & 231 & 1 & 10626 & & &&\\
\hline
\hline
5 & 1 & 21 & 29 & 231 & \multirow{2}{*}{330} & \multirow{2}{*}{7} &\multirow{2}{*}{1688}  &\multirow{2}{*}{1771} \\
6 & 25 & 231 & 925 & 10626 &  & &  &\\
\hline
\hline
7 & 1 & 31 & 32 & 496 & \multirow{2}{*}{11} & \multirow{2}{*}{21} &\multirow{2}{*}{497}  &\multirow{2}{*}{496} \\
8 & 32 & 496 & 1 & 46376 & &  & &\\
\hline
\hline
9 & 1 & 31 & 37 & 496 & \multirow{2}{*}{32} & \multirow{2}{*}{31} &\multirow{2}{*}{5116}  &\multirow{2}{*}{5456} \\
10 & 37 & 496 & 1 & 46376 &  & & &\\
\hline
\hline
11 & 1 & 31 & 43 & 496 & \multirow{2}{*}{32} & \multirow{2}{*}{31} &\multirow{2}{*}{5302}  &\multirow{2}{*}{5456} \\
12 & 37 & 496 & 2977 & 46376 & &  &  & \\
\hline
\end{tabular}
\end{center}
\end{table}
\paragraph{Discussion:}
Table \ref{tab:QCQP.on.unit.ball.nonneg} shows that
\revise{P\'olya SDP \eqref{eq:primal.problem.0.even}} is typically faster and more accurate than \revise{Putinar SDP \eqref{eq:sdp.Pu}}.
For instance, when $n=20$, $m_{\text{ineq}}=5$ and $m_{\text{eq}}=0$, \revise{Putinar SDP \eqref{eq:sdp.Pu}} takes $342$ seconds to return the lower bound $-0.350601$ for $f^\star$, while \revise{P\'olya SDP \eqref{eq:primal.problem.0.even}} only takes $9$ seconds to return the better \revise{lower} bound $-0.265883$ and an approximate optimal solution. \revise{This is because} \revise{P\'olya SDP \eqref{eq:primal.problem.0.even}} has $44$ matrix variables \revise{of maximal} matrix size $20$, while \revise{Putinar SDP \eqref{eq:sdp.Pu}} has $25$ matrix variables \revise{with maximal matrix size $231$} in this case. 

\subsection{Stability number of a graph}
\label{sec:stability.known.graph}
In order to compute the stability number $\alpha(G)$ of a given graph $G$, we solve the following POP on the unit simplex:
\begin{equation}
\label{eq:sta.num.POP}
\begin{array}{rl}
     \frac{1}{\alpha(G)}=\min\limits_{\x\in\R_+^n}\{\x^\top (\mathbf A+\mathbf I) \x\,:\,\sum_{j\in[n]}x_j=1\}\,,
\end{array}
\end{equation}
where $\mathbf A$ is the adjacency matrix of $G$ and $\mathbf I$ is the identity matrix.
\paragraph{Test problems:}  
We take some adjacency matrices of known graphs from \cite{nr}.
The numerical results are displayed in Tables \ref{tab:SN.known} and \ref{tab:SN.known.unitball}.
Note that in Table \ref{tab:SN.known.unitball}, we solve POP \eqref{eq:sta.num.POP} with an additional unit ball constraint $1-\|\x\|_2^2\ge 0$. 
The columns under ``val" show the approximations of $\alpha(G)$.
\begin{table}
    \caption{\small Numerical results for stability number of some known graphs in \cite{nr}.}
    \label{tab:SN.known}
\small
\begin{center}
\begin{tabular}{|c|c|c|c|c|c|c|c|c|c|c|c|c|c|c|c|c|c|c|c|c|c|c|c|}
        \hline
\multirow{2}{*}{Id} &\multirow{2}{*}{Pb}    &
\multicolumn{1}{c|}{POP size}&
\multicolumn{3}{c|}{\revise{Putinar SDP \eqref{eq:sdp.Pu}}}&      \multicolumn{4}{c|}{\revise{P\'olya SDP \eqref{eq:primal.problem.0.even}}}\\ \cline{3-10}
&
&{$n$}&
{$k$} & {val} & {time} & {$k$} & {$s$}& {val} & {time}\\
\hline
\hline
1 &\multirow{2}{*}{GD02\_a}& \multirow{2}{*}{23} &  1 &$\infty$  & 0.02 & \multirow{2}{*}{0} & \multirow{2}{*}{25} & \multirow{2}{*}{\bf 13.0000} & \multirow{2}{*}{1} \\
2&& &  2 &  13.0110 & 394&  &  &  & \\
\hline
\hline
3 &\multirow{2}{*}{johnson8-2-4}& \multirow{2}{*}{28} &  1 &$\infty$  & 0.03 & \multirow{2}{*}{0} & \multirow{2}{*}{30} & \multirow{2}{*}{\bf 7.00000} & \multirow{2}{*}{1}\\
4& &  &  2 &  {\bf 7.00000} & 2098&  & & &\\
\hline
\hline
5 &\multirow{2}{*}{johnson8-4-4}& \multirow{2}{*}{70} &  1 & $\infty$ & 1& \multirow{2}{*}{0} & \multirow{2}{*}{72} & \multirow{2}{*}{\bf 5.00000} & \multirow{2}{*}{5}\\
6& &  &  2 & $-$ & $-$& & & & \\
\hline
\hline
7 &\multirow{2}{*}{hamming6-2}& \multirow{2}{*}{64} &  1 & $\infty$  & 0.5& \multirow{2}{*}{0} & \multirow{2}{*}{66} & \multirow{2}{*}{\bf 1.99999} & \multirow{2}{*}{3}\\
8& & &  2 &  $-$ & $-$& & & & \\
\hline
\hline
9 &\multirow{2}{*}{hamming6-4}& \multirow{2}{*}{64} &  1 & $\infty$ & 0.6& \multirow{2}{*}{0} & \multirow{2}{*}{66} & \multirow{2}{*}{\bf 12.0000} & \multirow{2}{*}{3}\\
10& &  &  2 &  $-$ & $-$& & & &  \\
\hline
\hline
11 &\multirow{2}{*}{johnson16-2-4}& \multirow{2}{*}{120} &  1 & $\infty$& 0.6& \multirow{2}{*}{0} & \multirow{2}{*}{122} & \multirow{2}{*}{\bf 15.0001} & \multirow{2}{*}{54}\\
12&& &  2 &  $-$ & $-$&  &  &  & \\
\hline
\end{tabular}

\begin{tabular}{|c|c|c|c|c|c|c|c|c|c|c|c|c|c|c|c|c|}
        \hline
\multirow{2}{*}{Id}&
\multicolumn{4}{c|}{\revise{Putinar SDP \eqref{eq:sdp.Pu}}}&      \multicolumn{4}{c|}{\revise{P\'olya SDP \eqref{eq:primal.problem.0.even}}}\\ \cline{2-9}
&
nmat & msize & nscal & naff &nmat & msize & nscal & naff\\
\hline
\hline
1 & 1 & 24 & 25 & 300 & \multirow{2}{*}{1} & \multirow{2}{*}{24} & \multirow{2}{*}{301} & \multirow{2}{*}{300}\\
2 & 24 & 300 & 301 & 17550 &  &  & &\\
\hline
\hline
3 & 1 & 29 & 30 & 435 & \multirow{2}{*}{1} & \multirow{2}{*}{29} &\multirow{2}{*}{436} &\multirow{2}{*}{435} \\
4 & 29 & 435 & 436 & 35960 &  & & & \\
\hline
\hline
5 & 1 & 71 & 72 & 2556 & \multirow{2}{*}{1} & \multirow{2}{*}{71} &\multirow{2}{*}{2557} &\multirow{2}{*}{2556}\\
6 & 71 & 2556 & 2557 & 1150626 &  &  & &\\
\hline
\hline
7 & 1 & 65 & 66 & 2145 & \multirow{2}{*}{1} & \multirow{2}{*}{65} &\multirow{2}{*}{2146} &\multirow{2}{*}{2145}\\
8 & 65 & 2145&2146 & 814385 & & & &\\
\hline
\hline
9 & 1 & 65 & 66 & 2145 & \multirow{2}{*}{1} & \multirow{2}{*}{65} &\multirow{2}{*}{2146} &\multirow{2}{*}{2145}\\
10 & 65 & 2145&2146 & 814385 &  &  & & \\
\hline
\hline
11 & 1 & 121 & 122 & 7381 & \multirow{2}{*}{1} & \multirow{2}{*}{121} &\multirow{2}{*}{7382} &\multirow{2}{*}{7381}\\
12 & 121 & 7381&7380 & 9381251 &  & & & \\
\hline
\end{tabular}
\end{center}
\end{table}
\begin{table}
    \caption{\small Numerical results for stability number of some known graphs in \cite{nr} with an additional unit ball constraint.}
    \label{tab:SN.known.unitball}
\small
\begin{center}
\begin{tabular}{|c|c|c|c|c|c|c|c|c|c|c|c|c|c|c|c|c|c|c|c|c|c|c|c|}
        \hline
\multirow{2}{*}{Id} &\multirow{2}{*}{Pb}    &
\multicolumn{1}{c|}{POP size}&
\multicolumn{3}{c|}{\revise{Putinar SDP \eqref{eq:sdp.Pu}}}&      \multicolumn{4}{c|}{\revise{P\'olya SDP \eqref{eq:primal.problem.0.even}}}\\ \cline{3-10}
&
&{$n$}&
{$k$} & {val} & {time} & {$k$} & {$s$}& {val} & {time}\\
\hline
\hline
1 &\multirow{2}{*}{GD02\_a}& \multirow{2}{*}{23} &  1 &-0.62896 & 0.02 & \multirow{2}{*}{0} & \multirow{2}{*}{13} & \multirow{2}{*}{\bf 13.0000}& \multirow{2}{*}{1} \\
2&& &  2 &  13.0170 & 442&  &  &  & \\
\hline
\hline
3 &\multirow{2}{*}{johnson8-2-4}& \multirow{2}{*}{28} &  1 &-0.30434  & 0.03 & \multirow{2}{*}{0} & \multirow{2}{*}{23} & \multirow{2}{*}{\bf 7.00000}& \multirow{2}{*}{1}\\
4& &  &  2 &  {\bf 7.00000} & 3010&  &  &  &\\
\hline
\hline
5 &\multirow{2}{*}{johnson8-4-4}& \multirow{2}{*}{70} &  1 & -0.14056 & 1& \multirow{2}{*}{0} & \multirow{2}{*}{70} &  \multirow{2}{*}{\bf 5.00000} & \multirow{2}{*}{10}\\
6& & &  2 & $-$ & $-$&  &  &  &\\
\hline
\hline
7 &\multirow{2}{*}{hamming6-2}& \multirow{2}{*}{64} &  1 & -0.32989  & 1& \multirow{2}{*}{0} & \multirow{2}{*}{64} & \multirow{2}{*}{\bf 2.00000} & \multirow{2}{*}{7} \\
8& &  &  2 &  $-$ & $-$&  & &  &\\
\hline
\hline
9 &\multirow{2}{*}{hamming6-4}& \multirow{2}{*}{64} &  1 & -0.11764 & 0.6& \multirow{2}{*}{0} & \multirow{2}{*}{64} & \multirow{2}{*}{\bf 12.0000} & \multirow{2}{*}{6}\\
10& & &  2 &  $-$ & $-$&  & &  & \\
\hline
\hline
11 &\multirow{2}{*}{johnson16-2-4}& \multirow{2}{*}{120} &  1 & -0.08982& 26& \multirow{2}{*}{0} & \multirow{2}{*}{121} & \multirow{2}{*}{\bf 15.0000} & \multirow{2}{*}{75}\\
12& & &  2 &  $-$ & $-$&  &  &  & \\
\hline
\end{tabular}

\begin{tabular}{|c|c|c|c|c|c|c|c|c|c|c|c|c|c|c|c|c|c|c|c|c|}
        \hline
\multirow{2}{*}{Id}&
\multicolumn{4}{c|}{\revise{Putinar SDP \eqref{eq:sdp.Pu}}}&      \multicolumn{4}{c|}{\revise{P\'olya SDP \eqref{eq:primal.problem.0.even}}}\\ \cline{2-9}
&
nmat & msize & nscal & naff &nmat & msize & nscal & naff\\
\hline
\hline
1 & 1 & 24 & 26 & 300 & \multirow{2}{*}{12} & \multirow{2}{*}{13} & \multirow{2}{*}{302} &\multirow{2}{*}{300}  \\
2 & 25 & 300 & 301 & 17550 &  &  & &\\
\hline
\hline
3 & 1 & 29 & 31 & 435 & \multirow{2}{*}{7} & \multirow{2}{*}{23} & \multirow{2}{*}{437} &\multirow{2}{*}{435} \\
4 & 30 & 435 & 436 & 35960 &  &  & &  \\
\hline
\hline
5 & 1 & 71 & 73 & 2556 & \multirow{2}{*}{2} & \multirow{2}{*}{70} & \multirow{2}{*}{2558} &\multirow{2}{*}{2556} \\
6 & 72 & 2556 & 2557 & 1150626 &  &  & & \\
\hline
\hline
7 & 1 & 65 & 67 & 2145 & \multirow{2}{*}{2} & \multirow{2}{*}{64} &\multirow{2}{*}{2146} &\multirow{2}{*}{2145}\\
8 & 66 & 2145&2146 & 814385 &  &  & &\\
\hline
\hline
9 & 1 & 65 & 67 & 2145 & \multirow{2}{*}{2} & \multirow{2}{*}{64} &\multirow{2}{*}{2146} &\multirow{2}{*}{2145}\\
10 & 66 & 2145&2146 & 814385 &  &  & &\\
\hline
\hline
11 & 1 & 121 & 123 & 7381 & \multirow{2}{*}{1} & \multirow{2}{*}{121} &\multirow{2}{*}{7383} &\multirow{2}{*}{7381}\\
12 & 122 & 7381&7380 & 9381251 & & & &\\
\hline
\end{tabular}
\end{center}
\end{table}
\paragraph{Discussion:}
The graphs from Table \ref{tab:SN.known} are relatively dense so that we cannot exploit term sparsity or correlative sparsity for POP \eqref{eq:sta.num.POP} in these cases.
For the graph GD02\_a in Table \ref{tab:SN.known}, \revise{P\'olya SDP \eqref{eq:primal.problem.0.even}} provides  better bounds for $\alpha(G)$ compared to the ones returned by the second order relaxations of \revise{Putinar SDP \eqref{eq:sdp.Pu}}.
In Table \ref{tab:SN.known.unitball}, \revise{Putinar SDP \eqref{eq:sdp.Pu}} provides negative values for the first order relaxations.
The additional unit ball constraint does not help to improve the bound for the second order relaxation for Id 2.

\subsection{The MAXCUT problems}

The MAXCUT problem is given by:
\begin{equation}
\label{eq:maxcut}
    \max_{\x\in\{0,1\}^n} \x^\top \mathbf W(\e-\x)\,,
\end{equation}
where $\e=(1,\dots,1)$  and $W$ is the matrix of edge weights associated with a graph (see \cite[Theorem 1]{commander2009maximum}).
\paragraph{Test problems:}
The data of graphs is taken from TSPLIB \cite{reinelt1991tsplib}.

The numerical results are displayed in Table \ref{tab:maxcut}.
Note that all instances of matrix $W$ are dense. 
\begin{table}
    \caption{\small Numerical results for some instances of MAXCUT problems.}
    \label{tab:maxcut}
\small
\begin{center}
\begin{tabular}{|c|c|c|c|c|c|c|c|c|c|c|c|c|c|c|c|c|c|c|c|c|c|c|c|}
        \hline
\multirow{2}{*}{Id} &\multirow{2}{*}{Pb}    &
\multicolumn{1}{c|}{POP size}&
\multicolumn{3}{c|}{\revise{Putinar SDP \eqref{eq:sdp.Pu}}}&      \multicolumn{4}{c|}{\revise{P\'olya SDP \eqref{eq:primal.problem.0.even}}}\\ \cline{3-10}
&
&{$n$}&
{$k$} & {val} & {time} & {$k$} & {$s$}& {val} & {time} \\
\hline
\hline
1 &\multirow{2}{*}{burma14}& \multirow{2}{*}{14} &  1 &30310.915  & 0.2& \multirow{2}{*}{1} & \multirow{2}{*}{16} & \multirow{2}{*}{\bf 30302.000} & \multirow{2}{*}{1} \\
2& &  &  2 & {\bf 30301.999} & 4&  &  &  & \\
\hline
\hline
3 &\multirow{2}{*}{gr17}& \multirow{2}{*}{17} &  1 & 25089.044 & 0.2 & \multirow{2}{*}{1} & \multirow{2}{*}{19} & \multirow{2}{*}{\bf 24986.000} & \multirow{2}{*}{1}\\
4& & &  2 &  {\bf 24985.999} & 24  &  &  & & \\
\hline
\hline
5 &\multirow{2}{*}{fri26}& \multirow{2}{*}{26} &  1 & 22220.657  & 0.4& \multirow{2}{*}{1} & \multirow{2}{*}{28} & \multirow{2}{*}{\bf 22218.000} & \multirow{2}{*}{12}\\
6&& &  2 &  {\bf 22217.999} & 1970&  & &  & \\
\hline
\hline
7 &\multirow{2}{*}{att48}& \multirow{2}{*}{48} &  1 & 799281.420  & 1& \multirow{2}{*}{1} & \multirow{2}{*}{50} & \multirow{2}{*}{\bf 798857.049} & \multirow{2}{*}{1129}\\
8&&  &  2 &  $-$ & $-$&  &  &  &\\
\hline
\end{tabular}

\begin{tabular}{|c|c|c|c|c|c|c|c|c|c|c|c|c|c|c|c|c|c|c|c|c|}
        \hline
\multirow{2}{*}{Id}&
\multicolumn{4}{c|}{\revise{Putinar SDP \eqref{eq:sdp.Pu}}}&      \multicolumn{4}{c|}{\revise{P\'olya SDP \eqref{eq:primal.problem.0.even}}}\\ \cline{2-9}
&
nmat & msize & nscal & naff &nmat & msize & nscal & naff\\
\hline
\hline
1 & 1&15 & 29 &130 & \multirow{2}{*}{15} & \multirow{2}{*}{15} & \multirow{2}{*}{666} &\multirow{2}{*}{680}\\
2 & 15 & 120 & 1681 & 3060&  &  &  &\\
\hline
\hline
3 & 1 & 18 & 35 & 171 & \multirow{2}{*}{18} & \multirow{2}{*}{18} &\multirow{2}{*}{1123} &\multirow{2}{*}{1140}\\
4 & 18 & 171 & 2908 & 5985 &  &  & & \\
\hline
\hline
5 & 1 & 27 & 53 & 378 & \multirow{2}{*}{27} & \multirow{2}{*}{27} &\multirow{2}{*}{3628} &\multirow{2}{*}{3654}\\
6 & 27 & 378 &9829 & 27405 & &  & &\\
\hline
\hline
7 & 1 & 49 & 97 & 1225 & \multirow{2}{*}{49} & \multirow{2}{*}{49} &\multirow{2}{*}{20777} &\multirow{2}{*}{20825}\\
8 & 30 & 465 &13486 & 40920 & &  & &\\
\hline
\end{tabular}
\end{center}
\end{table}
\paragraph{Discussion:}
The behavior of our method is similar to that in Section \ref{sec:densePOP.exp}.
\subsection{Positive maximal singular values}
\label{sec:PMSV}
\paragraph{Test problems:}  We generate a matrix $\M$ as in
\cite[(12)]{ebihara2021l2}. Explicitly, 
\begin{equation}
    \M:=\begin{bmatrix}
    \mathbf D & \mathbf 0 &\mathbf 0& \dots& \mathbf 0\\
    \mathbf C\mathbf B& \mathbf D &\mathbf 0&\dots& \mathbf 0\\
    \mathbf C\mathbf A\mathbf B& \mathbf C\mathbf B & \mathbf D& \dots& \mathbf 0\\
    \dots&\dots&\dots&\dots&\dots\\
    \mathbf C\mathbf A^{m-2}\mathbf B&\mathbf C\mathbf A^{m-3}\mathbf B&\mathbf C\mathbf A^{m-4}\mathbf B & \dots &\mathbf D
    \end{bmatrix}\,,
\end{equation}
where $\mathbf A,\mathbf B,\mathbf C,\mathbf D$ are square matrices of size $r=m$.
Every entry of $\mathbf A,\mathbf B,\mathbf C,\mathbf D$ is taken uniformly in $(-1,1)$.
In order to compute the positive maximal singular value $\sigma_+(\M)$ of $\M$, we solve the following POP on the nonnegative orthant:
\begin{equation}
\begin{array}{rl}
     \sigma_+(\M)^2=\max\limits_{\x\in\R^n_+}\{\x^\top (\M^\top \M) \x\,:\,\|\x\|_2^2=1\}\,.
\end{array}
\end{equation}
Note that $n=m\times r=m^2$.

The numerical results are displayed in Table \ref{tab:PMSV}.
\begin{table}
    \caption{\small Numerical results for positive maximal singular values.}
    \label{tab:PMSV}
\small
\begin{center}
\begin{tabular}{|c|c|c|c|c|c|c|c|c|c|c|c|c|c|c|c|c|c|c|c|}
        \hline
\multirow{2}{*}{Id} &\multirow{2}{*}{Pb}    &
\multicolumn{2}{c|}{POP size}&
\multicolumn{3}{c|}{\revise{Putinar SDP \eqref{eq:sdp.Pu}}}&      \multicolumn{4}{c|}{\revise{P\'olya SDP \eqref{eq:primal.problem.0.even}}}\\ \cline{3-11}
&&{$m$}
&{$n$}&
{$k$} & {val} & {time} & {$k$} & {$s$}& {val} & {time}\\
\hline
\hline
1 &\multirow{2}{*}{1}&\multirow{2}{*}{4}& \multirow{2}{*}{16} &  1 & 47.48110  & 0.02&  \multirow{2}{*}{0} & \multirow{2}{*}{17} & \multirow{2}{*}{\bf 30.18791} & \multirow{2}{*}{1}\\
2& & & &  2 & {\bf 30.18791} & 16&   &  &  & \\
\hline
\hline
3 &\multirow{2}{*}{2}&\multirow{2}{*}{5}& \multirow{2}{*}{25} &  1 & 168.4450  & 0.04& \multirow{2}{*}{0} & \multirow{2}{*}{26} & \multirow{2}{*}{\bf 91.28158} & \multirow{2}{*}{0.7} \\
4&&& &  2 & {\bf 91.28158} & 877 &  &  &  & \\
\hline
\hline
5 &\multirow{2}{*}{3}&\multirow{2}{*}{6}& \multirow{2}{*}{36} &  1 & 4759.12  & 0.2&  \multirow{2}{*}{0} & \multirow{2}{*}{37} & \multirow{2}{*}{\bf 2462.03} & \multirow{2}{*}{0.9}\\
6&&& &  2 & $-$ & $-$&  & & & \\
\hline
\hline
7 &\multirow{2}{*}{4}& \multirow{2}{*}{7}&\multirow{2}{*}{49} &  1 & 1777.53  & 0.5& \multirow{2}{*}{0} & \multirow{2}{*}{50} & \multirow{2}{*}{\bf 970.202} & \multirow{2}{*}{2}\\
8&& & &  2 & $-$ & $-$&  &  &  & \\
\hline
\end{tabular}

\begin{tabular}{|c|c|c|c|c|c|c|c|c|c|c|c|c|c|c|c|c|}
        \hline
\multirow{2}{*}{Id}&
\multicolumn{4}{c|}{\revise{Putinar SDP \eqref{eq:sdp.Pu}}}&      \multicolumn{4}{c|}{\revise{P\'olya SDP \eqref{eq:primal.problem.0.even}}}\\ \cline{2-9}
&
nmat & msize & nscal & naff &nmat & msize & nscal & naff\\
\hline
\hline
1 & 1 & 17 & 38 & 153 & \multirow{2}{*}{1} & \multirow{2}{*}{17} &\multirow{2}{*}{138} &\multirow{2}{*}{153}\\
2 & 17 & 153 & 154 & 4845 &  &  & &\\
\hline
\hline
3 & 1 & 26 & 27 & 351 & \multirow{2}{*}{1} & \multirow{2}{*}{26} &\multirow{2}{*}{327} &\multirow{2}{*}{351}\\
4 & 26 & 351 & 352 & 23751 &  &  & &\\
\hline
\hline
5 & 1 & 37 & 38 & 703 & \multirow{2}{*}{1} & \multirow{2}{*}{37} &\multirow{2}{*}{668} &\multirow{2}{*}{703}\\
6 & 37 & 703 &704 & 91390 &  &  & &\\
\hline
\hline
7 & 1 & 50 & 51 & 1275& \multirow{2}{*}{1} & \multirow{2}{*}{50} &\multirow{2}{*}{1227} &\multirow{2}{*}{1275}\\
8 & 50 & 1275 &1276 & 292825 &  &  & & \\
\hline
\end{tabular}
\end{center}
\end{table}
The columns of val show the approximations of $\sigma_+(\M)^2$.
\paragraph{Discussion:}
The behavior of our method is similar to that in Section \ref{sec:densePOP.exp}.

\subsection{Stability
number of a graph}
Let us consider POP \eqref{eq:sta.num.POP} which  returns the stability number of a graph $G$.
\paragraph{Test problems:}  We generate the adjacency matrix $\mathbf A=(a_{ij})_{j,j\in[n]}$ of the graph $G$ by the following steps: 
\begin{enumerate}
    \item Set $a_{ii}=0$, for $i\in[n]$.
    \item For $i\in[n]$, for $j\in\{1,\dots,i-1\}$, let us select $a_{ij}=a_{ji}$ uniformly $\{0,1\}$. 
\end{enumerate}
The numerical results are displayed in Table \ref{tab:SN}.

Note that the columns of val show the approximations of $\alpha(G)$.
\begin{table}
    \caption{\small Numerical results for stability number of randomly generated graphs.}
    \label{tab:SN}
\small
\begin{center}
\begin{tabular}{|c|c|c|c|c|c|c|c|c|c|c|c|c|c|c|c|c|c|c|c|}
        \hline
\multirow{2}{*}{Id} &\multirow{2}{*}{Pb}    &
\multicolumn{1}{c|}{POP size}&
\multicolumn{3}{c|}{\revise{Putinar SDP \eqref{eq:sdp.Pu}}}&      \multicolumn{4}{c|}{\revise{P\'olya SDP \eqref{eq:primal.problem.0.even}}}\\ \cline{3-10}
&
&{$n$}&
{$k$} & {val} & {time} & {$k$} & {$s$}& {val} & {time}\\
\hline
\hline
1 &\multirow{2}{*}{1}& \multirow{2}{*}{10} &  1 &$\infty$  & 0.01 &\multirow{2}{*}{0} & \multirow{2}{*}{11} & \multirow{2}{*}{\bf 3.00000} & \multirow{2}{*}{1}\\
2& &  &  2 &  3.02305 & 0.6&  &  &  & \\
\hline
\hline
3 &\multirow{2}{*}{2}& \multirow{2}{*}{15} &  1 & $\infty$ & 0.01&\multirow{2}{*}{0} & \multirow{2}{*}{16} & \multirow{2}{*}{\bf 5.00000} & \multirow{2}{*}{1}\\
4& & &  2 &  5.01898 & 10&  &  &  & \\
\hline
\hline
5 &\multirow{2}{*}{3}& \multirow{2}{*}{20} &  1 & $\infty$  & 0.02& \multirow{2}{*}{1} & \multirow{2}{*}{21} & \multirow{2}{*}{\bf 5.00001} & \multirow{2}{*}{4}\\
6& & &  2 &  5.02951 & 119&  &  &  & \\
\hline
\hline
7 &\multirow{2}{*}{4}& \multirow{2}{*}{25} &  1 & $\infty$  & 0.04& \multirow{2}{*}{1} & \multirow{2}{*}{26} & \multirow{2}{*}{\bf 6.00000} & \multirow{2}{*}{10}\\
8& & &  2 &  6.05801 & 1064&  & &  & \\
\hline
\end{tabular}

\begin{tabular}{|c|c|c|c|c|c|c|c|c|c|c|c|c|c|c|c|c|}
        \hline
\multirow{2}{*}{Id}&
\multicolumn{4}{c|}{\revise{Putinar SDP \eqref{eq:sdp.Pu}}}&      \multicolumn{4}{c|}{\revise{P\'olya SDP \eqref{eq:primal.problem.0.even}}}\\ \cline{2-9}
&
nmat & msize & nscal & naff &nmat & msize & nscal & naff\\
\hline
\hline
1 & 1 & 11 & 12 & 66 & \multirow{2}{*}{1} & \multirow{2}{*}{11} &\multirow{2}{*}{67} &\multirow{2}{*}{66}\\
2 & 11 & 66 & 67 & 1001 &  & & &\\
\hline
\hline
3 & 1 & 16 & 17 & 136 & \multirow{2}{*}{1} & \multirow{2}{*}{16} &\multirow{2}{*}{137} &\multirow{2}{*}{136}\\
4 & 16 & 136 & 137 & 3876 &  & & &\\
\hline
\hline
5 & 1 & 21 & 22 & 231 & \multirow{2}{*}{21} & \multirow{2}{*}{21} &\multirow{2}{*}{1562} &\multirow{2}{*}{1771}\\
6 & 21 & 231&232 & 10626 &  &  & & \\
\hline
\hline
7 & 1 & 26 & 27 & 351 & \multirow{2}{*}{26} & \multirow{2}{*}{26} &\multirow{2}{*}{2952} &\multirow{2}{*}{3276}\\
8 & 26 & 351 &352 & 23751 &  &  & &\\
\hline
\end{tabular}
\end{center}
\end{table}
\paragraph{Discussion:}
The behavior of our method is similar to that in Section \ref{sec:densePOP.exp}. 
Note that the graphs from Tables \ref{tab:SN} are dense so that we cannot exploit term sparsity or correlative sparsity for POP \eqref{eq:sta.num.POP} in these cases.
Moreover, for all graphs in Table \ref{tab:SN}, \revise{P\'olya SDP \eqref{eq:primal.problem.0.even}} provides the better bounds for $\alpha(G)$ compared to the ones returned by the second order relaxations of \revise{Putinar SDP \eqref{eq:sdp.Pu}}.

\begin{remark}
In Pb 3, 4 of Table \ref{tab:SN}, \revise{P\'olya SDP \eqref{eq:primal.problem.0.even}} with $k=1$ provides a better bound than \revise{P\'olya SDP \eqref{eq:primal.problem.0.even}} with $k=0$.
As shown in Remark \ref{re:special.case}, each SDP relaxation of \revise{P\'olya SDP \eqref{eq:primal.problem.0.even}} with $k=0$ and sufficiently large $s$ corresponds to an SDP relaxation obtained after exploiting term sparsity. 
\end{remark}

\subsection{Deciding the copositivity of a real symmetric matrix}
\label{sec:deci.copo2}
Given a symmetric matrix $\mathbf A\in\R^{n\times n}$, we say that $\mathbf A$ is copositive if $ \mathbf{u}^\top \mathbf A \mathbf{u}\ge 0 $ for all $ \mathbf{u}\in \R_+^n$.
Consider the following POP:
\begin{equation}\label{eq:copo.POP2}
\begin{array}{rl}
     f^\star:=\min\limits_{\mathbf x \in \R^n_+} \{ \x^\top \mathbf A \x\,:\, \sum_{j\in[n]}x_j=1\}\,.
\end{array}
\end{equation}
The matrix  $\mathbf A$ is copositive iff $f^\star \ge 0$.
\paragraph{Test problems:}
We construct several instances of the matrix $\mathbf A$ as follows:
\begin{enumerate}
    \item Take $B_{ij}$ randomly in $(-1,1)$ w.r.t. the uniform distribution, for all $i,j\in\{1,\dots,n\}$.
    \item Set $\mathbf B:=(B_{ij})_{1\le i,j\le n}$ and $\mathbf A:=\frac{1}{2}(\mathbf B+\mathbf B^\top )$.
\end{enumerate}

The numerical results are displayed in Table \ref{tab:copos}.
\begin{table}
    \caption{\small Numerical results for   deciding the copositivity of a real symmetric matrix.}
    \label{tab:copos}
\small
\begin{center}
\begin{tabular}{|c|c|c|c|c|c|c|c|c|c|c|c|c|c|c|c|c|c|c|c|}
        \hline
\multirow{2}{*}{Id} &\multirow{2}{*}{Pb}    &
\multicolumn{1}{c|}{POP size}&
\multicolumn{3}{c|}{\revise{Putinar SDP \eqref{eq:sdp.Pu}}}&      \multicolumn{4}{c|}{\revise{P\'olya SDP \eqref{eq:primal.problem.0.even}}}\\ \cline{3-10}
&
&{$n$}&
{$k$} & {val} & {time} & {$k$} & {$s$}& {val} & {time}\\
\hline
\hline
1 &\multirow{2}{*}{1}& \multirow{2}{*}{10} &  1 &-1.45876  & 0.004 &  \multirow{2}{*}{0} & \multirow{2}{*}{8} & \multirow{2}{*}{\bf -0.94862$^*$} & \multirow{2}{*}{1}\\
2& &  &  2 & {\bf -0.94862} & 0.2  &  &  &  & \\
\hline
\hline
3 &\multirow{2}{*}{2}& \multirow{2}{*}{15} &  1 & -1.41319 & 0.007&\multirow{2}{*}{0} & \multirow{2}{*}{13} & \multirow{2}{*}{\bf -0.65197$^*$} & \multirow{2}{*}{1}\\
4& & &  2 &  {\bf -0.65197} & 10  & &  &  & \\
\hline
\hline
5 &\multirow{2}{*}{3}& \multirow{2}{*}{20} &  1 & -1.40431  & 0.02& \multirow{2}{*}{0} & \multirow{2}{*}{20} & \multirow{2}{*}{\bf -0.98026$^*$} & \multirow{2}{*}{1}\\
6& &  &  2 & {\bf  -0.98026} & 89&  & &  & \\
\hline
\hline
7 &\multirow{2}{*}{4}& \multirow{2}{*}{25} &  1 & -1.34450  & 0.03& \multirow{2}{*}{0} & \multirow{2}{*}{19} & \multirow{2}{*}{\bf -0.97345$^*$} & \multirow{2}{*}{2}\\
8& &  &  2 &  {\bf -0.97345} & 519 &  &  & & \\
\hline
\end{tabular}

\begin{tabular}{|c|c|c|c|c|c|c|c|c|c|c|c|c|c|c|c|c|}
        \hline
\multirow{2}{*}{Id}&
\multicolumn{4}{c|}{\revise{Putinar SDP \eqref{eq:sdp.Pu}}}&      \multicolumn{4}{c|}{\revise{P\'olya SDP \eqref{eq:primal.problem.0.even}}}\\ \cline{2-9}
&
nmat & msize & nscal & naff &nmat & msize & nscal & naff\\
\hline
\hline
1 & 1 & 11 & 12 & 66 & \multirow{2}{*}{4} & \multirow{2}{*}{8} &\multirow{2}{*}{67} &\multirow{2}{*}{66}\\
2 & 11 & 66 & 67 & 1001 &  &  & &\\
\hline
\hline
3 & 1 & 16 & 17 & 136 & \multirow{2}{*}{4} & \multirow{2}{*}{13} &\multirow{2}{*}{137} &\multirow{2}{*}{136}\\
4 & 16 & 136 & 137 & 3876 &  &  & & \\
\hline
\hline
5 & 1 & 21 & 22 & 231 & \multirow{2}{*}{2} & \multirow{2}{*}{20} &\multirow{2}{*}{232} &\multirow{2}{*}{231}\\
6 & 21 & 231&232 & 10626 &  &  & & \\
\hline
\hline
7 & 1 & 26 & 27 & 351 &  \multirow{2}{*}{8} & \multirow{2}{*}{19} &\multirow{2}{*}{352} &\multirow{2}{*}{351}\\
8 & 26 & 351 &352 & 23751 &  & & &\\
\hline
\end{tabular}
\end{center}
\end{table}
\paragraph{Discussion:}
The behavior of our method is similar to that in Section \ref{sec:densePOP.exp}.
In all cases, we can extract the solutions of the resulting POP and certify that $\mathbf A$ is not copositive since $f^\star$ is negative.

\subsection{Deciding the nonnegativity of an even degree form on the nonegative orthant}
\label{sec:deci.pos2}
Given a form $q\in\R[\x]$, $q$ is nonnegative on $\R^n_+$ iff $q$ is nonnegative on the unit simplex
\begin{equation}
\begin{array}{rl}
     \Delta:=\{\x\in\R^n_+\,:\,\sum_{j\in[n]}x_j=1\}\,.
\end{array}
\end{equation}
Given a form $f\in\R[\x]$ of degree $2d$, we consider the following POP:
\begin{equation}\label{eq:pos.POP2}
       f^\star:=\min\limits_{\x\in \Delta} f(\x) \,.
\end{equation}
Note that if $d=1$, problem \eqref{eq:pos.POP2} boils down to  deciding the copositivity of the Gram matrix associated to $f$. 
Thus, we consider the case where $d\ge 2$.
\paragraph{Test problems:}
We construct several instances of the form $f$ of degree $2d$ as follows:
\begin{enumerate}
    \item Take $f_\a$ randomly in $(-1,1)$ w.r.t. the uniform distribution, for each $\a\in\N^n$ with $|\a|=2d$.
    \item Set $f:=\sum_{|\a|=2d} f_\a  \x^\a$.
\end{enumerate}
The numerical results are displayed in Table \ref{tab:nonneg}.
\begin{table}
    \caption{\small Numerical results for   deciding the  nonnegativity of an even degree form on the nonegative orthant, with $d=2$.}
    \label{tab:nonneg}
\small
\begin{center}
\begin{tabular}{|c|c|c|c|c|c|c|c|c|c|c|c|c|c|c|c|c|c|c|c|}
        \hline
\multirow{2}{*}{Id} &\multirow{2}{*}{Pb}    &
\multicolumn{1}{c|}{POP size}&
\multicolumn{3}{c|}{\revise{Putinar SDP \eqref{eq:sdp.Pu}}}&      \multicolumn{4}{c|}{\revise{P\'olya SDP \eqref{eq:primal.problem.0.even}}}\\ \cline{3-10}
&
&{$n$}&
{$k$} & {val} & {time} & {$k$} & {$s$}& {val} & {time}\\
\hline
\hline
1 &\multirow{2}{*}{1}& \multirow{2}{*}{5} &  2 &-1.87958  & 0.001&  \multirow{2}{*}{0} & \multirow{2}{*}{8} & \multirow{2}{*}{\bf -0.68020$^*$} & \multirow{2}{*}{1} \\
2& & &  3 & {\bf -0.68020} & 0.06&  &  &  & \\
\hline
\hline
3 &\multirow{2}{*}{2}& \multirow{2}{*}{10} &  2 & -1.87491 & 0.1& \multirow{2}{*}{0} & \multirow{2}{*}{11} & \multirow{2}{*}{\bf -0.87524$^*$} & \multirow{2}{*}{5} \\
4& & &  3 &  {\bf -0.87524} & 10  &  &  &  & \\
\hline
\hline
5 &\multirow{2}{*}{3}& \multirow{2}{*}{15} &  2 & -2.01566  & 6& \multirow{2}{*}{0} & \multirow{2}{*}{44} & \multirow{2}{*}{\bf -0.86938$^*$} & \multirow{2}{*}{79}\\
6& &  &  3 & {\bf -0.86938} & 7675&  &  &  & \\
\hline
\end{tabular}

\begin{tabular}{|c|c|c|c|c|c|c|c|c|c|c|c|c|c|c|c|c|}
        \hline
\multirow{2}{*}{Id}&
\multicolumn{4}{c|}{\revise{Putinar SDP \eqref{eq:sdp.Pu}}}&      \multicolumn{4}{c|}{\revise{P\'olya SDP \eqref{eq:primal.problem.0.even}}}\\ \cline{2-9}
&
nmat & msize & nscal & naff &nmat & msize & nscal & naff\\
\hline
\hline
1 & 6 & 21 & 22 & 126 &  \multirow{2}{*}{31} & \multirow{2}{*}{6} &\multirow{2}{*}{72} &\multirow{2}{*}{126}\\
2 & 6 & 56 & 232 & 462 & &  & &\\
\hline
\hline
3 & 11 & 66 & 67 & 1001 &  \multirow{2}{*}{111} & \multirow{2}{*}{11} &\multirow{2}{*}{617} &\multirow{2}{*}{1001}\\
4 & 11 & 268 & 2212 & 8008 &  &  & &\\
\hline
\hline
5 & 16 & 136 & 137 & 3876 & \multirow{2}{*}{213} & \multirow{2}{*}{44} &\multirow{2}{*}{2637} &\multirow{2}{*}{3876}\\
6 & 16 & 816 &9317 & 54264 &  &  & &\\
\hline
\end{tabular}
\end{center}
\end{table}
\paragraph{Discussion:}
The behavior of our method is similar to that in Section \ref{sec:densePOP.exp}.
In these cases, we were able to extract the solution of the resulting POPs. One can then conclude that $f$ is not nonnegative on the nonnegative orthant since it has negative value at its atoms.
\subsection{Minimizing a polynomial over the boolean hypercube}
\label{sec:boolean}
Consider the optimization problem:
\begin{equation}
    \min_{\x\in\{0,1\}^n} f(\x)\,,
\end{equation}
where $f$ is a polynomial of degree at most $2d$.
It is equivalent to the following POP on the nonnegative orthant:
\begin{equation}
    \min_{\x\in\R^n_+} \{f(\x)\,:\,x_j(1-x_j)=0\,,\,j\in[n]\}\,,
\end{equation}
\paragraph{Test problems:}
We construct several instances by taking the coefficients of $f$ randomly in $(-1,1)$ w.r.t. to the uniform distribution.

The numerical results are displayed in Table \ref{tab:binary}.
\begin{table}
    \caption{\small Numerical results for   minimizing polynomials over the boolean hypercube, with $d=1$.}
    \label{tab:binary}
\small
\begin{center}
\begin{tabular}{|c|c|c|c|c|c|c|c|c|c|c|c|c|c|c|c|c|c|c|c|}
        \hline
\multirow{2}{*}{Id} &\multirow{2}{*}{Pb}    &
\multicolumn{1}{c|}{POP size}&
\multicolumn{3}{c|}{\revise{Putinar SDP \eqref{eq:sdp.Pu}}}&      \multicolumn{4}{c|}{\revise{P\'olya SDP}}\\ \cline{3-10}
&
&{$n$}&
{$k$} & {val} & {time} & {$k$} & {$s$}& {val} & {time}\\
\hline
\hline
1 &\multirow{2}{*}{1}& \multirow{2}{*}{10} &  1 &-4.61386  & 0.008& \multirow{2}{*}{1} & \multirow{2}{*}{11} & \multirow{2}{*}{\bf -4.34345} & \multirow{2}{*}{1}\\
2& &  &  2 & {\bf -4.34345} & 0.2&  &  &  & \\
\hline
\hline
3 &\multirow{2}{*}{1}& \multirow{2}{*}{20} &  1 & -15.4584 & 0.02& \multirow{2}{*}{1} & \multirow{2}{*}{21} & \multirow{2}{*}{\bf -14.9455} & \multirow{2}{*}{4} \\
4& &  &  2 &  {\bf -14.9455} & 108  &  &  &  & \\
\hline
\hline
5 &\multirow{2}{*}{3}& \multirow{2}{*}{30} &  1 & -29.3433  & 0.1& \multirow{2}{*}{1} & \multirow{2}{*}{31} & \multirow{2}{*}{\bf -27.6311} & \multirow{2}{*}{41}\\
6& &  &  2 & {\bf -27.6311} & 8068&  &  & & \\
\hline
\end{tabular}

\begin{tabular}{|c|c|c|c|c|c|c|c|c|c|c|c|c|c|c|c|c|}
        \hline
\multirow{2}{*}{Id}&
\multicolumn{4}{c|}{\revise{Putinar SDP \eqref{eq:sdp.Pu}}}&      \multicolumn{4}{c|}{\revise{P\'olya SDP}}\\ \cline{2-9}
&
nmat & msize & nscal & naff &nmat & msize & nscal & naff\\
\hline
\hline
1 & 1&11 & 21 &66 & \multirow{2}{*}{11} & \multirow{2}{*}{11} & \multirow{2}{*}{276} &\multirow{2}{*}{286}\\
2 & 6 & 56 & 232 & 462&  &  &  &\\
\hline
\hline
3 & 1 & 21 & 41 & 231 & \multirow{2}{*}{21} & \multirow{2}{*}{21} &\multirow{2}{*}{1751} &\multirow{2}{*}{1771}\\
4 & 21 & 231 & 4621 & 10626 &  &  & & \\
\hline
\hline
5 & 1 & 31 & 61 & 496 & \multirow{2}{*}{31} & \multirow{2}{*}{31} &\multirow{2}{*}{5426} &\multirow{2}{*}{5456}\\
6 & 31 & 496 &14881 & 46376 &  &  & &\\
\hline
\end{tabular}
\end{center}
\end{table}
\paragraph{Discussion:}
The behavior of our method is similar to that in Section \ref{sec:densePOP.exp}.
Note that \revise{P\'olya SDP \eqref{eq:primal.problem.0.even}} with order $k=0$ provides worse bounds than \revise{Putinar SDP \eqref{eq:sdp.Pu}} with order $k=2$. However, as shown in Table \ref{tab:binary}, P\'ol with order $k=1$ provides the same bounds as Put with order $k=2$.

\subsection{Sparse QCQPs}
\label{sec:sparse.exepr}

\paragraph{Test problems:} We construct randomly generated QCQPs in the form \eqref{eq:constrained.problem.poly.even}-\eqref{eq:semial.set.def.2.even} with correlative sparsity as follows:
\begin{enumerate}
    \item Take a positive integer $u$, $p:=\lfloor n/u\rfloor +1$ and let
    \begin{equation}
        I_c=\begin{cases}
        [u],&\text{if }c=1\,,\\
        \{u(c-1),\dots,uc\},&\text{if }c\in\{2,\dots,p-1\}\,,\\
        \{u(p-1),\dots,n\},&\text{if }c=p\,;
        \end{cases}
    \end{equation}
    \item Generate a quadratic polynomial objective function $f=\sum_{c\in[p]}f_c$ such that for each $c\in[p]$, $f_c\in\R[ x(I_c)]_2$, and the coefficient $f_{c,\a},\a\in\N^{I_c}_2$ of $f_c$ is randomly generated in $(-1,1)$ w.r.t. the uniform distribution;
    \item Take a random point $\mathbf{a}$ such that for every $c\in[p]$, $\mathbf{a}(I_c)$ belongs to the simplex
    \begin{equation}
    \begin{array}{l}
         \Delta^{(c)}:=\{\x(I_c)\in\R^{n_c}\,:\,x_j\ge 0\,,\,j\in I_c\,,\,\sum_{j\in I_c}x_j\le 1\}
    \end{array}
    \end{equation}
    \item Let $q:=\lfloor m_{\text{ineq}}/p\rfloor$ and \begin{equation}
        J_c:=\begin{cases} \{(c-1)q+1,\dots,cq\},&\text{if }c\in[p-1]\,,\\
    \{(p-1)q+1,\dots,l\},&\text{if } c=p\,.
    \end{cases}
    \end{equation}
     For every $c\in[p]$ and every $i\in J_c$, generate a quadratic polynomial $g_i\in\R[\x(I_c)]_2$ by
    \begin{enumerate}
        \item for each $\a\in\N^{I_c}_2\backslash \{\mathbf{0}\}$, taking a random coefficient $\mathbf{G}_{i,\a}$ of $h_i$ in $(-1,1)$ w.r.t. the uniform distribution;
        \item setting $g_{i, 0}:=0.125-\sum_{\a\in\N^{I_c}_2\backslash \{\mathbf{0}\}} g_{j,\a}  \mathbf{a}^\a$.
    \end{enumerate}
    \item Take $g_{i_c}:=1-\sum_{i\in I_c}x_{i}$, for some $i_c\in J_t$, for $c\in[p]$;
    \item Let $r:=\lfloor m_{\text{eq}}/p\rfloor$ and \begin{equation}\label{eq:assign.eqcons2}
        W_c:=\begin{cases} \{(c-1)r+1,\dots,cr\},&\text{if }c\in[p-1]\,,\\
    \{(p-1)r+1,\dots,l\},&\text{if } c=p\,.
    \end{cases}
    \end{equation}
     For every $c\in[p]$ and every $i\in W_c$, generate a quadratic polynomial $h_i\in\R[\x(I_c)]_2$ by
    \begin{enumerate}
        \item for each $\a\in\N^{I_c}_2\backslash \{\mathbf{0}\}$, taking a random coefficient $h_{i,\a}$ of $h_i$ in $(-1,1)$ w.r.t. the uniform distribution;
        \item setting $h_{i, \mathbf{0}}:=-\sum_{\a\in\N^{I_c}_2\backslash \{\mathbf{0}\}} h_{i,\a} \mathbf{a}^\a$.
    \end{enumerate}
    \item Take $g_{i+m_{\text{ineq}}}(\x):=h_i$ and set $g_{i+m_{\text{eq}}+m_{\text{ineq}}}=-h_i$, for $i\in[m_{\text{eq}}]$.
\end{enumerate}
Here $m=m_{\text{ineq}}+2 m_{\text{eq}}$ with $m_{\text{ineq}}$ (resp. $m_{\text{eq}}$) being the number of inequality (resp. equality) constraints except the nonnegative constraints $x_j\ge 0$.
The point $\mathbf{a}$ is a feasible solution of POP \eqref{eq:constrained.problem.poly.even}.

The numerical results are displayed in Table \ref{tab:QCQP.on.unit.ball.sp}.
\begin{table}
    \caption{\small Numerical results for randomly generated QCQPs with correlative sparsity, $n=1000$ and $d=\deg(f)=2$.}
    \label{tab:QCQP.on.unit.ball.sp}
\small
\begin{center}
\begin{tabular}{|c|c|c|c|c|c|c|c|c|c|c|c|c|c|c|c|c|c|c|c|c|}
        \hline
\multirow{2}{*}{Id} & \multirow{2}{*}{Pb}    &
\multicolumn{3}{c|}{POP size}&
\multicolumn{3}{c|}{\revise{Sparse Putinar SDP \eqref{eq:sdp.Pu.sp}}}&      \multicolumn{4}{c|}{\revise{Sparse P\'olya SDP \eqref{eq:primal.problem.0.even.sparse}}}\\ \cline{3-12}&
&{$u$}& {$m_{\text{ineq}}$}& {$m_{\text{eq}}$}& {$k$} & {val} & {time} & {$k$} & {$s$}& {val} & {time}\\
\hline
\hline
1 & \multirow{2}{*}{1} & \multirow{2}{*}{10} & \multirow{2}{*}{201} & \multirow{2}{*}{0} & 1 & $\infty$ & 1.5& \multirow{2}{*}{0} & \multirow{2}{*}{10} & \multirow{2}{*}{ -128.906} & \multirow{2}{*}{15}  \\
2 & & & & &  2 & -129.061 & 385 & & &  &  \\
\hline
\hline
3 & \multirow{2}{*}{2}&\multirow{2}{*}{10} & \multirow{2}{*}{201} & \multirow{2}{*}{200} & 1 & $\infty$ & 2.0& \multirow{2}{*}{1} & \multirow{2}{*}{12} & \multirow{2}{*}{ -65.3195} & \multirow{2}{*}{51}  \\
4 & & & & &  2 & -66.0696 & 475& &  & &  \\
\hline
\hline
5 & \multirow{2}{*}{3}&\multirow{2}{*}{20} & \multirow{2}{*}{201} & \multirow{2}{*}{0} & 1 & $\infty$ & 3.6&  \multirow{2}{*}{0} & \multirow{2}{*}{15} & \multirow{2}{*}{ -65.9794} & \multirow{2}{*}{19} \\
6 & & & & &  2 & -66.1306 & 56360&  & &  &  \\
\hline
\hline
7 & \multirow{2}{*}{4}&\multirow{2}{*}{20} & \multirow{2}{*}{201} & \multirow{2}{*}{200} & 1 & $\infty$ &9 & \multirow{2}{*}{1}& \multirow{2}{*}{22} & \multirow{2}{*}{ -38.2061} &\multirow{2}{*}{319}  \\
8 & & & & &2 & $-$ & $-$ & & &  &  \\
\hline
\end{tabular}

\begin{tabular}{|c|c|c|c|c|c|c|c|c|c|c|c|c|c|c|c|c|}
        \hline
\multirow{2}{*}{Id}&
\multicolumn{4}{c|}{\revise{Sparse Putinar SDP \eqref{eq:sdp.Pu.sp}}}&      \multicolumn{4}{c|}{\revise{Sparse P\'olya SDP \eqref{eq:primal.problem.0.even.sparse}}}\\ \cline{2-9}
&
nmat & msize & nscal & naff &nmat & msize & nscal & naff\\
\hline
\hline
1 & 100 & 12 & 1201 & 7491 & \multirow{2}{*}{299} & \multirow{2}{*}{10} &\multirow{2}{*}{7889} &\multirow{2}{*}{7491} \\
2 & 1300 & 78 & 1 & 135641 & & & &\\
\hline
\hline
3 & 100 & 12 & 1401 & 7491 & \multirow{2}{*}{1299} & \multirow{2}{*}{12} &\multirow{2}{*}{39920}  &\multirow{2}{*}{43813} \\
4 & 1300 & 78 & 15577 & 135641 & &  &  &\\
\hline
\hline
5 & 50 & 22 & 1201 & 12481 & \multirow{2}{*}{399} & \multirow{2}{*}{15} &\multirow{2}{*}{25407}  &\multirow{2}{*}{25109}\\
6 & 1250 & 253 & 1 & 630231 &  &  &  &\\
\hline
\hline
7 & 50 & 22 & 1401 &12481  & \multirow{2}{*}{1149} & \multirow{2}{*}{22} & \multirow{2}{*}{108641} & \multirow{2}{*}{113428} \\
8 & 1250  & 253 & 50513 & 630231 &  & & &\\
\hline
\end{tabular}
\end{center}
\end{table}
\paragraph{Discussion:}
Similarly to the previous discussion, \revise{Sparse P\'olya SDP \eqref{eq:primal.problem.0.even.sparse}} in Table \ref{tab:QCQP.on.unit.ball.sp} is also much faster and more accurate than \revise{Sparse Putinar SDP \eqref{eq:sdp.Pu.sp}}.
For instance, when $u=20$, $m_{\text{ineq}}=201$ and $m_{\text{eq}}=0$, \revise{Sparse P\'olya SDP \eqref{eq:primal.problem.0.even.sparse}} takes $20$ seconds to return the lower bound $-65.9794$ for  $f^\star$, while \revise{Sparse Putinar SDP \eqref{eq:sdp.Pu.sp}} takes $56360$ seconds to return a worse bound of $-66.1306$.
In this case, \revise{Sparse P\'olya SDP \eqref{eq:primal.problem.0.even.sparse}} has $399$ matrix variables with maximal matrix size $15$, while \revise{Sparse Putinar SDP \eqref{eq:sdp.Pu.sp}} has $1250$ matrix variables with maximal matrix size $253$.

\subsection{Robustness certification of deep neural networks}
\label{sec:robus.cert}

In \cite{raghunathan2018semidefinite}, the
robustness certification problem of a multi-layer neural network with ReLU activation function is formulated as the following QCQP for each $y$:
\begin{equation}
\label{eq:cert}
    \begin{array}{rl}
         l^\star_y(\bar{\x},\bar y):=\max\limits_{\x^0,\dots,\x^L} & (\mathbf c_y-\mathbf c_{\bar y})^\top \x^L\\
         \text{s.t.}& 
         x^i_t(x^i_t-\mathbf W^{i-1}_t\x^{i-1})=0\,,\,x^i_t\ge 0\,,\,x^i_t\ge \mathbf W^{i-1}_t\x^{i-1}\,,\\
         &\qquad\qquad\qquad\qquad\qquad\qquad\qquad t\in[m_i]\,,\,i\in[L]\\
         &-\varepsilon\le x_t^0-\bar x_{t}\le \varepsilon\,,\,t\in [m_0]\,,
    \end{array}
\end{equation}
where we use the same notation as in \cite[Section 2]{raghunathan2018semidefinite} and write $\mathbf W^{i-1}=\begin{bmatrix}
\mathbf W_1^{i-1}\\
\dots\\
\mathbf W_{m_i}^{i-1}
\end{bmatrix}$.

We say that the network is certifiably  $\varepsilon$-robust on $(\bar{\x},\bar y)$ if $l_y^\star(\bar{\x},\bar y)<0$ for all $y\ne \bar y$.
\paragraph{Test problems:}
To obtain an instance of weights $\mathbf W^i$, we train a classification model by using Keras\footnote{\url{https://keras.io/api/models/model_training_apis/}}. 
Explicitly, we minimize a loss function as follows:
\begin{equation}
\label{eq:model.weights}
    \begin{array}{rl}
        \min\limits_{\mathbf W^0,\dots,\mathbf W^{L-1}} & \frac{1}{2}\sum_{(\x^0,y^0)\in  \mathcal D}\|f(\x^0)-\mathbf e_{y^0}\|_2^2\,,
    \end{array}
\end{equation}
where the network $f$ is defined as in \cite[Section 2]{raghunathan2018semidefinite} and  $\mathbf e_{y^0}$ has 1 at the $y^0$-th element and zeros at the others.
Here the input set $\mathcal D$ is a part of Boston House Price Dataset (BHPD). 
The class label $y^0$ is assigned to the input $\x^0$.
We classify the inputs from BHPD into $3$ classes according to the MEDian Value of owner-occupied homes (MEDV) in \$1000 as follows:
\begin{equation}
    y^0=\begin{cases}
    1&\text{if }\text{MEDV}<10\,,\\
    2&\text{if }10\le\text{MEDV}<20\,,\\
    3&\text{otherwise}\,.
    \end{cases}
\end{equation}
We also take a clean input label pair $(\bar{\x},\bar y)\notin \mathcal D$ with $\bar y=3$ from BHPD.

As shown in \cite[Section 4.2]{chen2022sublevel}, POP \eqref{eq:cert} has correlative sparsity.
To use our method, we convert \eqref{eq:cert} to a POP on the nonnegative orthant by defining  new nonnegative variables $\bar z_t:=x_t^0-\bar x_{t}+\varepsilon$. 
Doing so, the constraints $ -\varepsilon \leq \bar{x}_t-x_t^0 \leq \varepsilon  $ become $0 \leq  \bar z_t \leq  2\varepsilon $ in the new coordinate system.
Here we choose $\varepsilon=0.1$.
More detailed information for our training model are available in Table \ref{tab:BHPD}.
\begin{table}
\caption{Information for the training model \eqref{eq:model.weights}.}
\label{tab:BHPD}
\begin{center}
\small
\begin{tabular}{ |l|c| }
 \hline
Dataset&  BHPD \\\hline
 Number of hidden layers&  $L=2$\\\hline
Length of an input&  13 \\\hline
Number of inputs&  506\\\hline
Test size&20\%\\\hline
Number of classes&  $k=3$ \\\hline
Numbers of units in layers&  $m=(13, 20, 10)$ \\\hline
Number of weights & 490\\
 \hline
 Opimization method &Adadelta algorithm\footnote{\url{https://keras.io/api/optimizers/adadelta/}}\\\hline
  Accuracy &70\%\\\hline
  Batch size & 128\\\hline
                  Epochs&200\\\hline
\end{tabular}
\end{center}
\end{table}

The numerical results are displayed in Table \ref{tab:cert}.
\begin{table}
    \caption{\small Numerical results for robustness certification on BHPD, $n=43$, $m_\text{ineq}=43$, $m_\text{eq}=30$ and $d=\deg(f)=2$.}
    \label{tab:cert}
\small
\begin{center}
\begin{tabular}{|c|c|c|c|c|c|c|c|c|c|c|c|c|c|c|c|c|c|c|c|c|}
        \hline
\multirow{2}{*}{Id} & \multirow{2}{*}{Pb}    &
\multicolumn{3}{c|}{\revise{Sparse Putinar SDP \eqref{eq:sdp.Pu.sp}}}&      \multicolumn{4}{c|}{\revise{Sparse P\'olya SDP \eqref{eq:primal.problem.0.even.sparse}}}\\ \cline{3-9}&
& {$k$} & {val} & {time} & {$k$} & {$s$}& {val} & {time}\\
\hline
\hline
1 & \multirow{2}{*}{$y=1$}& 1 & 88.1571 & 0.4& \multirow{2}{*}{1} & \multirow{2}{*}{35} &  \multirow{2}{*}{\bf -11.8706} & \multirow{2}{*}{625}\\
2 & &  2 & $-$ & $-$&  &  &  & \\
\hline
\hline
3 & \multirow{2}{*}{$y=2$}& 1 & 208.934& 0.4& \multirow{2}{*}{1} & \multirow{2}{*}{35} &  \multirow{2}{*}{\bf -13.3240} & \multirow{2}{*}{518}\\
4 & &  2 & $-$ & $-$&  & &  &\\
\hline
\end{tabular}

\begin{tabular}{|c|c|c|c|c|c|c|c|c|c|c|c|c|c|c|c|c|c|c|c|c|}
        \hline
\multirow{2}{*}{Id}&
\multicolumn{4}{c|}{\revise{Sparse Putinar SDP \eqref{eq:sdp.Pu.sp}}}&      \multicolumn{4}{c|}{\revise{Sparse P\'olya SDP \eqref{eq:primal.problem.0.even.sparse}}}\\ \cline{2-9}
&
nmat & msize & nscal & naff &nmat & msize & nscal & naff\\
\hline
\hline
1,3 & 23 & 22 & 117 & 737 & \multirow{2}{*}{297} & \multirow{2}{*}{35} &\multirow{2}{*}{46233} &\multirow{2}{*}{28195}\\
2,4 & 97 &595 &14431 & 86285 &  &  & &\\
\hline
\end{tabular}
\end{center}
\end{table}
\paragraph{Discussion:}
Compared to \revise{Sparse Putinar SDP \eqref{eq:sdp.Pu.sp}}, \revise{Sparse P\'olya SDP \eqref{eq:primal.problem.0.even.sparse}} provides better upper bounds in less total time.
Moreover, in Table \ref{tab:cert}, the values returned by \revise{Sparse Putinar SDP \eqref{eq:sdp.Pu.sp}} with $k=1$ are positive and are much larger than the negative ones returned by \revise{Sparse P\'olya SDP \eqref{eq:primal.problem.0.even.sparse}} with $k=1$.
Since in Table \ref{tab:cert}, the upper bounds on $l^\star_y(\bar{\x},\bar y)$ are negative, for all $y\ne \bar y$, $l^\star_y(\bar{\x},\bar y)$ must be negative.
Thus, we conclude that this network is certifiably $\varepsilon$-robust on $(\bar{\x},\bar y)$.

\section{Conclusion}
We have proposed in this paper semidefinite relaxations for solving dense POPs on the nonnegative orthant.
The basic idea is to apply a positivity certificate involving SOS of monomials for a POP with input polynomials being  even in each variable. 
It allows us to obtain a hierarchy of  linear relaxations. 
Afterwards we replace each SOS of monomials by an SOS associated with a block-diagonal Gram matrix, where each block has a prescribed size. 
This ensures the efficiency of the corresponding hierarchy of SDP relaxations in practice.
The convergence is still maintained, as it is based on the convergence guarantee of the hierarchy of linear relaxations. 
The resulting convergence rate of  $\mathcal{O}(\varepsilon^{-c})$ is similar to the one of Baldi and Mourrain \cite{baldi2023effective}. 

As a topic of further applications, we would like to use our method for
solving large-scale POPs for phase retrieval and feedforward neural networks.

\paragraph{Acknowledgements.} 
The first author was supported by the MESRI funding from EDMITT.
This work has benefited from the European Union's Horizon 2020 research and innovation programme under the Marie Sklodowska-Curie Actions, grant agreement 813211 (POEMA) as well as from the AI Interdisciplinary Institute ANITI funding, through the French ``Investing for the Future PIA3'' program under the Grant agreement n$^{\circ}$ANR-19-PI3A-0004.
This research is part of the programme DesCartes and is supported by the National Research Foundation, Prime Minister's Office, Singapore under its Campus for Research Excellence and Technological Enterprise (CREATE) programme.

\section{Appendix}
\subsection{Preliminary material}
For each $q=\sum_{\a}q_\a \x^\a\in\R[\x]$, we note $\| q \|: = \max_\a   \frac{| {{q_\a }} |}{c_\a}$ with $c_\a := \frac{|\a|!}{\alpha_1!\dots\alpha_n!}$ for each $\alpha \in \N^n$.
This defines a norm on the real vector space $\R[\x]$. 
Moreover, for $p_1,q_2\in\R[\x]$, we have
\begin{equation}\label{eq:mul.ineq}
    \|q_1q_2\|\le \|q_1\|\|q_2\|\,,
\end{equation}
according to \cite[Lemma 8]{schweighofer2004complexity}.

We recall the following bound for central binomial coefficient stated in \cite[page 590]{jukna2012boolean}:
\begin{lemma}\label{lem:bound.cetral.coe}
For all $t\in\N_{> 0}$, it holds that $\binom{2t}{t}\frac{1}{2^{2t}}\le \frac{1}{\sqrt{\pi t}}$.
\end{lemma}
\if{\begin{proof}
For all $t\in\N_{> 0}$, we get
\begin{equation}
\begin{array}{l}
     2t\left[\binom{2t}{t}\frac{1}{2^{2t}}\right]^2=\frac{1}{2}\frac{3}{2}\frac{3}{4}\frac{5}{5}\dots\frac{2t-1}{2t-2}\frac{2n-1}{2n}=\frac{1}{2}\prod_{j=2}^t\left(1+\frac{1}{4j(j+1)}\right).
\end{array}
\end{equation}
By Wallis's formula, the middle expression converges to $\frac{2}{\pi}$. 
The right hand side is increasing, yielding the desired result.
\end{proof}}\fi
Define the simplex 
\begin{equation}\label{eq:simplex.def2}
\begin{array}{l}
    \Delta_n:=\{\x\in\R^n\,:\,x_j\ge 0\,,\,j\in[n]\,,\,\sum_{j\in[n]}x_j=1\}\,.
\end{array}
\end{equation}

We recall the degree bound for P\'olya's Positivstellensatz \cite{polya1928positive}:
\begin{lemma}\label{lem:bound.polya}
(Powers and Reznick \cite{powers2001new})
If $q$ is a homogeneous polynomial of degree $d$ positive on $\Delta_n$,
then for all $k\in\N$ satisfying 
\begin{equation}\label{eq:bound.polya2}
    k\ge \frac{d(d-1)\|q\|}{2\min_{\x\in\Delta_n}q(\x)} -d\,,
\end{equation}
$(\sum_{j\in[n]}x_j)^kq$ has positive coefficients.
\end{lemma}

Let us recall the concept and the properties of polynomials even in each variable in \cite[Definition 3.3]{schabert2019uniform}. 
A polynomial $q$ is even in each variable if for every $j\in[n]$, 
\begin{equation}
    q(x_1,\dots,x_{j-1},-x_j,x_{j+1},\dots,x_{n})=q(x_1,\dots,x_{j-1},x_j,x_{j+1},\dots,x_{n})\,.
\end{equation}

If $q$ is even in each variable, then there exists a polynomial  $\tilde{q}$ such that $q=\tilde q(x_1^2,\dots,x_n^2)$.
Indeed, let $q=\sum_{\a\in\N^n}q_\a \x^\a$ be a polynomial even in each variable.
Let $j\in[n]$ be fixed.
Then 
$q(\x)=\frac{1}2(q(\x)+q(x_1,\dots,x_{j-1},-x_j,x_{j+1},\dots,x_{n}))$.
It implies that $q_\a=0$ if $\alpha_j$ is odd.
Thus, $q=\sum_{\a\in\N^n}q_{2\a} \x^{2\a}$ since $j$ is arbitrary in $[n]$.
This yields $\tilde{q}=\sum_{\a\in\N^n}q_{2\a} \x^{\a}$.

For convenience, we denote $\x^2:=(x_1^2,\dots,x_n^2)$.
Moreover, if $q$ is even in each variable and homogeneous of  degree $2d_q$, then $\tilde{q}$ is homogeneous of degree $d_q$.
Conversely, if $q$ is a polynomial of degree at most $2d$ such that $q$ is even in each variable, then the degree-$2d$ homogenization of $q$ is even in each variable.
\subsection{The proof of Theorem \ref{theo:complex.putinar.vasilescu.even}}
\label{proof:PV.SOSmono}

\begin{proof}
Let $\varepsilon>0$.
By assumption, $\deg(f)=2d_f$, $\deg(g_i)=2d_{g_i}$ for some $d_f,d_{g_i}\in\N$,  for $j\in [m]$.
\paragraph{Step 1: Converting to polynomials on the nonnegative orthant.}
We claim that $\tilde{f}$ is nonnegative on the semialgebraic set
\begin{equation}\label{eq:set.S.tilde}
    \tilde S:=\{\x\in\R^n\,:\,x_j\ge 0\,,\,j\in[n]\,,\,\tilde{g}_i(\x)\ge 0\,,\,i\in[m]\}\,.
\end{equation}
Let $y\in \tilde S$. 
Set $\z=(\sqrt{y_1},\dots,\sqrt{y_n})$.
Then $g_i(\z)=\tilde{g}_i(\z^2)=\tilde{g}_i(\y)\ge 0$, for $i\in[m]$.
By assumption, $\tilde f(\y)=\tilde f(\z^2)=f(\z)\ge 0$.
It implies that $\tilde{f}+\varepsilon(\sum_{j=1}^n x_j)^{d_f}$ is homogeneous and positive on $\tilde S\backslash \{\mathbf{0}\}$.

To prove the first statement, we proceed exactly as in the proof of \cite[Theorem 2.4]{dickinson2015extension} for $\tilde{f}+\varepsilon(\sum_{j=1}^n x_j)^{d_f}$ and derive the bound on the degree of polynomials having positive coefficients when applying P\'olya's Positivstellensatz.
To obtain \eqref{eq:represent.even}, we replace $\x$ by $\x^2$ in the representation of $\tilde{f}+\varepsilon(\sum_{j=1}^n x_j)^{d_f}$.

We shall prove the second statement. Assume that $S$ has nonempty interior.
Set $\bar m:=m+n$ and $ g_{m+j}:=x_j^2$ with $d_{g_{m+j}}:=1$, $j\in[n]$.
Then $\tilde g_{m+j}:=x_j$, $j\in[n]$, and
\begin{equation}
    \tilde S:=\{\x\in\R^n\,:\,\tilde{g}_i(\x)\ge 0\,,\,i\in[\bar m]\}\,.
\end{equation}
Note that $\deg(\tilde g_i)=d_{g_i}$, $i\in[\bar m]$.
Since $S$ has nonempty interior and $\cup_{j=1}^n\{\x\in\R^n\,:\,x_j=0\}$    has zero Lebesgue measure in $\R^n$,  $S\backslash ( \cup_{j=1}^n\{\x\in\R^n\,:\,x_j=0\})$ also has nonempty interior. 
Then there exists $\mathbf{a}\in  S\backslash ( \cup_{j=1}^n\{\x\in\R^n\,:\,x_j=0\})$ such that $g_i(\mathbf{a})>0$, $i\in[m]$.
Let $\mathbf{b}=(\sqrt{|a_1|},\dots,\sqrt{|a_n|})$.
Then $\mathbf{b}\in (0,\infty)^n$ and $\mathbf b^2=({|a_1|},\dots,{|a_n|})$.
Since each $g_i$ is even in each variable, $\tilde g_i(\mathbf{b})=g_i(\mathbf{b}^2)=g_i(\mathbf{a})>0$, $i\in[m]$, yielding $\tilde S$ has nonempty interior.
\paragraph{Step 2: Construction of the positive weight functions.}
We process similarly to the proof of \cite[Theorem 1]{mai2022complexity} (see \cite[Appendix A.2.1]{mai2022complexity}) to obtain functions  $\bar\varphi_{j}:\R^n\to\R$, $j\in[\bar m]$, such that,
\begin{enumerate}
    \item $\bar\varphi_{j}$ is positive and bounded from above by  $C_{\bar\varphi_{j}}=\bar r_j\varepsilon^{-r_j}$  on $ B(\mathbf 0,\sqrt{n}+j)$ for some positive constants $\bar r_j$ and $r_j$ independent of $\varepsilon$.
    \item $\bar\varphi_{j}$ is Lipschitz with Lipschitz constant $L_{\bar\varphi_{j}}=\bar t_j\varepsilon^{-t_j}$ for some positive constants $\bar t_j$ and $t_j$ independent of $\varepsilon$.
    \item The inequality
    \begin{equation}\label{eq:nonega2}
    \tilde f+\varepsilon - \sum_{i=1}^{\bar m} \bar\varphi_{i}^2\tilde g_{i}\ge \frac{\varepsilon}{2^{\bar m}}\text{ on }[-1,1]^n\,,
\end{equation}
holds.
\end{enumerate}
Note that we do not need to prove the even property for each weight $\bar \varphi_i$ above.
\paragraph{Step 3: Approximating with Bernstein polynomials.}
For each $i\in[\bar m]$, we now approximate $\bar \varphi_i$ on $[-1,1]^n$ with the following Bernstein polynomials defined as in \cite[Definition 1]{mai2022complexity}: 
\begin{equation}
    B_i^{(d)}(\x)=B_{\y\mapsto \bar \varphi_i(2\y-\e),d\e}\left(\frac{\x+\e}{2}\right)\,,\quad d\in\N\,,
\end{equation}
with $\e=(1,\dots,1)\in\R^n$.
By using \cite[Lemma 6]{mai2022complexity}, for all $\x\in [-1,1]^n$, for $i\in[\bar m]$,
\begin{equation}
    |B_i^{(d)}(\x) - \bar \varphi_i(\x) | \le
{L_{\bar \varphi_i}} \biggl(\frac{n}{d} \biggr)^{\frac{1}2}\,,\quad d\in\N\,,
\end{equation}
 and for all $\x\in [-1,1]^n$, for $i\in[\bar m]$:
\begin{equation}
\begin{array}{l}
    |B_i^{(d)}(\x)|\le  \sup_{\x\in[-1,1]^n}|\bar\varphi_i(\x)|\le C_{\bar\varphi_i}\,.
\end{array}
\end{equation}

For $i\in[\bar m]$, let 
\begin{equation}
    d_i:=2u_i\quad\text{with}\quad u_i=\Bigl\lceil{  \frac{2C_{\tilde g_i}^2C_{\bar \varphi_i}^2n L_{\bar \varphi_i}^2 (\bar m+1)^22^{2\bar m}}{\varepsilon^2}} \Bigr\rceil\,,
\end{equation}
 where $C_{\tilde g_i}$ is an upper bound of $|\tilde g_i|$ on ${B(\mathbf 0,\sqrt{n}+i)}$.
Set $q_i:=B_i^{(d_i)}$, $i\in[\bar m]$.
Then for all $\x\in[-1,1]^n$, 
\begin{equation}
\begin{array}{rl}
   |q_i(\x) - \bar \varphi_i(\x) |
   &= |B_i^{(d_i)}(\x) - \bar \varphi_i(\x) |  \\
   & \le
{L_{\bar \varphi_i}} \left(\frac{n}{d_i} \right)^{\frac{1}2}   \\
&\le
{L_{\bar \varphi_i}} \left(\frac{n}{ \frac{4C_{\tilde g_i}^2C_{\bar \varphi_i}^2n L_{\bar \varphi_i}^2 (\bar m+1)^22^{2\bar m}}{\varepsilon^2}} \right)^{\frac{1}2}\\
& = \frac{\varepsilon}{2C_{\tilde g_i}C_{\bar \varphi_i}(\bar m+1)2^{\bar m}}\,.
\end{array}
\end{equation}
\paragraph{Step 4: Estimating the lower and upper bounds of $\tilde f(\x)+\varepsilon - \sum_{i=1}^{\bar m} q_{i}(\x)^2\tilde g_{i}(\x)$ on $\Delta_n$.}
From these and \eqref{eq:nonega2}, for all $\x\in \Delta_n$, \begin{equation}
    \begin{array}{rl}
     & \tilde f(\x)+\varepsilon - \sum_{i=1}^{\bar m} q_{i}(\x)^2\tilde g_{i}(\x)\\
       =&\tilde f(\x)+\varepsilon - \sum_{i=1}^{\bar m} \bar\varphi_{i}(\x)^2\tilde g_{i}(\x) + \sum_{i=1}^{\bar m} \tilde g_{i}(\x)[\bar\varphi_{i}(\x)^2-q_{i}(\x)^2]\\
          \ge& \frac{\varepsilon}{2^{\bar m}} - \sum_{i=1}^{\bar m} |\tilde g_{i}(\x)||\bar\varphi_{i}(\x)+q_{i}(\x)||\bar\varphi_{i}(\x)-q_{i}(\x)|\\
         \ge& \frac{\varepsilon}{2^{\bar m}} - \sum_{i=1}^{\bar m} C_{\tilde g_{i}}(|\bar\varphi_{i}(\x)|+|q_{i}(\x)|)\frac{\varepsilon}{2C_{\tilde g_i}C_{\bar \varphi_i}(\bar m+1)2^{\bar m}}\\
         \ge &\frac{\varepsilon}{2^{\bar m}} - \sum_{i=1}^{\bar m} 2C_{\tilde g_{i}}C_{\bar\varphi_{i}}\frac{\varepsilon}{2C_{\tilde g_i}C_{\bar \varphi_i}(\bar m+1)2^{\bar m}}\\
         =& \frac{\varepsilon}{2^{\bar m}} - \frac{\bar m\varepsilon}{(\bar m+1)2^{\bar m}}=\frac{\varepsilon}{(\bar m+1)2^{\bar m}}\,.
    \end{array}
\end{equation}
Thus,
\begin{equation}\label{eq:lower.bound.tildef}
    \begin{array}{l}
         \tilde f+\varepsilon - \sum_{i=1}^{\bar m} q_{i}^2\tilde g_{i}\ge  \frac{\varepsilon}{(\bar m+1)2^{\bar m}}\text{ on } \Delta_n\,.
    \end{array}
\end{equation}
\paragraph{Step 5: Estimating the upper bound of $\|q_i\|$.}
For $i\in[\bar m]$, we write
\begin{equation}
\begin{array}{rl}
      q_i=B_i^{(2u_i)}
     =& \sum_{k_1=0}^{2u_i} \dots \sum_{k_n=0}^{2u_i}  \bar \varphi_i\left(\frac{k_1-u_i}{u_i},\dots,\frac{k_n-u_i}{u_i}\right)\\
     &\qquad\qquad\qquad\qquad
\times\prod_{j=1}^n \left[\binom{2u_i}{k_j} \left(\frac{x_j+1}2\right)^{k_j} \left(\frac{1-x_j}2\right)^{2u_i-k_j} \right]\,.
\end{array}
\end{equation}
Then
\begin{equation}
    \deg(q_i)\le 2nu_i\,,
\end{equation}
for $i\in[\bar m]$.
From \eqref{eq:mul.ineq}, we have
\begin{equation}
    \begin{array}{rl}
          \|q_i\| 
         \le & \sum_{k_1=0}^{2u_i} \dots \sum_{k_n=0}^{2u_i}  \left|\bar \varphi_i\left(\frac{k_1-u_i}{u_i},\dots,\frac{k_n-u_i}{u_i}\right)\right|\\
     &\qquad\qquad\qquad\qquad
\times
\prod_{j=1}^n \left[\binom{2u_i}{k_j} \frac{1}{2^{2u_i}}\|x_j+1\|^{k_j} \|1-x_j\|^{2u_i-k_j} \right]\\
\le & \sum_{k_1=0}^{2u_i} \dots \sum_{k_n=0}^{2u_i}  C_{\bar \varphi_i}\prod_{j=1}^n \left(\binom{2u_i}{u_i} \frac{1}{2^{2u_i}} \right)\\
=& C_{\bar \varphi_i} \left(\binom{2u_i}{u_i} \frac{2u_i+1}{2^{2u_i}}\right)^n \\
\le & C_{\bar \varphi_i} \left( \frac{2u_i+1}{\sqrt{\pi {u_i}}}\right)^n=:T_{q_i}\,.
    \end{array}
\end{equation}
The second inequality is due to $\|x_j+1\|= \|1-x_j\|=1$ and $\binom{2u_i}{u_i}\ge \binom{2u_i}{k_j}$, for $k_j=0,\dots,2u_i$.
The third inequality is implied from Lemma \ref{lem:bound.cetral.coe}.

\paragraph{Step 6: Converting to homogeneous polynomials.}
Thanks to \eqref{eq:lower.bound.tildef}, we get
\begin{equation}
\begin{array}{l}
  \tilde{f}+2\varepsilon -  \sum_{i\in[\bar m]} (q_{i}^2+\frac{\varepsilon}{\bar m C_{\tilde g_i}})\tilde{g}_{i}\ge \frac{\varepsilon}{(\bar m+1)2^{\bar m}}\quad\text{ on } \Delta_n\,,
\end{array}
\end{equation}
since $|\tilde g_i|\le C_{\tilde g_i}$ on $\Delta_n$.
Note that $\tilde{f},\tilde{g_i}$ are homogeneous polynomials of degree $d_f,d_{g_i}$, respectively. 

For each $q\in\R[\x]_d$,  $\hat q$ is a $d$-homogenization of $q$ if 
\begin{equation}
    \begin{array}{l}
         \hat q=\sum_{t=0}^{d}h^{(t)}(\sum_{j\in[n]}x_j)^{d-t}\,,
    \end{array}
\end{equation}
for some $h^{(t)}$ is the homogeneous polynomial of degree $t$ satisfying $q=\sum_{t=0}^{d}h^{(t)}$.
In this case, $\hat q=q$ on $\Delta_n$.

Let $p_i:=\hat q_{i}^2+\frac{\varepsilon}{\bar mC_{\tilde g_i}}(\sum_{j\in[n]}x_j)^{4nu_i}$ with $\hat q_i$ being a $2nu_i$-homogenization of $q_i$, for $i\in[\bar m]$.
Then $p_i$ is a homogeneous polynomial of degree $4nu_i$,
\begin{equation}
\begin{array}{l}
     p_i=q_{i}^2+\frac{\varepsilon}{\bar m C_{\tilde g_i}}\ge \frac{\varepsilon}{\bar m C_{\tilde g_i}}\text{ on }\Delta_n\,,
\end{array}
\end{equation}
and 
\begin{equation}
\begin{array}{l}
     \|p_i\|\le \|q_{i}\|^2+\frac{\varepsilon}{\bar m C_{\tilde g_i}}\le T_{q_i}^2+\frac{\varepsilon}{\bar m C_{\tilde g_i}}=:T_{p_i}\,.
\end{array}
\end{equation}
Set $D:=\max\{d_f,4nu_i +d_{g_i}\,:\,i\in[\bar m]\}$ and 
\begin{equation}\label{eq:rep.F}
\begin{array}{rl}
F:=&(\sum_{j\in[n]}x_j)^{D-d_f}(\tilde f+2\varepsilon(\sum_{j\in[n]}x_j)^{d_f})\\
&- \sum_{i\in[\bar m]} \tilde g_{i} p_i(\sum_{j\in[n]}x_j)^{D-4nu_i-d_{g_i}}\,.
\end{array}
\end{equation}
Then $F$ is a homogeneous polynomial of degree $D$ and 
\begin{equation}
\begin{array}{l}
  F=\tilde{f}+2\varepsilon -  \sum_{i\in[\bar m]} (q_{i}^2+\frac{\varepsilon}{\bar m C_{\tilde g_i}})\tilde{g}_{i}\ge \frac{\varepsilon}{(\bar m+1)2^{\bar m}}\quad\text{ on } \Delta_n\,,
\end{array}
\end{equation}
Moreover,
\begin{equation}
\begin{array}{rl}

\|F\|\le&\|\sum_{j\in[n]}x_j\|^{D-d_f}(\|\tilde f\|+2\varepsilon\|\sum_{j\in[n]}x_j\|^{d_f})\\
&+ \sum_{i\in[\bar m]} \|\tilde g_{i}\| \|p_i\|\|\sum_{j\in[n]}x_j\|^{D-4nu_i-d_{g_i}}\\
  \le & \|\tilde{f}\|+2\varepsilon +  \sum_{i\in[\bar m]} T_{p_i}\|\tilde{g}_{i}\|=:T_F\,,
\end{array}
\end{equation}
since $\|\sum_{j\in[n]}x_j\|=1$.
\paragraph{Step 7: Applying the degree bound of P\'olya's Positivstellensatz.}
Using Lemma \ref{lem:bound.polya}, we obtain:
\begin{itemize}
    \item For all $k\in\N$ satisfying 
\begin{equation}
    k\ge \frac{D(D-1)T_F}{\frac{\varepsilon}{(\bar m+1)2^{\bar m}}} =:K_0\,,
\end{equation}
$(\sum_{j\in[n]}x_j)^kF$ has positive coefficients.
\item For each $i\in[\bar m]$ and for all $k\in\N$ satisfying 
\begin{equation}
    k\ge \frac{4nu_i(4nu_i-1)T_{p_i}}{\frac{\varepsilon}{\bar mC_{\tilde g_i}}} =:K_i\,,
\end{equation}
$(\sum_{j\in[n]}x_j)^kp_i$ has positive coefficients.
\end{itemize}
Notice that $K_i$, $i=0,\dots,\bar m$, are obtained by composing finitely many times the following operators: ``$+$", ``$-$", ``$\times$", ``$\div$", ``$|\cdot|$", ``$\lceil\cdot\rceil$", ``$(x_1,x_2)\mapsto\max\{x_1,x_2\}$", ``$(x_1,x_2)\mapsto\min\{x_1,x_2\}$", ``$(\cdot)^{\alpha_m}$" and ``$\sqrt{\cdot}$", where all arguments possibly depend on $\varepsilon$.
Without loss of generality, let $\bar{\mathfrak{c}},\mathfrak{c}$ be positive constants independent of $\varepsilon$ such that $\bar{\mathfrak{c}}\varepsilon^{-\mathfrak{c}}\ge \max\{K_0,\dots,K_{\bar m}\}$.

Let $k\ge \bar{\mathfrak{c}}\varepsilon^{-\mathfrak{c}}$ be fixed.
Multiplying two sides of \eqref{eq:rep.F} with $(\sum_{j\in[n]}x_j)^k$, we get
\begin{equation}
  \begin{array}{rl}
s_0=&(\sum_{j\in[n]}x_j)^{D-d_f+k}(\tilde f+2\varepsilon(\sum_{j\in[n]}x_j)^{d_f})\\
&- \sum_{i\in[\bar m]} \tilde g_{i} s_i(\sum_{j\in[n]}x_j)^{D-4nu_i-d_{g_i}}\,,
\end{array}  
\end{equation}
where $s_0:=(\sum_{j\in[n]}x_j)^kF$ and $s_i:=(\sum_{j\in[n]}x_j)^kp_i$ are homogeneous polynomials having nonnegative coefficients.
Replacing $\x$ by $\x^2$, we obtain:
\begin{equation}
  \begin{array}{rl}
\|\x\|_2^{2(D-d_f+k)}( f+2\varepsilon\|\x\|_2^{2d_f})=\sigma_0+ \sum_{i\in[m]}  g_{i} \sigma_i\,,
\end{array}  
\end{equation}
where 
\begin{equation}
    \begin{array}{rl}
        \sigma_0&=s_0(\x^2)+ \sum_{j\in[n]} \tilde g_{j+m}(\x^2) s_{j+m}(\x^2)\|\x\|_2^{2(D-4nu_{j+m}-d_{g_{j+m}})}  \\
         &= s_0(\x^2)+\sum_{j\in[n]} x_j^2 s_{j+m}(\x^2)\|\x\|_2^{2(D-4nu_{j+m}-d_{g_{j+m}})} \,,
    \end{array}
\end{equation}
and
\begin{equation}
\begin{array}{l}
     \sigma_i=s_i(\x^2)\|\x\|_2^{2(D-4nu_i-d_{g_i})}\,,\,i\in[m]\,,
\end{array}
\end{equation}
are SOS of monomials.
Set $K=D-d_f+K$. 
Then $\|\x\|_2^{2K}(f+2\varepsilon\|\x\|_2^{2d_f})=\sigma_0+ \sum_{i=1}^m g_{i} \sigma_i$ with $\deg(\sigma_0)=\deg(g_i\sigma_i)=2(K+d_f)$, for $i\in[m]$.
This completes the proof of Theorem \ref{theo:complex.putinar.vasilescu.even}.
\end{proof}


\subsection{Variations of P\'olya's Positivstellensatz}
\label{sec:varia}
For every $t\in\N$, denote
\begin{equation}
    \begin{array}{rl}
         \bar \v_t(\x):=\v_{t}(\frac{1}{2}(\x+\e)\frac{1}{2}(\x-\e))=(\frac{1}{2^{|\a+\b|}}(\x+\e)^\a(\x-\e)^\b)_{(\a,\b)\in \N^{2n}_{t}}\,,
    \end{array}
\end{equation}
where $\e:=(1,\dots,1)\in\R^n$.

As a consequence of Corollary \ref{coro:compact.even}, the next proposition shows that the weighted SOS polynomials in Putinar--Vasilescu's Positivstellensatz can be associated with diagonal Gram matrices via a change of monomial basis.

\begin{proposition}(Putinar--Vasilescu's Positivstellensatz with diagonal Gram matrices)
\label{prop:rep.without.evan.PuVa}
Let $g_1,\dots,g_m$ be polynomials such that $g_1:=R-\|\x\|_2^2$ for some $R>0$ and $g_m:=1$.
Let $S$ be the semialgebraic set defined by
\begin{equation}
    S:=\{\x\in\R^n\,:\,g_1(\x)\ge 0\dots,g_m(\x)\ge 0\}\,.
\end{equation} 
Let $f$ be a polynomial of degree at most $2d_f$  such that $f$ is nonnegative on $S$.
Denote $d_{g_i}:=\lceil \deg(g_i)/2\rceil$.
Then the following statements hold:
\begin{enumerate}
    \item For all $\varepsilon>0$, there exists $K_\varepsilon\in\N$ such that for all $k\ge K_\varepsilon$, there exist vectors $\boldsymbol{\eta}^{(i)}\in\R_+^{b(2n,k+d_f-d_{g_i})}$ satisfying
\begin{equation}
\label{eq:rep.han.PuVa}
    \begin{array}{rl}
          (\|\x\|_2^2+n+2)^k(f+\varepsilon)=\sum_{i=1}^m    g_i\bar \v_{k+d_f-d_{g_i}}^\top  \diag(\boldsymbol{\eta}^{(i)}) \bar \v_{k+d_f-d_{g_i}}\,.
    \end{array}
\end{equation}
\item If $S$ has nonempty interior, then
there exist positive constants $\bar{\mathfrak{c}}$ and $\mathfrak{c}$ depending on $f,g_i$ such that for all $\varepsilon>0$,  one can take 
$K_{\varepsilon} = \bar{\mathfrak{c}}\varepsilon^{-\mathfrak{c}}$.
\end{enumerate} 
\end{proposition}
\begin{proof}
Take two vectors of $n$ variables $\y=(y_1,\dots,y_n)$ and $\z=(z_1,\dots,z_n)$. 
Given $q\in\R[\x]$, denote the polynomial $\hat q(\y,\z)=q(\y^2-\z^2)\in\R[\y,\z]$.
Let $\hat g_{m+1}:=\frac{1}{2}(L+n)-\|(\y,\z)\|_2^2$ and $d_{g_{m+1}}:=1$.
Define 
\begin{equation}
    \begin{array}{rl}
         \hat S:=\{(\y,\z)\in\R^{2n}\,:\,\hat g_i(\y,\z)\ge 0\,,\,i\in[m+1]\}
    \end{array}
\end{equation}
Note that $\hat g_1:=R-\|\y^2-\z^2\|_2^2$ and $\hat g_m:=1$.
Since $f\ge 0$ on $S$, replacing $\x$ by $\y^2-\z^2$ gives $\hat f\ge 0 \text{ on } \hat S$.
From this  and Corollary \ref{coro:compact.even}, there exist $\boldsymbol{\eta}^{(i)}\in\R_+^{b(2n,k+d_f-d_{g_i})}$ such that
\begin{equation}
    \begin{array}{rl}
         (\|(\y,\z)\|_2^2+1)^k(\hat f+\varepsilon)=\sum_{i=1}^{m+1} \hat g_i\v_{k+d_f-d_{g_i}}(\y,\z)^\top  \diag(\boldsymbol{\eta}^{(i)}) \v_{k+d_f-d_{g_i}}(\y,\z)\,.
    \end{array}
\end{equation}
With $\y=\frac{1}{2}(\x+\e)$ and $\z=\frac{1}{2}(\x-\e)$, it becomes
\begin{equation}\label{eq:near.rep}
    \begin{array}{rl}
          \frac{1}{2^k}(\|\x\|_2^2+n+2)^k(f+\varepsilon)=\sum_{i=1}^{m+1}  g_i  \bar \v_{k+d_f-d_{g_i}}^\top  \diag(\boldsymbol{\eta}^{(i)})\bar \v_{k+d_f-d_{g_i}} \,.
    \end{array}
\end{equation}
Here $g_{m+1}(\cdot):=\hat g_{m+1}(\frac{1}{2}(\cdot+\e),\frac{1}{2}(\cdot-\e))=\frac{1}{2}g_1(\cdot)$.
Indeed, since $\y^2-\z^2=\x$, $\hat f(\y,\z)=f(\x)$ and $\hat g_i(\y,\z)=g_i(\x)$, for $i\in[m]$.
Since $\y^2+\z^2=\frac{1}{2}(\x^2+\e)$, $\|\y\|_2^2+\|\z\|_2^2=\frac{1}{2}(\|\x\|_2^2+n)$.
This implies that
\begin{equation}
    \begin{array}{rl}
         \hat g_{m+1}(\y,\z)=\frac{1}{2}(L+n)-\|(\y,\z)\|_2^2=\frac{1}{2}(R-\|\x\|_2^2)=\frac{1}{2}g_1(\x)\,.
    \end{array}
\end{equation}
Moreover, if $\mathbf{a}$ belongs to the interior of $S$, then $(\frac{1}{2}(\mathbf{a}+\e),\frac{1}{2}(\mathbf{a}-\e))$ belongs to the interior of $\hat S$.
Thus, the desired result follows.
\end{proof}
\begin{remark}
In view of Propositions \ref{prop:rep.without.evan.PuVa}, replacing the standard monomial basis $\v_t$ by the new basis $\bar \v_t$ can provide a  Positivstellensatz involving  Gram matrix of factor width $1$.
Thus, ones can build up a hierarchy of semidefinite relaxations with any maximal matrix size, based on representation \eqref{eq:rep.han.PuVa}.
However, expressing the entries of the basis $\bar \v_t$ is a time-consuming task within the  modeling process.
A potential workaround is to impose \eqref{eq:rep.han.PuVa} on a set of generic
points similarly to \cite[Section 2.3]{lasserre2017bounded}. 
This needs further study.
\end{remark}

\subsection{Polynomial optimization on the nonnegative orthant: Noncompact case}
\label{sec:pop.noncompact}
\subsubsection{Linear relaxations}
Given $\varepsilon>0$, consider the hierarchy of linear programs indexed by $k\in\N$: 
\begin{equation}\label{eq:dual-sdp.even.LP}
\begin{array}{rl}
\tau_k^{\textup{\revise{P\'olya}}}{(\varepsilon)}: = \inf\limits_\y &{L_\y}( {{\theta ^k}( {\check f +  \varepsilon {\theta ^{d_f}}} )} )\\
\st&\y= {(y_\a )_{\a  \in \N^n_{2( {d_f + k})}}} \subset \R\,,\,{L_\y}( {{\theta ^k}}) = 1\,,\\
&\diag({\M_{k_i}}( {{\check g_i}\y}))\in \R_+^{b(n,k_i)}\,,\,i \in[m]\,,
\end{array}
\end{equation}
where $k_i:=k + d_f - d_{g_i}$, $i\in[m]$.
Here $\check g_m=1$.
Note that 
\begin{equation}
\begin{array}{rl}
     \diag({\M_{k_i}}( {{\check g_i}\y}))=(\sum_{\g\in\N^n_{2d_{g_i}}}y_{2\a+\g}\check g_{i,\g})_{\a\in\N^n_{k_i}}\,.
\end{array}
\end{equation}

\begin{theorem}\label{theo:constr.theo.even.LP}
Let $f,g_i\in\R[\x]$, $i\in[m]$, with $g_m=1$.
Consider POP \eqref{eq:constrained.problem.poly.even.LP} with $S$ being defined as in \eqref{eq:semial.set.def.2.even.LP}.
Let $\varepsilon>0$ be fixed. 
For every $k\in\N$, the dual of \eqref{eq:dual-sdp.even.LP} reads as:
\begin{equation}\label{eq:primal.problem.even.LP}
\begin{array}{rl}
   {\rho _k^{\textup{\revise{P\'olya}}}}(\varepsilon):= \sup\limits_{\lambda,\mathbf{u}_i} & \lambda\\
   \st & \lambda\in\R\,,\,\mathbf{u}_i\in\R_+^{b(n,k_i)}\,,\,i\in[m]\,,\\
   &\theta^k(\check f-\lambda+\varepsilon\theta^{d_f})=\sum_{i=1}^m \check g_i\v_{k_i}^\top \diag(\mathbf{u}_i) \v_{k_i}\,.
\end{array}
\end{equation}
Here $\check g_m=1$.
The following statements hold:
\begin{enumerate}
\item For all $k\in\N$,
$\rho_k^{\textup{\revise{P\'olya}}}{(\varepsilon)}\le\rho_{k+1}^{\textup{\revise{P\'olya}}}{(\varepsilon)}\le f^\star$.
\item 
There exists $K\in\N$ such that for all $k\ge K$, 
$0\le f^\star - \rho_k^{\textup{\revise{P\'olya}}}{(\varepsilon)} \le \varepsilon \theta {( {{\x^{\star2}}})^{d_f}}$.
\item If $S$ has nonempty interior, there exist positive constants $\bar{\mathfrak{c}}$ and $\mathfrak{c}$ depending on $f,g_i$ such that for all $k\ge \bar{\mathfrak{c}}\varepsilon^{-\mathfrak{c}}$, 
$0\le f^\star-\rho_k^{\textup{\revise{P\'olya}}}{(\varepsilon)} \le \varepsilon \theta {( {{\x^{\star2}}})^{d_f}}$.
\end{enumerate}
\end{theorem}
The proof of Theorem \ref{theo:constr.theo.even.LP} relies on Corollary \ref{coro:dehomo.even} and is exactly the same as the proof of \cite[Theorem 7]{mai2021positivity}.

\subsubsection{Semidefinite relaxations}
\label{sec:dense.POP.noncompact}

Given $\varepsilon>0$, consider the hierarchy of semidefinite programs indexed by $s\in\N_{> 0}$ and $k\in\N$: 
\begin{equation}\label{eq:dual-sdp.0.even.noncompact}
\begin{array}{rl}
{\tau_{k,s}^\textup{\revise{P\'olya}}(\varepsilon)}: = \inf\limits_\y &{L_\y}( {\theta ^k}(\check f +\varepsilon\theta^{d_f}) )\\
\st&\y= {(y_\a )_{\a  \in \N^n_{2( {d_f + k})}}} \subset \R\,,\,{L_\y}( {{\theta ^k}}) = 1\,,\\
&\M_{\cA^{(s,k_i)}_j}(\check g_i\y)\succeq 0\,,\,j\in[b(n,k_i)]\,,\,i \in[m]\,,
\end{array}
\end{equation}
where $k_i:=k + d_f - d_{g_i}$, $i\in[m]$. Here $\check g_m=1$.
\begin{theorem}\label{theo:constr.theo.0.even.noncompact}
Let $f,g_i\in\R[\x]$, $i\in[m]$, with $g_m=1$.
Consider POP \eqref{eq:constrained.problem.poly.even.LP} with $S$ being defined as in \eqref{eq:semial.set.def.2.even.LP}.
Let $\varepsilon>0$ be fixed. 
For every $k\in\N$ and for every $s\in\N_{> 0}$, the dual of \eqref{eq:dual-sdp.0.even.noncompact} reads as:
\begin{equation}\label{eq:primal.problem.0.even.noncompact}
\begin{array}{rl}
   {\rho _{k,s}^\textup{\revise{P\'olya}}}(\varepsilon):= \sup\limits_{\lambda,\mathbf{G}_{ij}} & \lambda\\
   \st& \lambda\in\R\,,\,\mathbf{G}_{ij}\succeq 0\,,\,j\in[b(n,k_i)]\,,\,i\in[m]\,,\\[5pt]
   &\theta^k(\check f-\lambda+\varepsilon\theta^{d_f})=\sum_{i\in[m]} \check g_i \big(\sum_{j\in[b(n,k_i)]}\v_{\cA^{(s,k_i)}_j}^\top \mathbf{G}_{ij} \v_{\cA^{(s,k_i)}_j}\big)\,.
\end{array}
\end{equation}
The following statements hold:
\begin{enumerate}
\item For all $k\in\N_{> 0}$ and for every $s\in\N_{> 0}$,
$\rho_{k}^\textup{\revise{P\'olya}}=\rho_{k,1}^\textup{\revise{P\'olya}}(\varepsilon)\le\rho_{k,s}^\textup{\revise{P\'olya}}(\varepsilon)$.
\item For every $s\in\N_{> 0}$, there exists $K\in\N$ such that for every $k\in\N$  satisfying $k\ge K$,
$0\le f^\star -\rho_{k,s}^\textup{\revise{P\'olya}}(\varepsilon)\le \varepsilon \theta {( {{\x^{\star2}}})^{d_f}}$.
\item If $S$ has nonempty interior, there exist positive constants $\bar{\mathfrak{c}}$ and $\mathfrak{c}$ depending on $f,g_i$ such that for every $s\in\N_{> 0}$ and for every $k\in\N$ satisfying $k\ge \bar{\mathfrak{c}}\varepsilon^{-\mathfrak{c}}$,
$0\le f^\star-\rho_{k,s}^\textup{\revise{P\'olya}}(\varepsilon) \le \varepsilon \theta {( {{\x^{\star2}}})^{d_f}}$.
\item If $S$ has nonempty interior, for every $s\in\N_{> 0}$ and for every $k\in\N$ strong duality holds for the  primal-dual problems \eqref{eq:dual-sdp.0.even.noncompact}-\eqref{eq:primal.problem.0.even.noncompact}.
\end{enumerate}
\end{theorem}
The proof of Theorem \ref{theo:constr.theo.0.even.noncompact} is based on Theorem \ref{theo:constr.theo.even.LP}, \cite[Theorem 3]{mai2022complexity} and the inequalities
$\rho_k^\textup{\revise{P\'olya}}(\varepsilon)\le \rho_{k,s}^\textup{\revise{P\'olya}}(\varepsilon) \le \rho_k^{(\varepsilon)}$,
where $\rho_k^{(\varepsilon)}$ is defined as in \cite[(113)]{mai2022complexity}.
For each $q\in\R[\x]_d$, denote   the degree-$d$ homogenization of $q$ by $x_{n+1}^dq(\frac{\x}{x_{n+1}})\in\R[\x,x_{n+1}]$.
\begin{remark}\label{rem:SDSOS}
Let $(\lambda,\mathbf{G}_{ij})$ be a feasible solution of \eqref{eq:primal.problem.0.even.noncompact} and consider the case of $m=1$. Then the equality constraint of \eqref{eq:primal.problem.0.even.noncompact} becomes
\begin{equation}
    \begin{array}{rl}
         \theta^k(\check f-\lambda+\varepsilon\theta^{d_f})= \sum_{j\in[b(n,k_m)]}\v_{\cA^{(s,k_m)}_j}^\top \mathbf{G}_{mj} \v_{\cA^{(s,k_m)}_j}\,.
    \end{array}
\end{equation}
It implies that the degree-$2d_f$ homogenization of $\check f-\lambda+\varepsilon\theta^{d_f}$ belongs to the cone $k$-$\text{DSOS}_{n+1,2d_f}$ (resp. $k$-$\text{SDSOS}_{n+1,2d_f}$) when $s=1$ (resp. $s=2$)  according to \cite[Definition  3.10]{ahmadi2019dsos}.
More generally, the polynomial $\theta^k(\check f-\lambda+\varepsilon\theta^{d_f})$ belongs to the cone of SOS polynomials whose Gram matrix has  factor width at most $s$ (see \cite[Section 5.3]{ahmadi2019dsos}).
\end{remark}

\subsection{Sparse representations: Extension of P\'olya's Positivstellensatz}
\label{sec:sparse.rep}

For every $I=\{i_1,\dots,i_r\}\subset[n]$ with $i_1<\dots<i_r$, denote $\x(I)=(x_{i_1},\dots,x_{i_r})$. 

We will make the following assumptions:
\begin{assumption}\label{ass:RIP}
With $p\in\N_{> 0}$, the following conditions hold:
\begin{enumerate}
    \item There exists $(I_{c})_{c\in[p]}$ being a  sequence of subsets of $[n]$  such that
$\cup_{c\in[p]}I_c=[n]$ and 
\begin{equation}\label{eq:RIP}
    \begin{array}{l}
          \forall c\in\{2,\dots,p\}\,,\,\exists r_c\in[c-1]\,:\,I_c\cap (\cup_{t=1}^{c-1}I_t)\subset I_{r_c}\,.
    \end{array}
\end{equation}
Denote $n_c:=|I_c|$, for $c\in[p]$.
\item With $m\in\N_{> 0}$ and $(g_i)_{i\in[m]}\subset \R[\x]$, there exists $(J_c)_{c\in[p]}$ being a finite sequence of subsets of $[m]$ such that $\cup_{c\in[p]}J_c=[m]$ and
\begin{equation}
    \begin{array}{l}
         \forall c\in[p]\,,\, (g_i)_{i\in J_c}\subset \R[\x(I_c)]\,.
    \end{array}
\end{equation}
\item For every $c\in[p]$, there exists $i_c\in J_c$ and $R_c>0$ such that 
\begin{equation}
    g_{i_c}:=R_c-\|\x(I_c)\|_2^2\,.
\end{equation}
\end{enumerate}
\end{assumption}
The condition \eqref{eq:RIP} is called the running intersection property (RIP).

Let $\theta_c:=1+\|\x(I_c)\|_2^2$, $c\in[p]$. 

We state the sparse representation in the following theorem:
\begin{theorem}\label{theo:complex.putinar.vasilescu.even.sparse}
Let $g_1,\dots,g_m$ be  polynomials  such that $g_1,\dots,g_m$ are even in each variable and Assumption \ref{ass:RIP} holds.
Let $S$ be the semialgebraic set defined by
\begin{equation}
    S:=\{\x\in\R^n\,:\,g_1(\x)\ge 0\dots,g_m(\x)\ge 0\}\,.
\end{equation}
Let $f=f_1+\dots+f_p$ be a polynomial such that $f$ is positive on $S$ and for every $c\in[p]$, $f_c\in\R[\x(I_c)]$ is  even in each variable.
Then there exist $d,k\in\N$, $h_c\in\R[\x(I_c)]$, $\sigma_{0,c},\sigma_{j,c}\in\R[\x(I_c)]$, for $j\in J_c$ and $c\in[p]$, such that the following conditions hold:
\begin{enumerate}
    \item The equality $f=h_1+\dots+h_p$ holds and $h_c$ is a polynomial of degree at most $2d$ which is even in each variable.
    \item For all $i\in J_c$ and  $c\in[p]$, $\sigma_{0,c},\sigma_{i,c}$ are SOS of monomials satisfying
    \begin{equation}\label{eq:degree.SOS.even.sparse}
    \deg(\sigma_{0,c})\le 2(k+d)\quad\text{and}\quad\deg(\sigma_{i,c}g_i)\le 2(k+d)
\end{equation}
 and 
\begin{equation}\label{eq:represent.even.sparse}
\begin{array}{l}
     \theta_c^{k}h_c=\sigma_{0,c}+\sum_{i\in J_c}\sigma_{i,c}g_i\,.
\end{array}
\end{equation}
\end{enumerate}
\end{theorem}
\begin{proof}
Let $\varepsilon>0$. 
Similarly as in Step 1 of the proof of Theorem \ref{theo:complex.putinar.vasilescu.even}, $\tilde f=\tilde f_1+\dots+\tilde f_m$ is positive on the semialgebraic set $\tilde S$ defined as in \eqref{eq:set.S.tilde}.
For every $c\in[p]$, let $\tilde J_c:=J_c\cup (m+I_c)$.
Recall that $\tilde g_{m+j}:=x_j$, $j\in[n]$.
By applying \cite[Lemma 4]{grimm2007note}, there exist polynomials $s_c,q_{i,c}\in\R[\x(I_c)]$, for $j\in \tilde J_c$ and $c\in[p]$, such that 
\begin{equation}
\begin{array}{l}
     \tilde f = \sum_{c=1}^p(s_c+ \sum_{i\in \tilde J_c} q_{i,c}^2 \tilde g_i)\,,
\end{array}
\end{equation}
and for all $c\in[p]$, $s_c$ is positive on the set 
\begin{equation}
    \begin{array}{l} \{\x(I_c)\in\R^{n_c}\,:\,x_j\ge 0\,,\,j\in I_c\,,\,\tilde g_{i_c}(\x)=R_c-\sum_{j\in I_c}x_j\ge 0\}\,.
    \end{array}
\end{equation}
Set $h_c:=s_c(\x^2)+ \sum_{i\in \tilde J_c} q_{i,c}(\x^2)^2 \tilde g_i(\x^2)$, $c\in[p]$.
Let $d\in\N$ such that $2d-1\ge \max\{\deg(h_c)\,:\,c\in[p]\}$.
Then $f=\sum_{c=1}^p h_c$ with $h_c\in\R[\x(I_c)]_{2d}$ being even in each variable and  positive on the semialgebraic set
\begin{equation}
    S_c:=\{\x(I_c)\in\R^{n_c}\,:\,g_i(\x)\ge 0\,,\,i\in J_c\}\,.
\end{equation}
Note that $g_{i_c}:=R_c-\|\x(I_c)\|_2^2$ with $i_c\in J_c$.
By applying Corollary \ref{coro:compact.even}, 
there exists $k_c\in\N$ such that for all $K\ge k_c$, there exist  $\sigma_{0,c},\sigma_{i,c}\in\R[\x(I_c)]$, $i\in J_c$, such that $\sigma_{0,c},\sigma_{i,c}$ are SOS of monomials satisfying
    \begin{equation}\label{eq:degree.SOS.even.sparse.proof}
    \deg(\sigma_{0,c})\le 2(K+d)\quad\text{and}\quad\deg(\sigma_{i,c}g_i)\le 2(K+d)
\end{equation}
for all $i\in J_c$, and
\begin{equation}\label{eq:represent.even.sparse.proof}
\begin{array}{l}
     \theta_c^{K} h_c=\sigma_{0,c}+\sum_{i\in J_c}\sigma_{i,c}g_i\,.
\end{array}
\end{equation} 
Set $k=\max\{k^{(c)}\,:\,c\in[p]\}$.
Finally, we obtain the desired results.
\end{proof}
\begin{remark}
 In Theorem \ref{theo:complex.putinar.vasilescu.even.sparse}, it is not hard to see that $f$ has a rational SOS decomposition
\begin{equation}
    \begin{array}{rl}
         f=\sum_{c\in[p]} \frac{\sigma_{0,c}+\sum_{i\in J_c}\sigma_{i,c}g_i}{\theta_c^{k}}\,.
    \end{array}
\end{equation}
This decomposition is simpler than the ones provided in \cite{mai2022sparse} and thus is more applicable to polynomial optimization.

Another sparse representation without denominators can be found in the next theorem. However, the number of SOS of monomials is not fixed in this case. 
\end{remark}

\subsection{Sparse polynomial optimization on the nonnegative orthant}
\label{sec:sparse.POP.nonneg}

Consider the following POP:
\begin{equation}\label{eq:constrained.problem.poly.even.sparse.linear}
\begin{array}{l}
f^\star:=\inf\limits_{\x\in S} f(\x)\,,
\end{array}
\end{equation}
where $f\in\R[\x]$ and
\begin{equation}\label{eq:semial.set.def.2.even.sparse.linear}
    S=\{\x\in\R^n\,:\,x_j\ge 0\,,\,j\in[n]\,,\,g_i(\x)\ge 0\,,\,i\in[m]\}\,,
\end{equation}
for some $g_i\in\R[\x]$, $i\in[m]$, with $g_m=1$.
 Assume that $f^\star>-\infty$  and problem \eqref{eq:constrained.problem.poly.even.sparse.linear} has an optimal solution $\x^\star$. 
 
Then POP \eqref{eq:constrained.problem.poly.even.sparse.linear} is equivalent to
\begin{equation}\label{eq:equi.POP.sparse}
    f^\star:=\inf_{\x\in \check S}\check f\,,
\end{equation}
where
\begin{equation}
    \check S=\{\x\in\R^n\,:\,\check g_i(\x)\ge 0\,,\,i\in[m]\}\,,
\end{equation}
with optimal solution $\x^{\star2}$.

We will make the following assumptions:
\begin{assumption}\label{ass:sparsePOP}
With $p\in\N_{> 0}$, the first two conditions of Assumption \ref{ass:RIP} and the following conditions hold:
\begin{enumerate}
\if{
    \item There exits $(I_{c})_{c\in[p]}$ being a finite sequence of subsets of $[n]$  such that $\cup_{c\in[p]}I_c=[n]$ and 
    \begin{equation}\label{eq:RIP.POP}
    \begin{array}{l}
          \forall c\in\{2,\dots,p\}\,,\,\exists r_c\in[c-1]\,:\,I_c\cap (\cup_{t=1}^{c-1}I_t)\subset I_{r_c}\,.
    \end{array}
\end{equation}
Denote $n_c:=|I_c|$, for $c\in[p]$.
    \item There exits $(J_c)_{c\in[p]}$ being a finite sequence of subsets of $[m]$ such that $\cup_{c\in[p]}J_c=[m]$ and
    \begin{equation}
        \forall c\in[p]\,,\, m\in J_c\,,\,  \text{ and }g_i\in \R[\x(I_c)]\,,\, i\in J_c\,.
    \end{equation}
}\fi
    \item For every $c\in[p]$, there exist $ i_c\in J_c$ and $R_c>0$ such that 
    \begin{equation}
    \begin{array}{rl}
         g_{i_c}=R_c-\sum_{j\in I_c}x_j\,.
    \end{array}
    \end{equation}
    \item There exist  $f_c\in\R[\x(I_c)]$, for $c\in[p]$, such that $f=f_1+\dots+f_p$.
\end{enumerate}
\end{assumption}

\if{
The second condition of Assumption \ref{ass:sparsePOP} implies  that for every $c\in[p]$, $m\in J_c$,  $\check g_i\in \R[\x(I_c)]$ for all $i\in J_c$.
The third condition of Assumption \ref{ass:sparsePOP} implies that $\check g_{i_c}=R_c-\|\x(I_c)\|_2^2$.
Moreover, the final condition of Assumption \ref{ass:sparsePOP} yields
$\check f=\check f_1+\dots+\check f_p$ and for every $c\in[p]$, $\check f_c\in\R[\x(I_c)]$.

For every $\a=(\alpha_1,\dots,\alpha_n)\in\N^n$, denote
$\supp(\a):=
    \{i\in[n]\,:\,\alpha_i>0\}$.
For every $I\subset [n]$, denote $\N^{I}:=\{\a\in\N^n\,:\,\supp(\a)\subset I\}$. 
Let $\N^{I}_d:=\N^{I}\cap \N^n_d$.
For every $I\subset [n]$, with  $q = \sum_{\g\in \N^I} q_\g \x^\g  \in \R[\x(I)]$ and $\y=(y_\a)_{\a\in\N^n}\subset \N$, the localizing submatrix $\M_t(q\y,I)$ of degree $t$ associated to $\y$ and $q$ is the real symmetric matrix of the size $b(|I|,t):=\binom{|I|+t}{t}$ given by $\M_t(q\y,I) := (\sum_{\g\in\N^I}  {{q_\g }{y_{\g  + \a+\b }}})_{\a, \b\in \N^I_t}$.
Denote $d_{g_i}: =  {\deg ( {{g_i}} )} $, for $i\in[m]$.
}\fi

\subsubsection{Linear relaxations based on the extension of P\'olya's Positivstellensatz}
\label{sec:linear.CS}
Consider the hierarchy of linear programs indexed by $k,d\in\N$: 
\begin{equation}\label{eq:dual.problem.0.even.sparse.linear}
\begin{array}{rl}
   \tau _{k,d}^{\textup{\revise{SparseP\'olya}}}:= \inf\limits_{\y,\y^{(t)}} & L_\y(\check f)\\
   \st& \y= {(y_\a )_{\a  \in \N^n_{2d}}} \subset \R\,,\,\y^{(c)} = {(y_\a^{(c)} )_{\a  \in \N^n_{2( {d + k})}}} \subset \R\,,\,c\in[p]\,,\\
   & \diag(\M_d(\y,I_c))=\diag(\M_d(\theta_c^k \y^{(c)},I_c))\,,\,c\in[p]\,,\\
   &\diag(\M_{k^{(d)}_i}(\check g_i \y^{(c)},I_c))\in \R_+^{b(n_c,k_i^{(d)})}\,,\,i\in[m]\,,\,c\in[p]\,,\,y_\mathbf{0}=1\,,
\end{array}
\end{equation}
where $k_i^{(d)}:=k+d-d_{g_i}$.
\begin{theorem}\label{theo:constr.theo.0.even.sparse.linear}
Let $f,g_i\in\R[\x]$, $i\in[m]$, with $g_m=1$.
Consider POP \eqref{eq:constrained.problem.poly.even.sparse.linear} with $S$ being defined as in \eqref{eq:semial.set.def.2.even.sparse.linear}.
Let Assumption \ref{ass:sparsePOP} hold.
The dual of SDP \eqref{eq:dual.problem.0.even.sparse.linear} reads as:
\begin{equation}\label{eq:primal.problem.0.even.sparse.linear}
\begin{array}{rl}
   \rho _{k,d}^{\textup{\revise{SparseP\'olya}}}:= \sup\limits_{\lambda,\mathbf{u}_c,\mathbf{w}_{i}^{(c)}} & \lambda\\
   \st& \lambda\in\R\,,\,\mathbf{u}_c\in\R^{b(n_c,d)}\,,\,\mathbf{w}_{i}^{(c)}\in \R_+^{b(n_c,k_i^{(d)})}\,,\,i\in J_c\,,\,c\in[p]\,,\\
    &\check f-\lambda=\sum_{c\in[p]}h_c\,,\,h_c=\v_{\N^{I_c}_d}^\top\diag(\mathbf{u}_c)\v_{\N^{I_c}_d}\,,\,c\in[p]\,,\\
   &\theta_c^kh_c=\sum_{i\in J_c} \check g_i\v_{\N^{I_c}_{k^{(d)}_i}}^\top \diag(\mathbf{w}_{i}^{(c)}) \v_{\N^{I_c}_{k^{(d)}_i}}\,,\,c\in[p]\,.
\end{array}
\end{equation}
The following statements hold:
\begin{enumerate}
\item For all $k\in\N$ and for every $s\in\N_{> 0}$,
$\rho_{k-1,d}^{\textup{\revise{SparseP\'olya}}}\le\rho_{k,d}^{\textup{\revise{SparseP\'olya}}}\le\rho_{k,d+1}^{\textup{\revise{SparseP\'olya}}}\le f^\star$.
\item  One has
\begin{equation}\label{eq:sup.fstar}
    \sup\{\rho_{k,d}^{\textup{\revise{SparseP\'olya}}}\,:\,(k,d)\in\N^2\}=f^\star\,.
\end{equation}
\end{enumerate}
\end{theorem}
\begin{proof}
It is fairly easy to see that the first statement holds.
Let us prove the second one.
Let $\check g_{i_c}:=R_c-\|\x(I_c)\|_2^2$ and 
 $\varepsilon>0$.
Then $\check f-(f^\star-\varepsilon)>0$ on $S$.
By applying Theorem \ref{theo:complex.putinar.vasilescu.even.sparse}, 
there exist $d,k\in\N$, $h_c\in\R[\x(I_c)]$, $\sigma_{0,c},\sigma_{j,c}\in\R[\x(I_c)]$, for $j\in J_c$ and $c\in[p]$, such that the following conditions hold:
\begin{enumerate}
    \item The equality $\check f-(f^\star-\varepsilon)=h_1+\dots+h_p$ holds and $h_c$ is a polynomial of degree at most $2d$ which is even in each variable.
    \item For all $i\in J_c$ and  $c\in[p]$, $\sigma_{0,c},\sigma_{i,c}$ are SOS of monomials satisfying
    \begin{equation}\label{eq:degree.SOS.even.sparse.proof2}
    \deg(\sigma_{0,c})\le 2(k+d)\quad\text{and}\quad\deg(\sigma_{i,c}\check g_i)\le 2(k+d)
\end{equation}
 and 
\begin{equation}\label{eq:represent.even.sparse.proof2}
\begin{array}{l}
     \theta_c^{k}h_c=\sigma_{0,c}+\sum_{i\in J_c}\sigma_{i,c}\check g_i\,.
\end{array}
\end{equation}
\end{enumerate}
It implies that there exists $\mathbf{u}_c\in\R^{b(n_c,d)}$, $\mathbf{w}_{i}^{(c)}\in \R_+^{b(n_c,k_i^{(d)})}$ such that
\begin{equation}
    h_c=\v_{\N^{I_c}_d}^\top\diag(\mathbf{u}_c)\v_{\N^{I_c}_d}
\quad\text{and}\quad
    \sigma_{i,c}:=\v_{\N^{I_c}_{k^{(d)}_i}}^\top \diag(\mathbf{w}_{i}^{(c)}) \v_{\N^{I_c}_{k^{(d)}_i}}\,,
\end{equation}
for $i\in J_c$ and $c\in[p]$.
It implies that $(f^\star-\varepsilon,\mathbf{u}_c,\mathbf{w}_{i}^{(c)})$ is an optimal solution of LP \eqref{eq:primal.problem.0.even.sparse.linear}.
Thus $\rho _{k,d}^{\textup{\revise{SparseP\'olya}}}\ge f^\star -\varepsilon$, yielding \eqref{eq:sup.fstar}.
\end{proof}

\subsubsection{Semidefinite relaxations based on the extension of P\'olya's Positivstellensatz}
\label{sec:interrup.CS}

For every $I\subset [n]$, we write $\N^{I}=\{\a_{1}^{(I)},\a_{2}^{(I)},\dots,\a_{r}^{(I)},\a_{r+1}^{(I)},\dots\}$ such that 
\begin{equation}
    \a_{1}^{(I)}< \a_{2}^{(I)}< \dots< \a_{r}^{(I)}<\a_{r+1}^{(I)}<\dots\,.
\end{equation}
Let 
\begin{equation}
    W_{j}^{(I)}:=\{i\in\N\,:\,i\ge j\,,\,\a_i^{(I)}+\a_j^{(I)}\in2\N^{I}\}\,,\quad j\in\N_{> 0}\,,\,I\subset [n]\,.
\end{equation}
Then for all $j\in\N_{> 0}$ and for all $I\subset [n]$, $W_j^{(I)}\ne \emptyset$ since $j\in W_j^{(I)}$.
For every $j\in\N$ and for every $I\subset [n]$, we write $W_j^{(I)}:=\{i^{(j)}_{1,I},i^{(j)}_{2,I},\dots\}$ such that $i^{(j)}_{1,I}<i^{(j)}_{2,I}<\dots$.
Let 
\begin{equation}
    \mathcal T_{j,I}^{(s,d)}=\{\a_{i^{(j)}_{1,I}}^{(I)},\dots,\a_{i^{(j)}_{s,I}}^{(I)}\}\cap \N^{I}_d\,,\quad I\subset[n]\,,\, j,s\in\N_{> 0}\,,\,d\in\N\,.
\end{equation}
For every $s\in \N_{> 0}$, for every $d\in\N$ and for every $I\subset[n]$, define $\mathcal A^{(s,d)}_{1,I}:=\mathcal T_{1,I}^{(s,d)}$ and for $j=2,\dots,b(|I|,d)$, define
\begin{equation}
    \mathcal A^{(s,d)}_{j,I}:=\begin{cases}
         \mathcal {\cal T}_{j,I}^{(s,d)}&\text{if }T_{j,I}^{(s,d)}\backslash \mathcal A^{(s,d)}_{l,I}\ne \emptyset\,,\,\forall l\in[j-1]\,,\\
         \emptyset & \text{otherwise}\,.
         \end{cases}
\end{equation}
Note that $\cup_{j=1}^{b(|I|,d)} \mathcal A^{(s,d)}_{j,I} = \N^{I}_d$ and $| \mathcal A^{(s,d)}_{j,I}|\le s$.
Then the sequence
\begin{equation}
    (\a+\b)_{\big(\a,\b\in\mathcal A^{(s,d)}_{j,I}\big)}\,,\,j\in[b(|I|,d)]
\end{equation}
are overlapping blocks of size at most $s$ in 
$(\a+\b)_{(\a,\b\in\N^{I}_d)}$.
\begin{example}
Consider the case of $n=d=s=2$, $I_1=\{1\}$ and $I_2=\{2\}$.
Matrix $(\a+\b)_{(\a,\b\in\N^2_2)}$ is written explicitly as in \eqref{eq:mat.exam}.
We obtain two blocks:
\begin{equation}
    (\a+\b)_{(\a,\b\in\N^{I_1}_2)}=\begin{bmatrix}
{\bf(0,0)} & (1,0) & {\bf(2,0)}  \\
(1,0) & {\bf(2,0)} & (3,0)  \\
{\bf(2,0)} & (3,0) & {\bf(4,0)} \\
\end{bmatrix}
\end{equation}
and 
\begin{equation}
    (\a+\b)_{(\a,\b\in\N^{I_2}_2)}=\begin{bmatrix}
{\bf(0,0)}  & (0,1)  & {\bf(0,2)}\\
(0,1)  & {\bf(0,2)}  & (0,3)\\
{\bf(0,2)} & (0,3) & {\bf(0,4)}
\end{bmatrix}
\end{equation}
Then  $\mathcal A^{(2,2)}_{1,I_1}=\{{(0,0)}, {(2,0)}\}$, $\mathcal A^{(2,2)}_{2,I_1}=\{{(1,0)}\}$, $\mathcal A^{(2,2)}_{3,I_1}=\emptyset$ and  $\mathcal A^{(2,I_2)}_{1,2}=\{{(0,0)}, (0,2)\}$, $\mathcal A^{(2,2)}_{2,I_2}=\{{(0,1)}\}$, $\mathcal A^{(3,2)}_{2,I_2}=\emptyset$.
\end{example}

For every $I\subset [n]$, with $\mathcal{B}=\{\b_1,\dots,\b_r\}\subset \N^I$ such that $\b_1<\dots<\b_r$, for every $h=\sum_{\g\in\N^I}h_\g \x^\g\in\R[\x(I)]$ and  $\y=(y_\a)_{\a\in\N^n}\subset \R$, denote $\M_{\mathcal{B}}(h\y,I):=(\sum_{\g\in\N^I}h_\g y_{\g+\b_i+\b_j})_{i,j\in[r]}$.

Consider the hierarchy of linear programs indexed by $k,d\in\N$ and $s\in\N_{> 0}$: 
\begin{equation}\label{eq:dual.problem.0.even.sparse}
\begin{array}{rl}
   \tau _{k,d,s}^\textup{\revise{SparseP\'olya}}:= \inf\limits_{\y,\y^{(c)}} & L_\y(\check f)\\
   \st& \y= {(y_\a )_{\a  \in \N^n_{2d}}} \subset \R\,,\,\y^{(c)} = {(y_\a^{(c)} )_{\a  \in \N^n_{2( {d + k})}}} \subset \R\,,\,c\in[p]\,,\\
   & \diag(\M_d(\y,I_c))=\diag(\M_d(\theta_c^k \y^{(c)},I_c))\,,\,c\in[p]\,,\,y_\mathbf{0}=1\,,\\
   &\M_{\cA^{(s,k_i^{(d)})}_{j,I_c}}(\check g_i \y^{(c)},I_c)\succeq 0\,,\, j\in [b(n_c,k_i^{(d)})]\,,\,i\in J_c\,,\,c\in[p]\,,
\end{array}
\end{equation}
where $k_i^{(d)}:=k+d-d_{g_i}$.
\begin{theorem}\label{theo:constr.theo.0.even.sparse}
Let $f,g_i\in\R[\x]$, $i\in[m]$, with $g_m=1$.
Consider POP \eqref{eq:constrained.problem.poly.even.sparse.linear} with $S$ being defined as in \eqref{eq:semial.set.def.2.even.sparse.linear}.
Let Assumption \ref{ass:sparsePOP} hold.
The dual of SDP \eqref{eq:dual.problem.0.even.sparse} reads as: 
\begin{equation}\label{eq:primal.problem.0.even.sparse}
\begin{array}{rl}
   {\rho _{k,d,s}^\textup{\revise{SparseP\'olya}}}:= \sup\limits_{\lambda,\mathbf{u}_c,\mathbf{G}_{i,j}^{(c)}} & \lambda\\
   \st& \lambda\in\R\,,\,\mathbf{u}_c\in\R^{b(n_c,d)}\,,\,\mathbf{G}_{i,j}^{(c)}\succeq 0\,,\,j\in[b(n_c,k_i^{(d)})]\,,\,i\in J_c\,,\,c\in[p]\,,\\[5pt]
    &\check f-\lambda=\sum_{c\in[p]}h_c\,,\,h_c=\v_{\N^{I_c}_d}^\top\diag(\mathbf{u}_c)\v_{\N^{I_c}_d}\,,\,c\in[p]\,,\\
   &\theta_c^kh_c=\sum_{i\in J_c} \check g_i \Big(\sum_{j\in[b(n_c,k_i^{(d)})]}\v_{\cA^{(s,k_i^{(d)})}_{j,I_c}}^\top \mathbf{G}_{i,j}^{(c)} \v_{\cA^{(s,k_i^{(d)})}_{j,I_c}}\Big)\,,\,c\in[p]\,.
\end{array}
\end{equation}
The following statements hold:
\begin{enumerate}
\item For all $k,d\in\N$ and for every $s\in\N_{> 0}$,
$\rho_{k,d}^\textup{\revise{SparseP\'olya}}=\rho_{k,d,1}^\textup{\revise{SparseP\'olya}}\le\rho_{k,d,s}^\textup{\revise{SparseP\'olya}}\if{\le\rho_{k,d+1,s}^{\revise{SparseP\'olya}}}\fi\le f^\star$.
\item For every $s\in\N_{> 0}$, 
$\sup\{\rho_{k,d,s}^\textup{\revise{SparseP\'olya}}\,:\,(k,d)\in\N^2\}=f^\star$.
\end{enumerate}
\end{theorem}
\begin{proof}
It is not hard to prove the first statement. The second one is due to the second statement of Theorem \ref{theo:constr.theo.0.even.sparse.linear} and the inequalities
$\rho_{k,d}^{\textup{\revise{SparseP\'olya}}}\le\rho_{k,d,s}^\textup{\revise{SparseP\'olya}}\le f^\star$.
\end{proof}

\subsubsection{Obtaining an optimal solution}
In other to extract an optimal solution of POP \eqref{eq:constrained.problem.poly.even.sparse.linear} with correlative sparsity, we first extract atoms on each clique similarly to Algorithm \ref{alg:extract.sol} and then connect them together to obtain atoms in $\R^n$.
Explicitly, we  use  heuristic  extraction  Algorithm \ref{alg:extract.sol.sparse}. 
\begin{algorithm}
\caption{Extraction  algorithm for sparse POPs on the nonnegative orthant}
\label{alg:extract.sol.sparse}
\textbf{Input:} precision parameter $\varepsilon>0$ and an optimal solution $(\lambda,\mathbf{u}_c,\mathbf{G}_{i,j}^{(c)})$ of SDP \eqref{eq:primal.problem.0.even.sparse}.\\
\textbf{Output:} an optimal solution $\x^\star$ of POP \eqref{eq:constrained.problem.poly.even.sparse.linear}.
\begin{algorithmic}[1]
    \State For $c\in[p]$, do:
    \begin{algsubstates}
    \State For $j\in[b(n_c,k_m^{(d)})]$, let $\bar{\mathbf{G}}_{j}^{(c)}=(w_{\mathbf p\mathbf q}^{(c,j)})_{\mathbf p,\mathbf q\in \N^{I_c}_{k_m^{(d)}}}$ such that $(w_{\mathbf p\mathbf q}^{(c,j)})_{\mathbf p,\mathbf q\in \cA^{(s,k_m^{(d)})}_{j,I_c}}=\mathbf{G}_{m,j}^{(c)}$ and $w_{\mathbf p\mathbf q}^{(c,j)}=0$ if $(\mathbf p,\mathbf q)\notin (\cA^{(s,k_m^{(d)})}_{j,I_c})^2$.
    Then $\bar{\mathbf{G}}_{j}^{(c)}\succeq 0$ and
    \begin{equation}
    \begin{array}{r}
         \v_{\N^{I_c}_{k_m^{(d)}}}^\top \bar{\mathbf{G}}_{j}^{(c)} \v_{\N^{I_c}_{k_m^{(d)}}}=\v_{\mathcal A^{(s,k_m^{(d)})}_{j,I_c}}^\top \mathbf{G}_{m,j}^{(c)} \v_{\mathcal A^{(s,k_m^{(d)})}_{j,I_c}}\,;
    \end{array}
    \end{equation}
    \State Let  $\mathbf{G}^{(c)}:=\sum_{j\in[b(n_c,k_m^{(d)})]}\bar{\mathbf{G}}_j^{(c)}$. Then $\mathbf{G}^{(c)}$ is the Gram matrix corresponding to $\sigma_{m,c}$ in the rational SOS decomposition
    \begin{equation}
    \begin{array}{rl}
         \check f-\lambda=\sum_{c\in[p]} \frac{\sum_{i\in J_c}\sigma_{i,c}\check g_i}{\theta_c^{k}}\,.
    \end{array}
\end{equation}
    where each $\sigma_{i,c}$ is an SOS polynomial and $ \check g_m=1$;
    \State Obtain an atom $ \z^{\star (c)}\in\R^{n_c}$ by using the extraction algorithm of Henrion and Lasserre in  \cite{henrion2005detecting}, where the matrix $\mathbf V$ in \cite[(6)]{henrion2005detecting} is taken such that the columns of $\mathbf V$ form a basis of the null space $\{ \mathbf{u}\in\R^{\omega_k}\,:\, {\mathbf{G}^{(c)}}\mathbf  u=0\}$;   
\end{algsubstates}
\State Let $\z^\star\in\R^n$ such that $\z^\star(I_c)=\z^{\star (c)}$, for $c\in[p]$.
    \State If $ \z^\star$ exists, verify that $ \z^\star$ is  an approximate optimal solution of POP \eqref{eq:equi.POP.sparse} by checking the following inequalities:
    \begin{equation}\label{eq:check.sol.sparse}
        |\check f( \z^\star)-\lambda|\le \varepsilon \|\check f\|_{\max}\text{ and }\check g_i( \z^\star)\ge -\varepsilon \|\check g_i\|_{\max}\,,\, i\in [m]\,,
    \end{equation}
where   $\|q\|_{\max}:=\max_\a |q_\a|$ for any $q\in\R[\x]$.
\State If the inequalities \eqref{eq:check.sol.sparse} hold, set $\x^\star:=\z^{\star 2}$.
\end{algorithmic}
\end{algorithm}

\bibliographystyle{abbrv}

    
\end{document}